\documentclass{cpamart1}     
\received{Month 200X}       
\volume{000}
\startingpage{1}

\authorheadline{K. L. Ho and L. Ying}
\titleheadline{HIF-DE}

\usepackage[breaklinks]{hyperref}
\usepackage{algorithm,algpseudocode,amssymb,booktabs,graphicx,fixltx2e,multicol,multirow,subcaption}

\usepackage{mathptmx}

\newtheorem{theorem}{Theorem}[section]
\newtheorem{lemma}[theorem]{Lemma}
\newtheorem{corollary}[theorem]{Corollary}

\theoremstyle{definition}

\theoremstyle{remark}
\newtheorem{remark}[theorem]{Remark}

\numberwithin{algorithm}{section}

\DeclareMathOperator{\diag}{diag}

\newcommand{\alg}[1]{{\sf #1}}
\newcommand{\cmp}{\mathsf{C}}
\newcommand{\defn}[1]{{\em #1}}
\newcommand{\ext}{\mathsf{E}}

\newcommand{\nbr}{\mathsf{N}}
\newcommand{\rd}[1]{\check{#1}}
\newcommand{\sk}[1]{\hat{#1}}
\newcommand{\skel}{\mathcal{Z}}
\newcommand{\sumprime}{\sideset{}{'}\sum}
\newcommand{\tabdash}{\multicolumn{1}{|c}{---}}
\newcommand{\trans}{\mathsf{T}}

\begin{document}                        


\title{Hierarchical Interpolative Factorization for Elliptic Operators: Differential Equations}

\author{Kenneth L. Ho}{Stanford University}
\author{Lexing Ying}{Stanford University}





\begin{abstract}
 This paper introduces the hierarchical interpolative factorization for elliptic partial differential equations (HIF-DE) in two (2D) and three dimensions (3D). This factorization takes the form of an approximate generalized LU/LDL decomposition that facilitates the efficient inversion of the discretized operator. HIF-DE is based on the nested dissection multifrontal method but uses skeletonization on the separator fronts to sparsify the dense frontal matrices and thus reduce the cost. We conjecture that this strategy yields linear complexity in 2D and quasilinear complexity in 3D. Estimated linear complexity in 3D can be achieved by skeletonizing the compressed fronts themselves, which amounts geometrically to a recursive dimensional reduction scheme. Numerical experiments support our claims and further demonstrate the performance of our algorithm as a fast direct solver and preconditioner. MATLAB codes are freely available.
\end{abstract}

\maketitle






\section{Introduction}
This paper considers elliptic partial differential equations (PDEs) of the form
\begin{align}
 -\nabla \cdot (a(x) \nabla u(x)) + b(x) u(x) = f(x), \quad x \in \Omega \subset \mathbb{R}^{d}
 \label{eqn:pde}
\end{align}
with appropriate boundary conditions on $\partial \Omega$, where $a(x)$, $b(x)$, and $f(x)$ are given functions, and $d = 2$ or $3$. Such equations are of fundamental importance in science and engineering and encompass (perhaps with minor modification) many of the PDEs of classical physics, including the Laplace, Helmholtz, Stokes, and time-harmonic Maxwell equations. We will further assume that \eqref{eqn:pde} is not highly indefinite. Discretization using local schemes such as finite differences or finite elements then leads to a linear system
\begin{align}
 Au = f,
 \label{eqn:linear-system}
\end{align}
where $A \in \mathbb{R}^{N \times N}$ is sparse with $u$ and $f$ the discrete analogues of $u(x)$ and $f(x)$, respectively. This paper is concerned with the efficient factorization and solution of such systems.

\subsection{Previous Work}
A large part of modern numerical analysis and scientific computing has been devoted to the solution of \eqref{eqn:linear-system}. We classify existing approaches into several groups. The first consists of classical direct methods like Gaussian elimination or other standard matrix factorizations \cite{golub:1996:johns-hopkins-univ}, which compute the solution exactly (in principle, to machine precision, up to conditioning) without iteration. Naive implementations generally have $O(N^{3})$ complexity but can be heavily accelerated by exploiting sparsity \cite{davis:2006:siam}. A key example is the nested dissection multifrontal method (MF) \cite{duff:1983:acm-trans-math-software,george:1973:siam-j-numer-anal,liu:1992:siam-rev}, which performs elimination according to a special hierarchy of separator fronts in order to minimize fill-in. These fronts correspond geometrically to the cell interfaces in a domain partitioning and grow as $O(N^{1/2})$ in two dimensions (2D) and $O(N^{2/3})$ in three dimensions (3D), resulting in solver complexities of $O(N^{3/2})$ and $O(N^{2})$, respectively. This is a significant improvement and, indeed, MF has proven very effective in many environments. However, it remains unsuitable for truly large-scale problems, especially in 3D.

The second group is that of iterative methods \cite{saad:2003:siam}, with conjugate gradient (CG) \cite{hestenes:1952:j-res-nat-bur-stand,van-der-vorst:1992:siam-j-sci-stat-comput} and multigrid \cite{brandt:1977:math-comp,hackbusch:1985:springer,xu:1992:siam-rev} among the most popular techniques. These typically work well when $a(x)$ and $b(x)$ are smooth, in which case the number of iterations required is small and optimal $O(N)$ complexity can be achieved. However, the iteration count can grow rapidly in the presence of ill-conditioning, which can arise when the coefficient functions lack regularity or have high contrast. In such cases, convergence can be delicate and specialized preconditioners are often required. Furthermore, iterative methods can be inefficient for systems involving multiple right-hand sides or low-rank updates, which is an important setting for many applications of increasing interest, including time stepping, inverse problems, and design.

The third group covers rank-structured direct solvers, which exploit the observation that certain off-diagonal blocks of $A$ and $A^{-1}$ are numerically low-rank \cite{bebendorf:2005:math-comp,bebendorf:2003:numer-math,borm:2010:numer-math,chandrasekaran:2010:siam-j-matrix-anal-appl} in order to dramatically lower the cost. The seminal work in this area is due to Hackbusch et al.\ \cite{hackbusch:1999:computing,hackbusch:2002:computing,hackbusch:2000:computing}, whose $\mathcal{H}$- and $\mathcal{H}^{2}$-matrices have been shown to achieve linear or quasilinear complexity. These methods were originally introduced for integral equations characterized by structured dense matrices but apply also to PDEs as a special case. Although their work has had significant theoretical impact, in practice, the constants implicit in the asymptotic scalings tend to be quite large due to the recursive nature of the inversion algorithms, the use of expensive hierarchical matrix-matrix multiplication, and the lack of sparsity optimizations.

More recent developments aimed at improving practical performance have combined MF with structured matrix algebra on the dense frontal matrices only. This better exploits the inherent sparsity of $A$ and has been carried out under both the $\mathcal{H}$- \cite{grasedyck:2009:numer-math,schmitz:2012:j-comput-phys,schmitz:2014:j-comput-phys} and hierarchically semiseparable (HSS) \cite{gillman:2014:siam-j-sci-comput,gillman:2014:adv-comput-math,xia:2013:siam-j-sci-comput,xia:2013:siam-j-matrix-anal-appl,xia:2009:siam-j-matrix-anal-appl} matrix frameworks, among other related schemes \cite{amestoy:siam-j-sci-comput,aminfar:arxiv,martinsson:2009:j-sci-comput}. Those under the former retain their quasilinear complexities and have improved constants but can still be somewhat expensive. On the other hand, those using HSS operations, which usually have much more favorable constants, are optimal in 2D but require $O(N^{4/3})$ work in 3D. In principle, it is possible to further reduce this to $O(N)$ work by using multi-layer HSS representations, but this procedure is quite complicated and has yet to be achieved.

\subsection{Contributions}
In this paper, we introduce the hierarchical interpolative factorization for PDEs (HIF-DE), which produces an approximate generalized LU/LDL decomposition of $A$ with linear or quasilinear complexity estimates. HIF-DE is based on MF but augments it with frontal compression using a matrix sparsification technique that we call skeletonization. The resulting algorithm is similar in structure to the accelerated MF solvers above and is sufficient for estimated scalings of $O(N)$ in 2D and $O(N \log N)$ in 3D. Unlike \cite{gillman:2014:siam-j-sci-comput,gillman:2014:adv-comput-math,grasedyck:2009:numer-math,schmitz:2012:j-comput-phys,schmitz:2014:j-comput-phys,xia:2009:siam-j-matrix-anal-appl}, however, which keep the entire fronts but work with them implicitly using fast structured methods, our sparsification approach allows us to reduce the fronts explicitly (see also \cite{xia:2013:siam-j-sci-comput,xia:2013:siam-j-matrix-anal-appl}). This obviates the need for internal hierarchical matrix representations and substantially simplifies the algorithm. Importantly, it also makes any additional compression straightforward to accommodate, thereby providing a ready means to achieve estimated $O(N)$ complexity in 3D by skeletonizing the compressed fronts themselves. This corresponds geometrically to a recursive dimensional reduction, whose interpretation is directly enabled by the skeletonization formalism.

Figure \ref{fig:schematic} shows a schematic of HIF-DE as compared to MF in 2D.
\begin{figure}
 \includegraphics{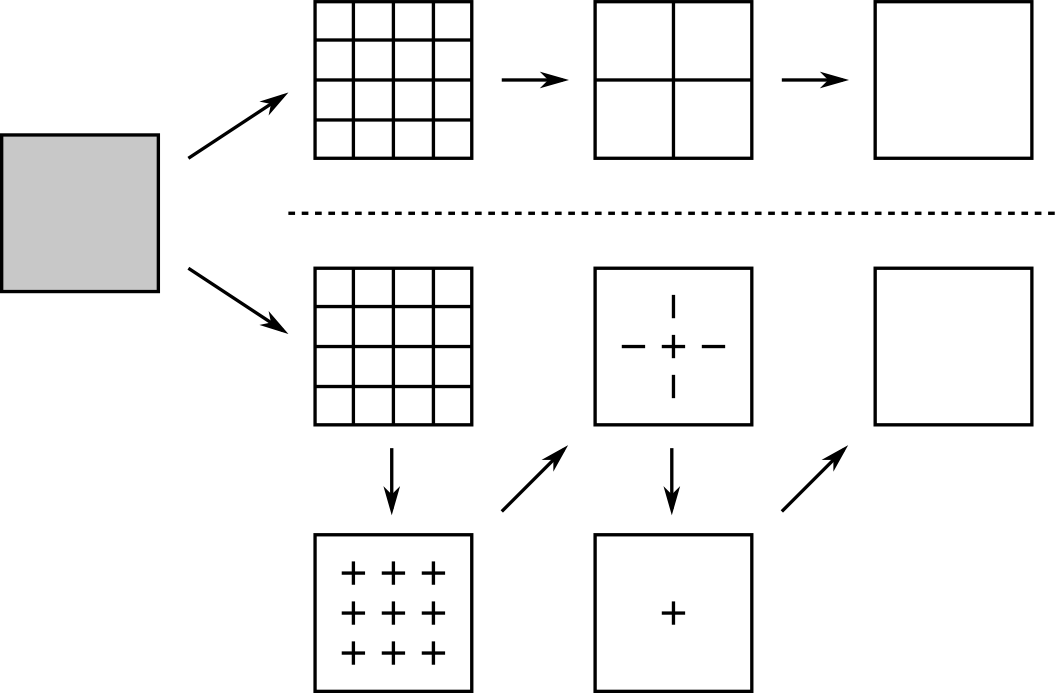}
 \caption{Schematic of MF (top) and HIF-DE (bottom) in 2D. The gray box (left) represents a uniformly discretized square; the lines in the interior of the boxes (right) denote the remaining DOFs after each level of elimination or skeletonization.}
 \label{fig:schematic}
\end{figure}
In MF (top), the domain is partitioned by a set of separators into ``interior'' square cells at each level of a tree hierarchy. Each cell is eliminated starting from the finest level to the coarsest, leaving degrees of freedom (DOFs) only on the separators, which constitute the so-called fronts. This process can be understood as the compression of data from the cells to their interfaces, which evidently grow as we march up the tree, ultimately leading to the observed $O(N^{3/2})$ complexity.

In contrast, in HIF-DE (bottom), we start by eliminating interior cells as in MF but, before proceeding further, perform an additional level of compression by skeletonizing the separators. For this, we view the separator DOFs as living on the interfacial edges of the interior cells then skeletonize each cell edge. This respects the one-dimensional (1D) structure of the separator geometry and allows more DOFs to be eliminated, in effect reducing each edge to only those DOFs near its boundary. Significantly, this occurs {\em without any loss of existing sparsity}. The combination of interior cell elimination and edge skeletonization is then repeated up the tree, with the result that the frontal growth is now suppressed. The reduction from 2D (square cells) to 1D (edges) to zero dimensions (0D) (points) is completely explicit. Extension to 3D is immediate by eliminating interior cubic cells then skeletonizing cubic faces at each level to execute a reduction from 3D to 2D to 1D at a total estimated cost of $O(N \log N)$. We can further reduce this to $O(N)$ (but at the price of introducing some fill-in) by adding subsequent cubic edge skeletonization at each level for full reduction to 0D. This tight control of the front size is critical for achieving near-optimal scaling.

Once the factorization has been constructed, it can be used to rapidly apply $A^{-1}$ and therefore serves as a fast direct solver or preconditioner, depending on the accuracy. (It can also be used to apply $A$ itself, but this is not particularly advantageous since $A$ typically has only $O(N)$ nonzeros.) Other capabilities are possible, too, though they will not be pursued here.

HIF-DE can also be understood in relation to the somewhat more general \defn{hierarchical interpolative factorization for integral equations} (HIF-IE) described in the companion paper \cite{ho:comm-pure-appl-math}, which, like other structured dense methods, can apply to PDEs as a special case. However, HIF-IE does not make use of sparsity and so is not very competitive in practice. HIF-DE remedies this by essentially embedding HIF-IE into the framework of MF in order to maximally exploit sparsity.

Extensive numerical experiments reveal strong evidence for quasilinear complexity and demonstrate that HIF-DE can accurately approximate elliptic partial differential operators in a variety of settings with high practical efficiency.

\subsection{Outline}
The remainder of this paper is organized as follows. In Section \ref{sec:prelim}, we introduce the basic tools needed for our algorithm, including our new skeletonization operation. In Section \ref{sec:mf}, we review MF, which will serve to establish the necessary algorithmic foundation as well as to highlight its fundamental difficulties. In Section \ref{sec:hifde}, we present HIF-DE as an extension of MF with frontal skeletonization corresponding to recursive dimensional reduction. Although we cannot yet provide a rigorous complexity analysis, estimates based on well-supported rank assumptions suggest that HIF-DE achieves linear or quasilinear complexity. This conjecture is borne out by numerical experiments, which we detail in Section \ref{sec:results}. Finally, Section \ref{sec:conclusion} concludes with some discussion and future directions.

\section{Preliminaries}
\label{sec:prelim}
In this section, we first list our notational conventions and then describe the basic elements of our algorithm.

Uppercase letters will generally denote matrices, while the lowercase letters $c$, $p$, $q$, $r$, and $s$ denote ordered sets of indices, each of which is associated with a DOF in the problem. For a given index set $c$, its cardinality is written $|c|$. The (unordered) complement of $c$ is given by $c^{\cmp}$, with the parent set to be understood from the context. The uppercase letter $C$ is reserved to denote a collection of disjoint index sets.

Given a matrix $A$, $A_{pq}$ is the submatrix with rows and columns restricted to the index sets $p$ and $q$, respectively. We also use the MATLAB notation $A_{:,q}$ to denote the submatrix with columns restricted to $q$. The \defn{neighbor set} of an index set $c$ with respect to $A$ is then $c^{\nbr} = \{ i \notin c : \text{$A_{i,c}$ or $A_{c,i} \neq 0$} \}$.

Throughout, $\| \cdot \|$ refers to the $2$-norm.

For simplicity, we hereafter assume that the matrix $A$ in \eqref{eqn:linear-system} is symmetric, though this is not strictly necessary \cite{ho:comm-pure-appl-math}.

\subsection{Sparse Elimination}
\label{sec:sparse-elim}
Let
\begin{align}
 A =
 \begin{bmatrix}
  A_{pp} & A_{qp}^{\trans}\\
  A_{qp} & A_{qq} & A_{rq}^{\trans}\\
  & A_{rq} & A_{rr}
 \end{bmatrix}
 \label{eqn:sparse-matrix}
\end{align}
be a symmetric matrix defined over the indices $(p, q, r)$. This matrix structure often appears in sparse PDE problems such as \eqref{eqn:linear-system}, where, for example, $p$ corresponds to the interior DOFs of a region $\mathcal{D}$, $q$ to the DOFs on the boundary $\partial \mathcal{D}$, and $r$ to the external region $\Omega \setminus \bar{\mathcal{D}}$, which should be thought of as large. In this setting, the DOFs $p$ and $r$ are separated by $q$ and hence do not directly interact, resulting in the form \eqref{eqn:sparse-matrix}.

Our first tool is quite standard and concerns the efficient elimination of DOFs from such sparse matrices.

\begin{lemma}
 \label{lem:sparse-elim}
 Let $A$ be given by \eqref{eqn:sparse-matrix} and write $A_{pp} = L_{p} D_{p} L_{p}^{\trans}$ in factored form, where $L_{p}$ is a unit triangular matrix (up to permutation). If $A_{pp}$ is nonsingular, then
 \begin{align}
  S_{p}^{\trans} A S_{p} =
  \begin{bmatrix}
   D_{p}\\
   & B_{qq} & A_{rq}^{\trans}\\
   & A_{rq} & A_{rr}
  \end{bmatrix},
  \label{eqn:sparse-elim}
 \end{align}
 where
 \begin{align*}
  S_{p} =
  \begin{bmatrix}
   L_{p}^{-\trans}\\
   & I\\
   & & I
  \end{bmatrix}
  \begin{bmatrix}
   I & -D_{p}^{-1} L_{p}^{-1} A_{qp}^{\trans}\\
   & I\\
   & & I
  \end{bmatrix}
 \end{align*}
 and $B_{qq} = A_{qq} - A_{qp} A_{pp}^{-1} A_{qp}^{\trans}$ is the associated Schur complement.
\end{lemma}

Note that the indices $p$ have been decoupled from the rest. Regarding the subsystem in \eqref{eqn:sparse-elim} over the indices $(q, r)$ only, we may therefore say that the DOFs $p$ have been eliminated. The operator $S_{p}$ carries out this elimination, which furthermore is particularly efficient since the interactions involving the large index set $r$ are unchanged. However, some fill-in is generated through the Schur complement $B_{qq}$, which in general is completely dense. Clearly, the requirement that $A_{pp}$ be invertible is satisfied if $A$ is symmetric positive definite (SPD), as is the case for many such problems in practice.

In this paper, we often work with a collection $C$ of disjoint index sets, where $A_{c,c'} = A_{c',c} = 0$ for any $c, c' \in C$ with $c \neq c'$. Applying Lemma \ref{lem:sparse-elim} to each $p = c$, $q = c^{\nbr}$, and $r = (c \cup c^{\nbr})^{\cmp}$ gives $W^{\trans} AW$ for $W = \prod_{c \in C} S_{c}$, where each set of DOFs $c \in C$ has been decoupled from the rest and the matrix product over $C$ can be taken in any order. The resulting matrix has a block diagonal structure over the index groups
\begin{align*}
 \theta = \left( \bigcup_{c \in C} \{ c \} \right) \cup \left\{ s \setminus \bigcup_{c \in C} c \right\},
\end{align*}
where the outer union is to be understood as acting on collections of index sets and $s = \{ 1, \dots, N \}$ is the set of all indices, but with dense fill-in covering $(W^{\trans} AW)_{c^{\nbr},c^{\nbr}}$ for each $c \in C$.

\subsection{Interpolative Decomposition}
Our next tool is the interpolative decomposition (ID) \cite{cheng:2005:siam-j-sci-comput} for low-rank matrices, which we present in a somewhat nonstandard form below (see \cite{ho:comm-pure-appl-math} for details).

\begin{lemma}
 \label{lem:id}
 Let $A = A_{:,q} \in \mathbb{R}^{m \times n}$ with rank $k \leq \min (m, n)$. Then there exist a partitioning $q = \sk{q} \cup \rd{q}$ with $|\sk{q}| = k$ and a matrix $T_{q} \in \mathbb{R}^{k \times n}$ such that $A_{:,\rd{q}} = A_{:,\sk{q}} T_{q}$.
\end{lemma}

We call $\sk{q}$ and $\rd{q}$ the \defn{skeleton} and \defn{redundant} indices, respectively. Lemma \ref{lem:id} states that the redundant columns of $A$ can be interpolated from its skeleton columns. The following shows that the ID can also be viewed as a sparsification operator.

\begin{corollary}
 \label{cor:id-sparse}
 Let $A = A_{:,q}$ be a low-rank matrix. If $q = \sk{q} \cup \rd{q}$ and $T_{q}$ are such that $A_{:,\rd{q}} = A_{:,\sk{q}} T_{q}$, then
 \begin{align*}
  \begin{bmatrix}
   A_{:,\rd{q}} & A_{:,\sk{q}}
  \end{bmatrix}
  \begin{bmatrix}
   I\\
   -T_{q} & I
  \end{bmatrix} =
  \begin{bmatrix}
   0 & A_{:,\sk{q}}
  \end{bmatrix}.
 \end{align*}
\end{corollary}

In general, let $A_{:,\rd{q}} = A_{:,\sk{q}} T_{q} + E$ for some error matrix $E$. If $\| T_{q} \|$ and $\| E \|$ are not too large, then the reconstruction of $A_{:,\rd{q}}$ is stable and accurate. In this paper, we use the algorithm of \cite{cheng:2005:siam-j-sci-comput} based on a simple pivoted QR decomposition to compute an ID that typically satisfies
\begin{align*}
 \| T_{q} \| \leq \sqrt{4k(n - k)}, \quad \| E \| \leq \sqrt{1 + 4k(n - k)} \sigma_{k + 1} (A),
\end{align*}
where $\sigma_{k + 1} (A)$ is the $(k + 1)$st largest singular value of $A$, at a cost of $O(kmn)$ operations. Fast algorithms based on random sampling are also available \cite{halko:2011:siam-rev}, but these can incur some loss of accuracy (see also Section \ref{sec:hifde:3dx}).

The ID can be applied in both fixed and adaptive rank settings. In the former, the rank $k$ is specified, while, in the latter, the approximation error is specified and the rank adjusted to achieve (an estimate of) it. Hereafter, we consider the ID only in the adaptive sense, using the relative magnitudes of the pivots to adaptively select $k$ such that $\| E \| \lesssim \epsilon \| A \|$ for any specified relative precision $\epsilon > 0$.

\subsection{Skeletonization}
We now combine Lemmas \ref{lem:sparse-elim} and \ref{lem:id} to efficiently eliminate redundant DOFs from dense matrices with low-rank off-diagonal blocks.

\begin{lemma}
 \label{lem:skel}
 Let
 \begin{align*}
  A =
  \begin{bmatrix}
   A_{pp} & A_{qp}^{\trans}\\
   A_{qp} & A_{qq}
  \end{bmatrix}
 \end{align*}
 be symmetric with $A_{qp}$ low-rank, and let $p = \sk{p} \cup \rd{p}$ and $T_{p}$ be such that $A_{q \rd{p}} = A_{q \sk{p}} T_{p}$. Without loss of generality, write
 \begin{align*}
  A =
  \begin{bmatrix}
   A_{\rd{p} \rd{p}} & A_{\sk{p} \rd{p}}^{\trans} & A_{q \rd{p}}^{\trans}\\
   A_{\sk{p} \rd{p}} & A_{\sk{p} \sk{p}} & A_{q \sk{p}}^{\trans}\\
   A_{q \rd{p}} & A_{q \sk{p}} & A_{qq}
  \end{bmatrix}
 \end{align*}
 and define
 \begin{align*}
  Q_{p} =
  \begin{bmatrix}
   I\\
   -T_{p} & I\\
   & & I
  \end{bmatrix}.
 \end{align*}
 Then
 \begin{align}
  Q_{p}^{\trans} A Q_{p} =
  \begin{bmatrix}
   B_{\rd{p} \rd{p}} & B_{\sk{p} \rd{p}}^{\trans}\\
   B_{\sk{p} \rd{p}} & A_{\sk{p} \sk{p}} & A_{q \sk{p}}^{\trans}\\
   & A_{q \sk{p}} & A_{qq}
  \end{bmatrix},
  \label{eqn:id-sparse}
 \end{align}
 where
 \begin{align*}
  B_{\rd{p} \rd{p}} &= A_{\rd{p} \rd{p}} - T_{p}^{\trans} A_{\sk{p} \rd{p}} - A_{\sk{p} \rd{p}}^{\trans} T_{p} + T_{p}^{\trans} A_{\sk{p} \sk{p}} T_{p},\\
  B_{\sk{p} \rd{p}} &= A_{\sk{p} \rd{p}} - A_{\sk{p} \sk{p}} T_{p},
 \end{align*}
 so
 \begin{align}
  \label{eqn:skel}
  S_{\rd{p}}^{\trans} Q_{p}^{\trans} A Q_{p} S_{\rd{p}} =
  \begin{bmatrix}
   D_{\rd{p}}\\
   & B_{\sk{p} \sk{p}} & A_{q \sk{p}}^{\trans}\\
   & A_{q \sk{p}} & A_{qq}
  \end{bmatrix} \equiv \skel_{p} (A),
 \end{align}
 where $S_{\rd{p}}$ is the elimination operator of Lemma \ref{lem:sparse-elim} associated with $\rd{p}$ and $B_{\sk{p} \sk{p}} = A_{\sk{p} \sk{p}} - B_{\sk{p} \rd{p}} B_{\rd{p} \rd{p}}^{-1} B_{\sk{p} \rd{p}}^{\trans}$, assuming that $B_{\rd{p}\rd{p}}$ is nonsingular.
\end{lemma}

In essence, the ID sparsifies $A$ by decoupling $\rd{p}$ from $q$, thereby allowing it to be eliminated using efficient sparse techniques. We refer to this procedure as \defn{skeletonization} since only the skeletons $\sk{p}$ remain. Note that the interactions involving $q = p^{\cmp}$ are unchanged. A very similar approach has previously been described in the context of HSS Cholesky decompositions \cite{xia:2010:numer-linear-algebra-appl} by combining the structure-preserving rank-revealing factorization \cite{xia:2012:siam-j-matrix-anal-appl} with reduced matrices \cite{xia:2013:siam-j-sci-comput}.

In general, the ID often only approximately sparsifies $A$ (for example, if its off-diagonal blocks are low-rank only to a specified numerical precision) so that \eqref{eqn:id-sparse} and consequently \eqref{eqn:skel} need not hold exactly. In such cases, the skeletonization operator $\skel_{p} (\cdot)$ should be interpreted as also including an intermediate truncation step that enforces sparsity explicitly. For notational convenience, however, we will continue to identify the left- and right-hand sides of \eqref{eqn:skel} by writing $\skel_{p} (A) \approx S_{\rd{p}}^{\trans} Q_{p}^{\trans} A Q_{p} S_{\rd{p}}$, with the truncation to be understood implicitly.

In this paper, we often work with a collection $C$ of disjoint index sets, where $A_{c,c^{\cmp}}$ and $A_{c^{\cmp},c}$ are numerically low-rank for all $c \in C$. Applying Lemma \ref{lem:skel} to all $c \in C$ gives
\begin{align*}
 \skel_{C} (A) \approx U^{\trans} AU, \quad U = \prod_{c \in C} Q_{c} S_{\rd{c}},
\end{align*}
where the redundant DOFs $\rd{c}$ for each $c \in C$ have been decoupled from the rest and the matrix product over $C$ can be taken in any order. The resulting skeletonized matrix $\skel_{C} (A)$ is significantly sparsified and has a block diagonal structure over the index groups
\begin{align*}
 \theta = \left( \bigcup_{c \in C} \{ \rd{c} \} \right) \cup \left\{ s \setminus \bigcup_{c \in C} \rd{c} \right\}.
\end{align*}

\section{Multifrontal Factorization}
\label{sec:mf}
In this section, we review MF, which constructs a multilevel LDL decomposition of $A$ by using Lemma \ref{lem:sparse-elim} to eliminate DOFs according to a hierarchical sequence of domain separators. Our presentation will tend to emphasize its geometric aspects \cite{george:1973:siam-j-numer-anal}; more algebraic treatments can be found in \cite{duff:1983:acm-trans-math-software,liu:1992:siam-rev}.

We begin with a detailed description of MF in 2D before extending to 3D in the natural way. The same presentation framework will also be used for HIF-DE in Section \ref{sec:hifde}, which we hope will help make clear the specific innovations responsible for its improved complexity estimates.

\subsection{Two Dimensions}
\label{sec:mf:2d}
Consider the PDE \eqref{eqn:pde} on $\Omega = (0, 1)^{2}$ with zero Dirichlet boundary conditions, discretized using finite differences via the five-point stencil over a uniform $n \times n$ grid for simplicity. More general domains, boundary conditions, and discretizations can be handled without difficulty, but the current setting will serve to fix ideas. Let $h$ be the step size in each direction and assume that $n = 1/h = 2^{L} m$, where $m = O(1)$ is a small integer. Integer pairs $j = (j_{1}, j_{2})$ index the grid points $x_{j} = hj = h(j_{1}, j_{2})$ for $1 \leq j_{1}, j_{2} \leq n - 1$. The discrete system \eqref{eqn:linear-system} then reads
\begin{multline*}
 \frac{1}{h^{2}} \left( a_{j - e_{1} / 2} + a_{j + e_{1} / 2} + a_{j - e_{2} / 2} + a_{j + e_{2} / 2} \right) u_{j}\\
 - \frac{1}{h^{2}} \left( a_{j - e_{1} / 2} u_{j - e_{1}} + a_{j + e_{1} / 2} u_{j + e_{1}} + a_{j - e_{2} / 2} u_{j - e_{2}} + a_{j + e_{2} / 2} u_{j + e_{2}} \right) + b_{j} u_{j} = f_{j}
\end{multline*}
at each $x_{j}$, where $a_{j} = a(hj)$ is sampled on the ``staggered'' dual grid for $e_{1} = (1, 0)$ and $e_{2} = (0, 1)$ the unit coordinate vectors, $b_{j} = b(x_{j})$, $f_{j} = f(x_{j})$, and $u_{j}$ is the approximation to $u(x_{j})$. The resulting matrix $A$ is sparse and symmetric, consisting only of nearest-neighbor interactions. The total number of DOFs is $N = (n - 1)^{2}$, each of which is associated with a point $x_{j}$ and an index in $s$.

The algorithm proceeds by eliminating DOFs level by level. At each level $\ell$, the set of DOFs that have not been eliminated are called \defn{active} with indices $s_{\ell}$. Initially, we set $A_{0} = A$ and $s_{0} = s$. Figure \ref{fig:mf2} shows the active DOFs at each level for a representative example.
\begin{figure}
 \centering
 \begin{subfigure}{0.2\textwidth}
  \includegraphics[width=\textwidth]{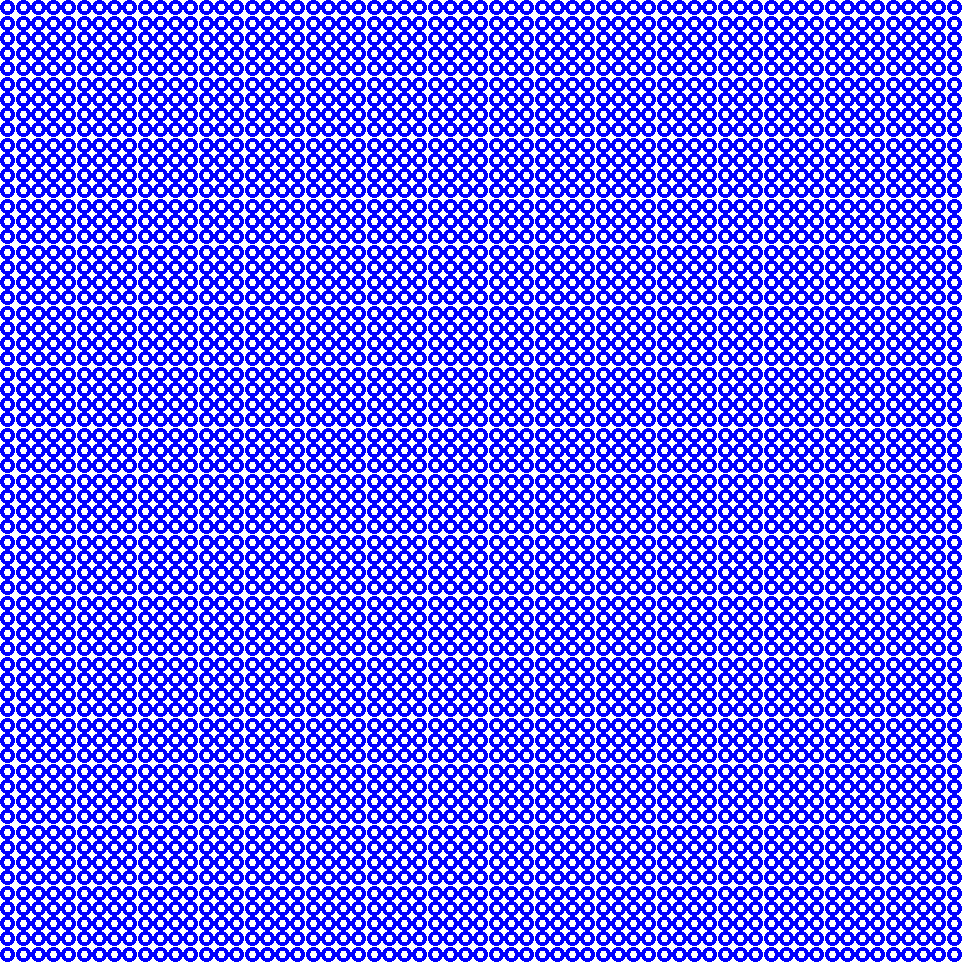}
  \caption*{$\ell = 0$}
 \end{subfigure}
 \quad
 \begin{subfigure}{0.2\textwidth}
  \includegraphics[width=\textwidth]{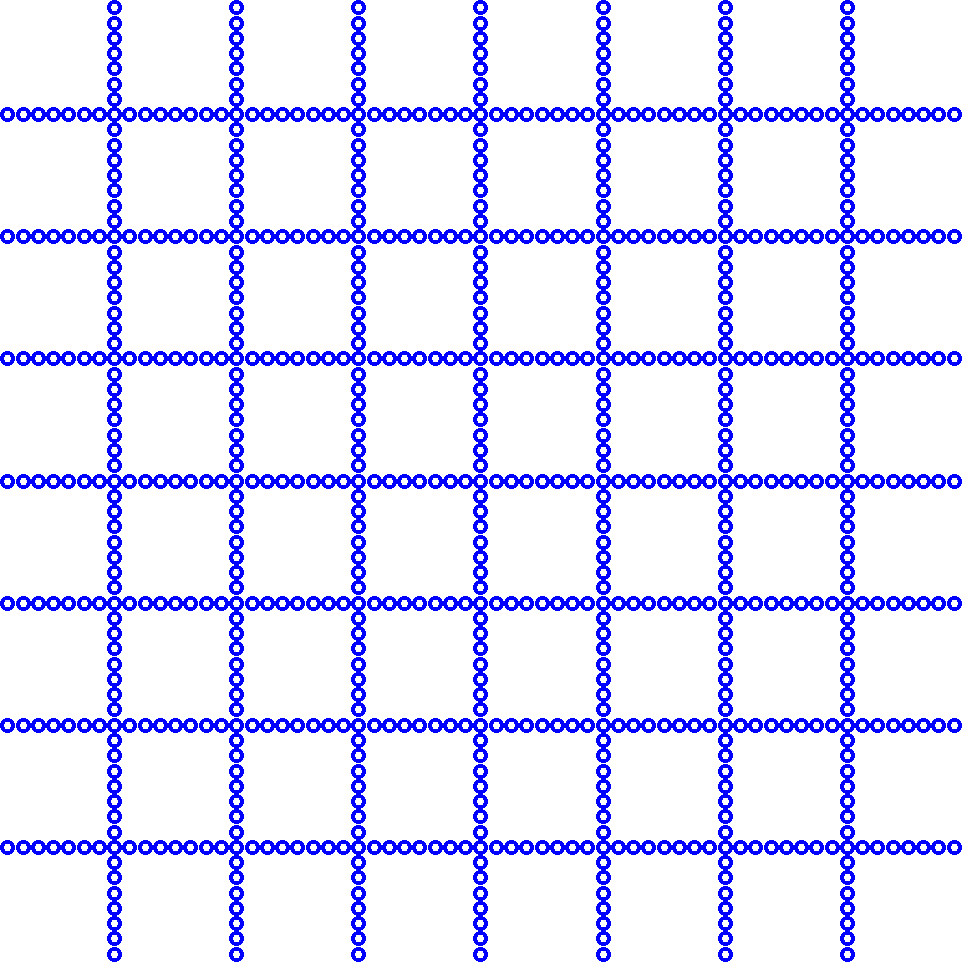}
  \caption*{$\ell = 1$}
 \end{subfigure}
 \quad
 \begin{subfigure}{0.2\textwidth}
  \includegraphics[width=\textwidth]{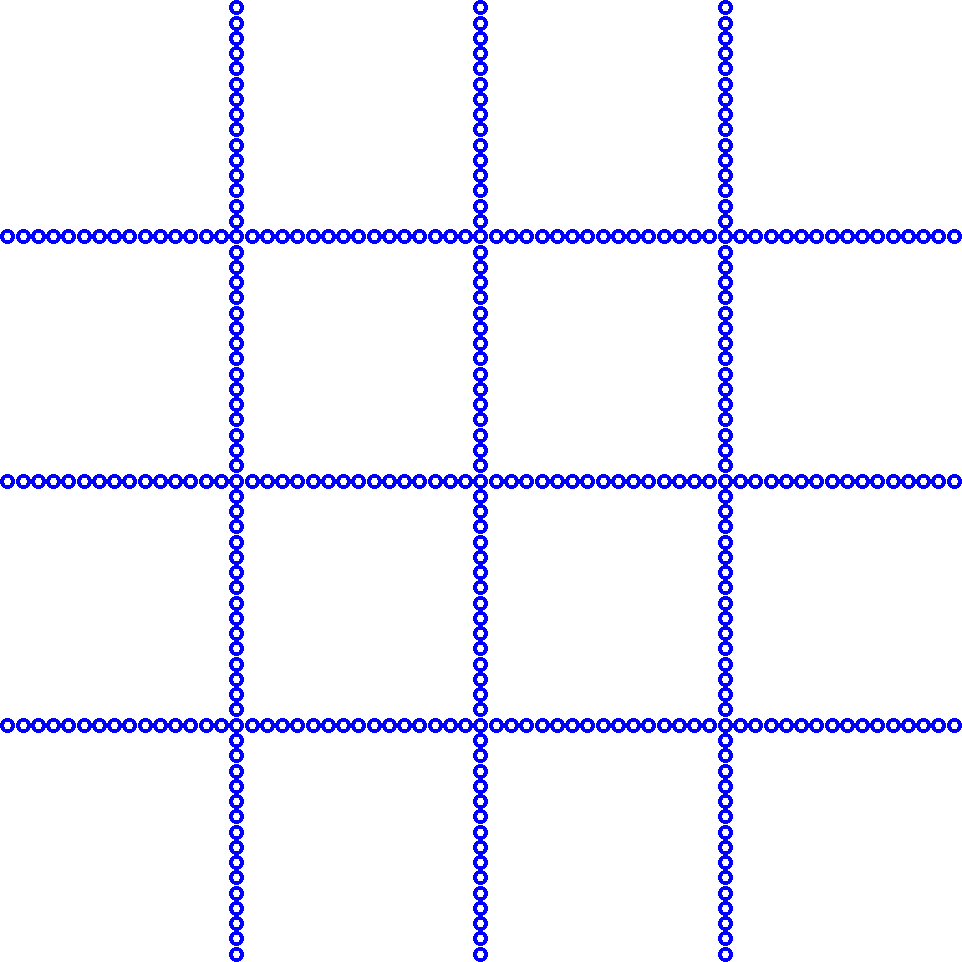}
  \caption*{$\ell = 2$}
 \end{subfigure}
 \quad
 \begin{subfigure}{0.2\textwidth}
  \includegraphics[width=\textwidth]{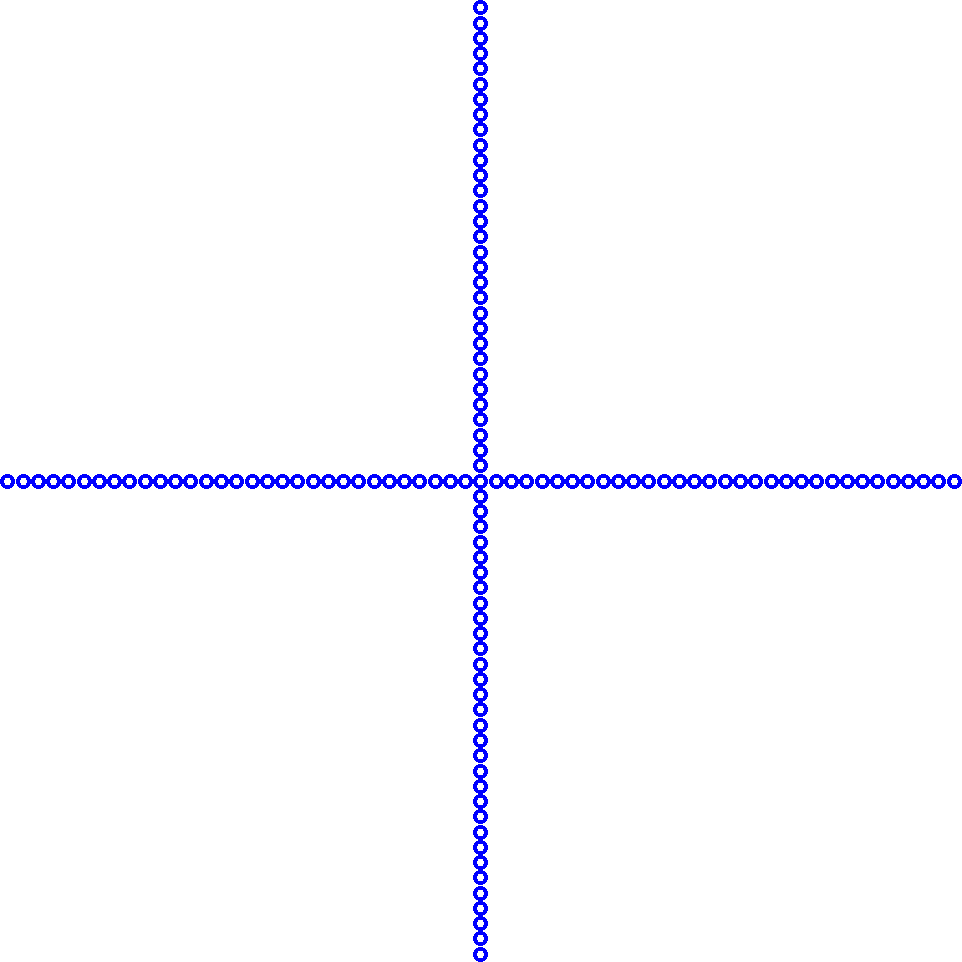}
  \caption*{$\ell = 3$}
 \end{subfigure}
 \caption{Active DOFs at each level $\ell$ of MF in 2D.}
 \label{fig:mf2}
\end{figure}

\subsubsection*{Level $0$}
Defined at this stage are $A_{0}$ and $s_{0}$. Partition $\Omega$ by 1D separators $mh(j_{1}, \cdot)$ and $mh(\cdot, j_{2})$ for $1 \leq j_{1}, j_{2} \leq 2^{L} - 1$ every $mh = n / 2^{L}$ units in each direction into interior square cells $mh(j_{1} - 1, j_{1}) \times mh(j_{2} - 1, j_{2})$ for $1 \leq j_{1}, j_{2} \leq 2^{L}$. Observe that distinct cells do not interact with each other since they are buffered by the separators. Let $C_{0}$ be the collection of index sets corresponding to the active DOFs of each cell. Then elimination with respect to $C_{0}$ gives
\begin{align*}
 A_{1} = W_{0}^{\trans} A_{0} W_{0}, \quad W_{0} = \prod_{c \in C_{0}} S_{c},
\end{align*}
where the DOFs $\bigcup_{c \in C_{0}} c$ have been eliminated (and marked inactive). Let $s_{1} = s_{0} \setminus \bigcup_{c \in C_{0}} c$ be the remaining active DOFs. The matrix $A_{1}$ is block diagonal with block partitioning
\begin{align*}
 \theta_{1} = \left( \bigcup_{c \in C_{0}} \{ c \} \right) \cup \{ s_{1} \}.
\end{align*}

\subsubsection*{Level $\ell$}
Defined at this stage are $A_{\ell}$ and $s_{\ell}$. Partition $\Omega$ by 1D separators $2^{\ell} mh(j_{1}, \cdot)$ and $2^{\ell} mh(\cdot, j_{2})$ for $1 \leq j_{1}, j_{2} \leq 2^{L - \ell} - 1$ every $2^{\ell} mh = n / 2^{L - \ell}$ units in each direction into interior square cells $2^{\ell} mh(j_{1} - 1, j_{1}) \times 2^{\ell} mh(j_{2} - 1, j_{2})$ for $1 \leq j_{1}, j_{2} \leq 2^{L - \ell}$. Let $C_{\ell}$ be the collection of index sets corresponding to the active DOFs of each cell. Elimination with respect to $C_{\ell}$ then gives
\begin{align*}
 A_{\ell + 1} = W_{\ell}^{\trans} A_{\ell} W_{\ell}, \quad W_{\ell} = \prod_{c \in C_{\ell}} S_{c},
\end{align*}
where the DOFs $\bigcup_{c \in C_{\ell}} c$ have been eliminated. The matrix $A_{\ell + 1}$ is block diagonal with block partitioning
\begin{align*}
 \theta_{\ell + 1} = \left( \bigcup_{c \in C_{0}} \{ c \} \right) \cup \cdots \cup \left( \bigcup_{c \in C_{\ell}} \{ c \} \right) \cup \{ s_{\ell + 1} \},
\end{align*}
where $s_{\ell + 1} = s_{\ell} \setminus \bigcup_{c \in C_{\ell}} c$.

\subsubsection*{Level $L$}
Finally, we have $A_{L}$ and $s_{L}$, where $D \equiv A_{L}$ is block diagonal with block partitioning
\begin{align*}
 \theta_{L} = \left( \bigcup_{c \in C_{0}} \{ c \} \right) \cup \cdots \cup \left( \bigcup_{c \in C_{L - 1}} \{ c \} \right) \cup \{ s_{L} \}.
\end{align*}
Combining over all levels gives
\begin{align*}
 D = W_{L - 1}^{\trans} \cdots W_{0}^{\trans} A W_{0} \cdots W_{L - 1},
\end{align*}
where each $W_{\ell}$ is a product of unit upper triangular matrices, each of which can be inverted simply by negating its off-diagonal entries. Therefore,
\begin{subequations}
 \label{eqn:mf}
 \begin{align}
  A &= W_{0}^{-\trans} \cdots W_{L - 1}^{-\trans} D W_{L - 1}^{-1} \cdots W_{0}^{-1} \equiv F,\\
  A^{-1} &= W_{0} \cdots W_{L - 1} D^{-1} W_{L - 1}^{\trans} \cdots W_{0}^{\trans} = F^{-1}.
 \end{align}
\end{subequations}
The factorization $F$ is an LDL decomposition of $A$ that is numerically exact (to machine precision, up to conditioning), whose inverse $F^{-1}$ can be applied as a fast direct solver. Clearly, if $A$ is SPD, then so are $F$ and $F^{-1}$; in this case, $F$ can, in fact, be written as a Cholesky decomposition by storing $D$ in Cholesky form. We emphasize that $F$ and $F^{-1}$ are not assembled explicitly and are used only in their factored representations.

The entire procedure is summarized compactly as Algorithm \ref{alg:mf}. In general, we can construct the cell partitioning at each level using an adaptive quadtree \cite{samet:1984:acm-comput-surv}, which recursively subdivides the domain until each node contains only $O(1)$ DOFs, provided that some appropriate postprocessing is done to define ``thin'' separators in order to optimally exploit sparsity (see Section \ref{sec:hifde:3dx}).
\begin{algorithm}
 \caption{MF.}
 \label{alg:mf}
 \begin{algorithmic}
  \State $A_{0} = A$ \Comment{initialize}
  \For{$\ell = 0, 1, \dots, L - 1$} \Comment{loop from finest to coarsest level}
   \State $A_{\ell + 1} = W_{\ell}^{\trans} A_{\ell} W_{\ell}$ \Comment{eliminate interior cells}
  \EndFor
  \State $A = W_{0}^{-\trans} \cdots W_{L - 1}^{-\trans} A_{L} W_{L - 1}^{-1} \cdots W_{0}^{-1}$ \Comment{LDL decomposition}
 \end{algorithmic}
\end{algorithm}

\subsection{Three Dimensions}
\label{sec:mf:3d}
Consider now the analogous setting in 3D, where $\Omega = (0, 1)^{3}$ is discretized using the seven-point stencil over a uniform $n \times n \times n$ mesh with grid points $x_{j} = hj = h(j_{1}, j_{2}, j_{3})$ for $j = (j_{1}, j_{2}, j_{3})$:
\begin{multline*}
 \frac{1}{h^{2}} \left( a_{j - e_{1} / 2} + a_{j + e_{1} / 2} + a_{j - e_{2} / 2} + a_{j + e_{2} / 2} + a_{j - e_{3} / 2} + a_{j + e_{3} / 2} \right) u_{j}\\
 - \frac{1}{h^{2}} \left( a_{j - e_{1} / 2} u_{j - e_{1}} + a_{j + e_{1} / 2} u_{j + e_{1}} + a_{j - e_{2} / 2} u_{j - e_{2}} + a_{j + e_{2} / 2} u_{j + e_{2}} \right.\\
 \left. + a_{j - e_{3} / 2} u_{j - e_{3}} + a_{j + e_{3} / 2} u_{j + e_{3}} \right) + b_{j} u_{j} = f_{j},
\end{multline*}
where $e_{1} = (1, 0, 0)$, $e_{2} = (0, 1, 0)$, and $e_{3} = (0, 0, 1)$. The total number of DOFs is $N = (n - 1)^{3}$.

The algorithm extends in the natural way with 2D separators $2^{\ell} mh(j_{1}, \cdot, \cdot)$, $2^{\ell} mh(\cdot, j_{2}, \cdot)$, and $2^{\ell} mh(\cdot, \cdot, j_{3})$ for $1 \leq j_{1}, j_{2}, j_{3} \leq 2^{L - \ell} - 1$ every $2^{\ell} mh = n / 2^{L - \ell}$ units in each direction now partitioning $\Omega$ into interior cubic cells $2^{\ell} mh(j_{1} - 1, j_{1}) \times 2^{\ell} mh(j_{2} - 1, j_{2}) \times 2^{\ell} mh(j_{3} - 1, j_{3})$ at level $\ell$ for $1 \leq j_{1}, j_{2}, j_{3} \leq 2^{L - \ell}$. With this modification, the rest of the algorithm remains unchanged. Figure \ref{fig:mf3} shows the active DOFs at each level for a representative example.
\begin{figure}
 \centering
 \begin{subfigure}{0.2\textwidth}
  \includegraphics[width=\textwidth]{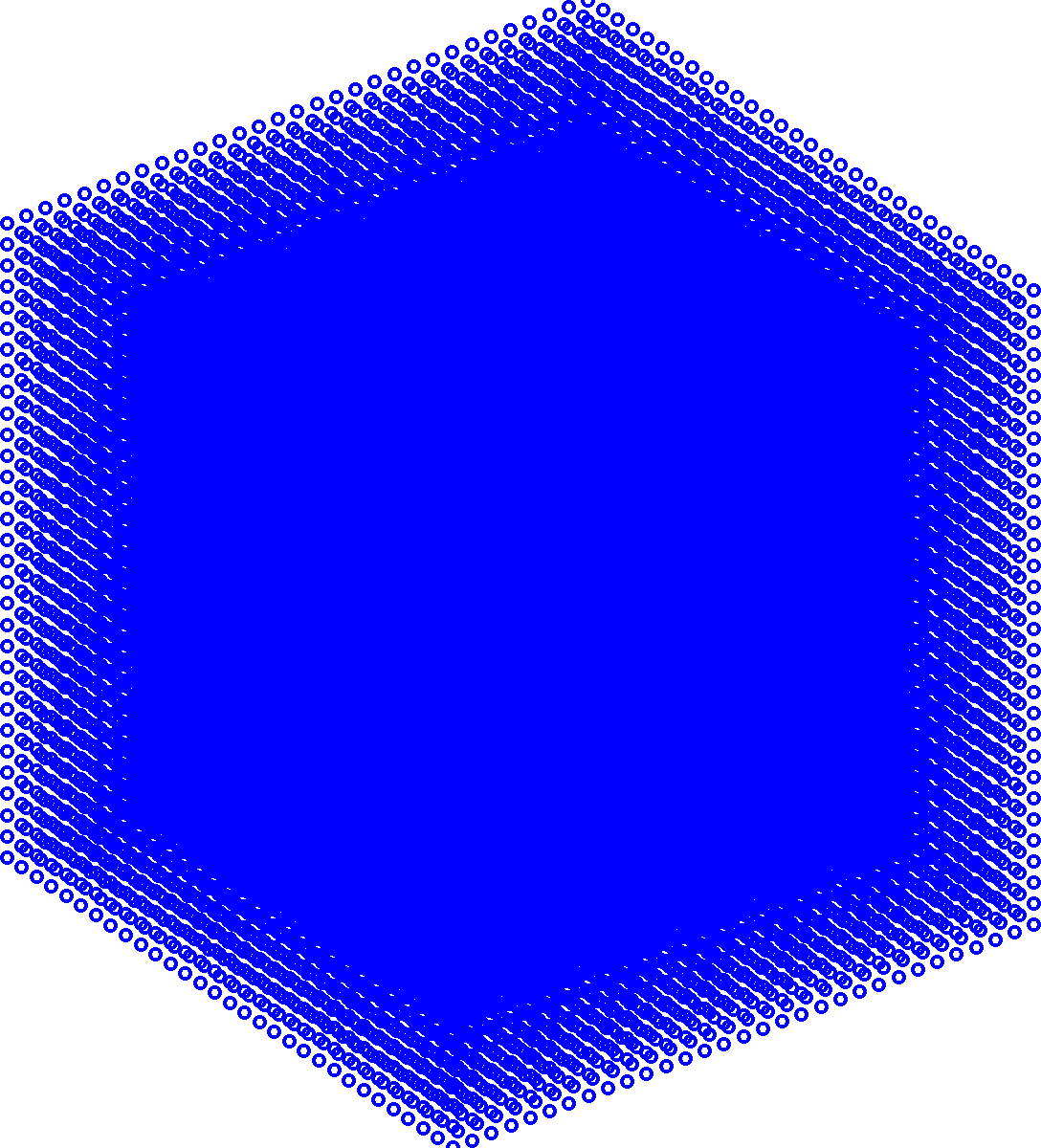}
  \caption*{$\ell = 0$}
 \end{subfigure}
 \quad
 \begin{subfigure}{0.2\textwidth}
  \includegraphics[width=\textwidth]{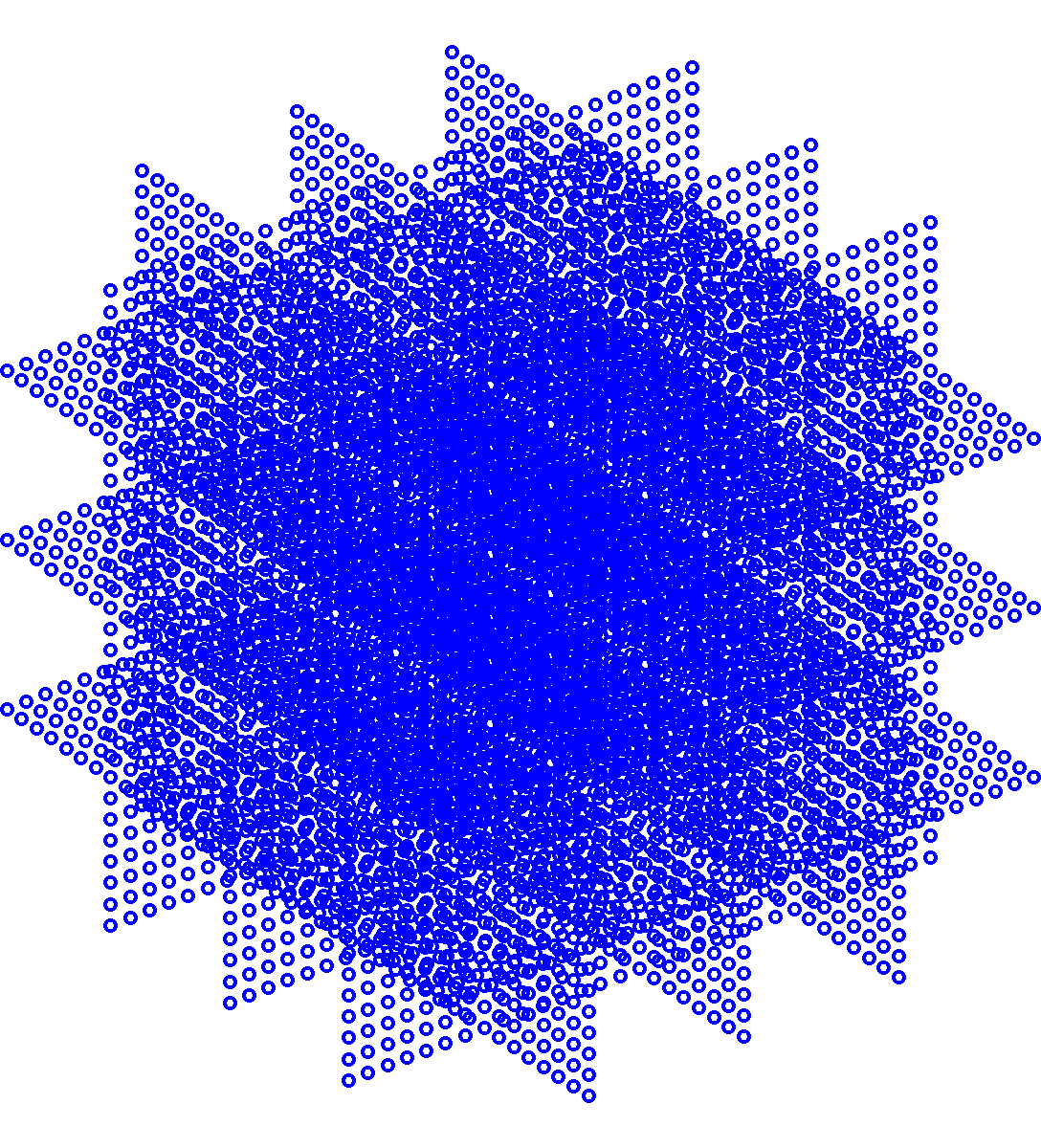}
  \caption*{$\ell = 1$}
 \end{subfigure}
 \quad
 \begin{subfigure}{0.2\textwidth}
  \includegraphics[width=\textwidth]{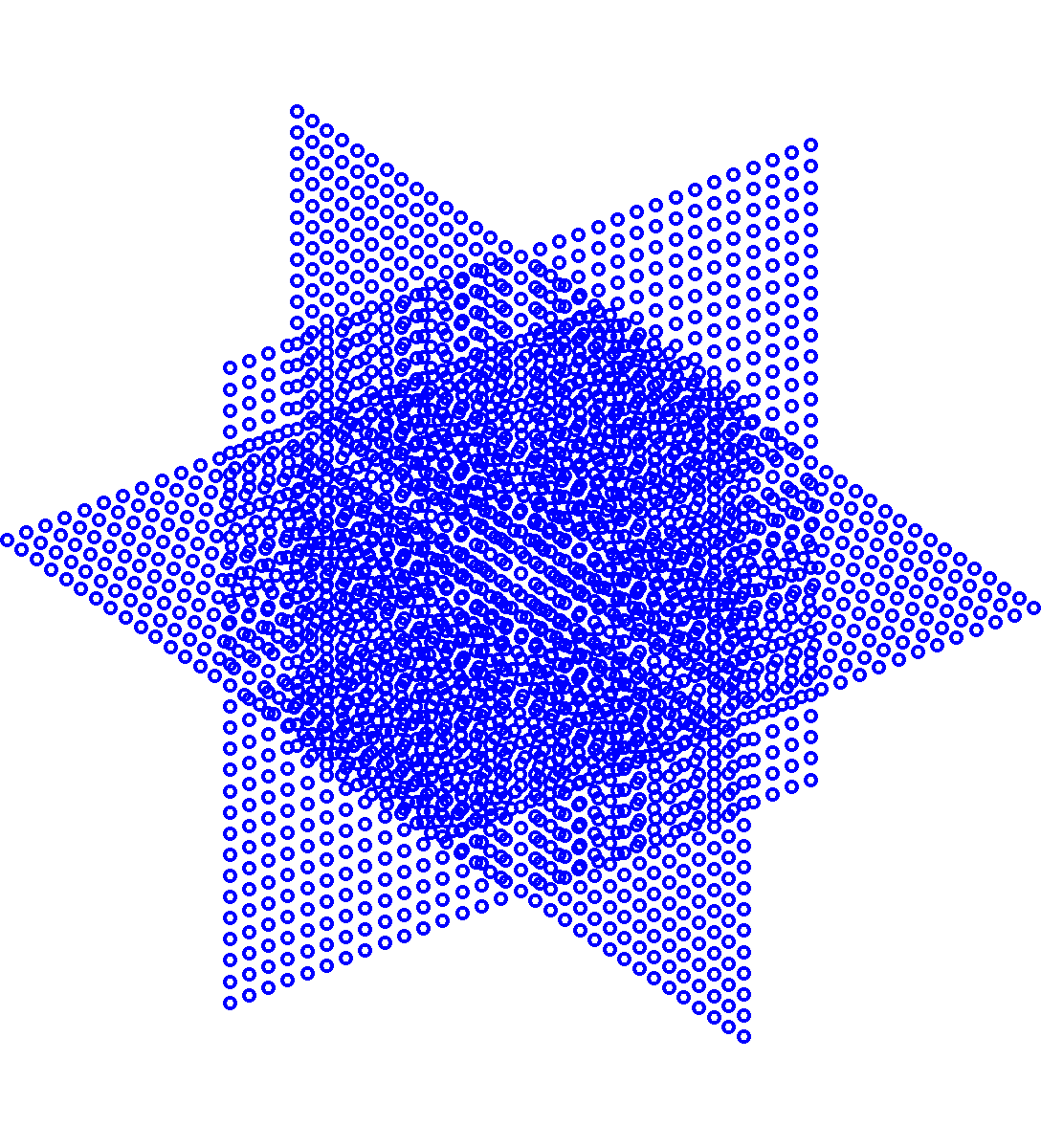}
  \caption*{$\ell = 2$}
 \end{subfigure}
 \caption{Active DOFs at each level $\ell$ of MF in 3D.}
 \label{fig:mf3}
\end{figure}
The output is again a factorization of the form \eqref{eqn:mf}. General geometries can be treated using an adaptive octree.

\subsection{Complexity Estimates}
\label{sec:mf:complexity}
We next analyze the computational complexity of MF. This is determined by the size $|c|$ of a typical index set $c \in C_{\ell}$, which we write as $k_{\ell} = O(2^{(d - 1) \ell})$ following the separator structure. Note furthermore that $|c^{\nbr}| = O(k_{\ell})$ as well since $c^{\nbr}$ is restricted to the separators enclosing the DOFs $c$.

\begin{theorem}[\cite{george:1973:siam-j-numer-anal}]
 \label{thm:mf}
 The cost of constructing the factorization $F$ in \eqref{eqn:mf} using MF is
 \begin{align}
  t_{f} = \sum_{\ell = 0}^{L} 2^{d(L - \ell)} O(k_{\ell}^{3}) =
  \begin{cases}
   O(N), & d = 1\\
   O(N^{3(1 - 1/d)}), & d \geq 2,
  \end{cases}
  \label{eqn:mf-t_f}
 \end{align}
 while that of applying $F$ or $F^{-1}$ is
 \begin{align}
  t_{a/s} = \sum_{\ell = 0}^{L} 2^{d(L - \ell)} O(k_{\ell}^{2}) =
  \begin{cases}
   O(N), & d = 1\\
   O(N \log N), & d = 2\\
   O(N^{2(1 - 1/d)}), & d \geq 3.
  \end{cases}
  \label{eqn:mf-t_as}
 \end{align}
\end{theorem}

\begin{proof}
 Consider first the factorization cost $t_{f}$. There are $2^{d(L - \ell)}$ cells at level $\ell$, where each cell $c \in C_{\ell}$ requires various local dense matrix operations (due to fill-in) at a total cost of $O((|c| + |c^{\nbr}|)^{3}) = O(k_{\ell}^{3})$, following Lemma \ref{lem:sparse-elim}. Hence, we derive \eqref{eqn:mf-t_f}. A similar argument yields \eqref{eqn:mf-t_as} by observing that each $c \in C_{\ell}$ requires local dense matrix-vector products with cost $O((|c|+ |c^{\nbr}|)^{2})$.
\end{proof}

\begin{remark}
 If a tree is used, then there is also a cost of $O(N \log N)$ for tree construction, but the associated constant is tiny and so we can ignore it for all practical purposes.
\end{remark}

The memory cost to store $F$ or $F^{-1}$ is clearly $m_{f} = O(t_{a/s})$ and so is also given by \eqref{eqn:mf-t_as}. Theorem \ref{thm:mf} is, in fact, valid for all $d$, including the 1D case where $k_{\ell} = O(1)$ and optimal linear complexity is achieved. It is immediate that the suboptimal complexities in 2D and 3D are due to the geometric growth of $k_{\ell}$.

\section{Hierarchical Interpolative Factorization}
\label{sec:hifde}
In this section, we present HIF-DE, which builds upon MF by introducing additional levels of compression based on skeletonizing the separator fronts. The key observation is that the Schur complements characterizing the dense frontal matrices accumulated throughout the algorithm possess significant rank structures. This can be understood by interpreting the matrix $A_{pp}^{-1}$ in \eqref{eqn:sparse-elim} as the discrete Green's function of a local elliptic PDE. By elliptic regularity, such Green's functions typically have numerically low-rank off-diagional blocks. The same rank structure essentially carries over to the Schur complement $B_{qq}$ itself, as indeed has previously been recognized and successfully exploited \cite{amestoy:siam-j-sci-comput,aminfar:arxiv,gillman:2014:siam-j-sci-comput,gillman:2014:adv-comput-math,grasedyck:2009:numer-math,martinsson:2009:j-sci-comput,schmitz:2012:j-comput-phys,schmitz:2014:j-comput-phys,xia:2013:siam-j-sci-comput,xia:2013:siam-j-matrix-anal-appl,xia:2009:siam-j-matrix-anal-appl}.

The interaction ranks of the Schur complement interactions (SCIs) constituting $B_{qq}$ have been the subject of several analytic studies \cite{bebendorf:2005:math-comp,bebendorf:2003:numer-math,borm:2010:numer-math,chandrasekaran:2010:siam-j-matrix-anal-appl}, though none have considered the exact type with which we are concerned in this paper. Such an analysis, however, is not our primary goal, and we will be content simply with an empirical description. In particular, we have found through extensive numerical experimentation (Section \ref{sec:results}) that standard multipole estimates \cite{greengard:1987:j-comput-phys,greengard:1997:acta-numer} appear to hold for SCIs. We hereafter take this as an assumption, from which we can expect that the skeletons of a given cluster of DOFs will tend to lie along its boundary \cite{ho:2012:siam-j-sci-comput,ho:comm-pure-appl-math}, thus exhibiting a dimensional reduction.

We are now in a position to motivate HIF-DE. Considering the 2D case for concreteness, the main idea is simply to employ an additional level $\ell + 1/2$ of edge skeletonization after each level $\ell$ of interior cell elimination. This fully exploits the 1D geometry of the active DOFs and effectively reduces each front to 0D. An analogous strategy is adopted in 3D for reduction to either 1D by skeletonizing cubic faces or to 0D by skeletonizing faces then edges. In principle, the latter is more efficient but can generate fill-in and so must be used with care.

The overall approach of HIF-DE is closely related to that of \cite{aminfar:arxiv,gillman:2014:siam-j-sci-comput,gillman:2014:adv-comput-math,grasedyck:2009:numer-math,schmitz:2012:j-comput-phys,schmitz:2014:j-comput-phys,xia:2009:siam-j-matrix-anal-appl}, but our sparsification framework permits a much simpler implementation and analysis. As with MF, we begin first in 2D before extending to 3D.

\subsection{Two Dimensions}
\label{sec:hifde:2d}
Assume the same setup as Section \ref{sec:mf:2d}. HIF-DE supplements interior cell elimination (2D to 1D) at level $\ell$ with edge skeletonization (1D to 0D) at level $\ell + 1/2$ for each $\ell = 0, 1, \dots, L - 1$. Figure \ref{fig:hifde2} shows the active DOFs at each level for a representative example.
\begin{figure}
 \centering
 \begin{subfigure}{0.2\textwidth}
  \includegraphics[width=\textwidth]{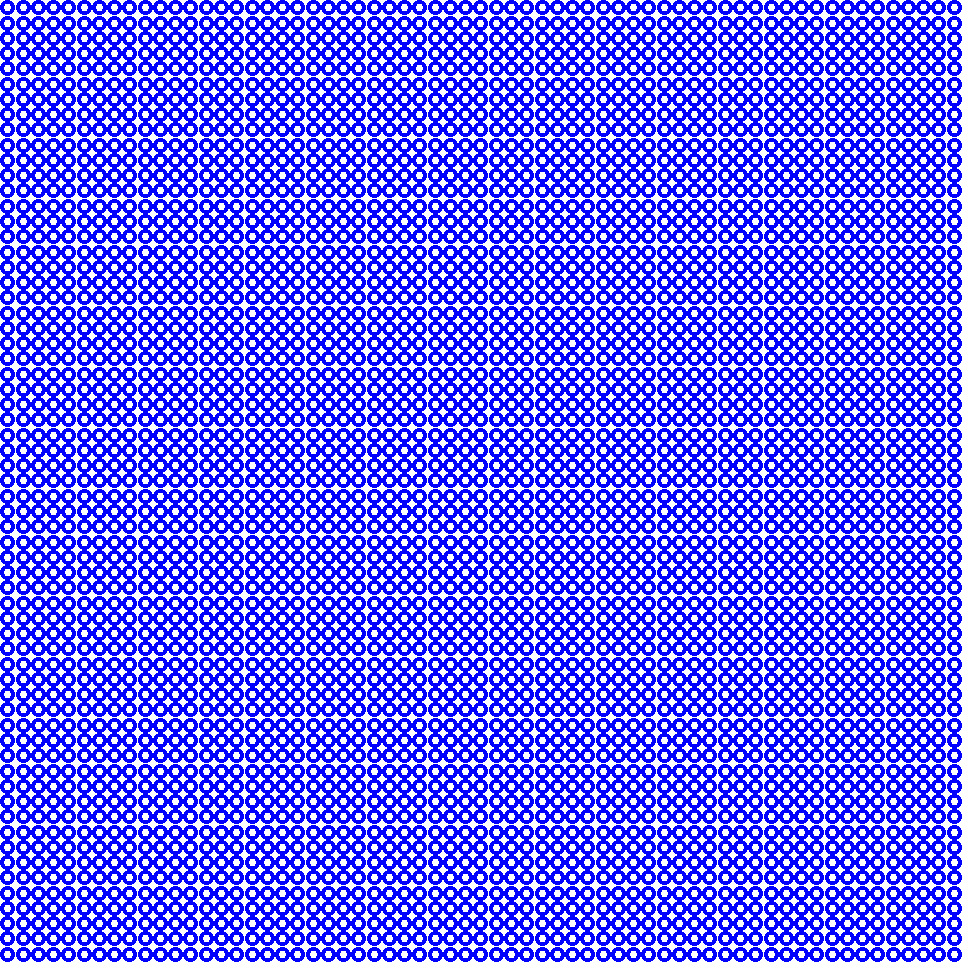}
  \caption*{$\ell = 0$}
 \end{subfigure}
 \quad
 \begin{subfigure}{0.2\textwidth}
  \includegraphics[width=\textwidth]{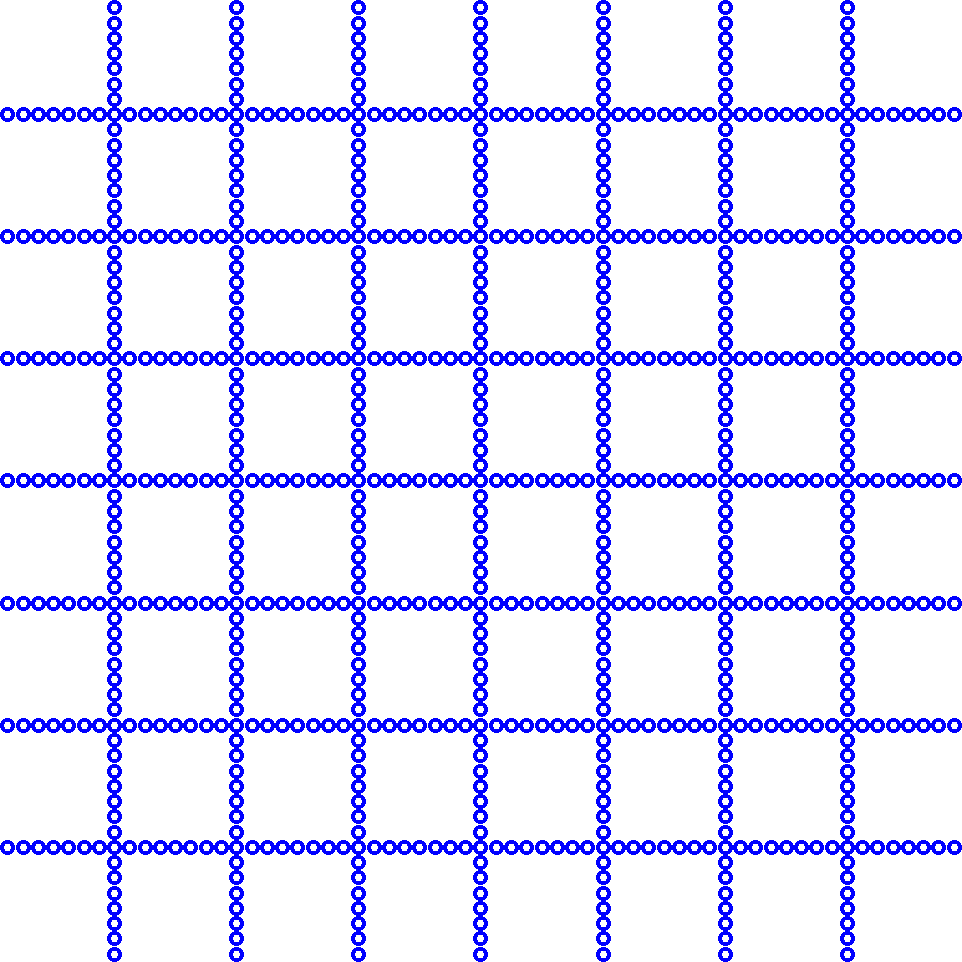}
  \caption*{$\ell = 1/2$}
 \end{subfigure}
 \quad
 \begin{subfigure}{0.2\textwidth}
  \includegraphics[width=\textwidth]{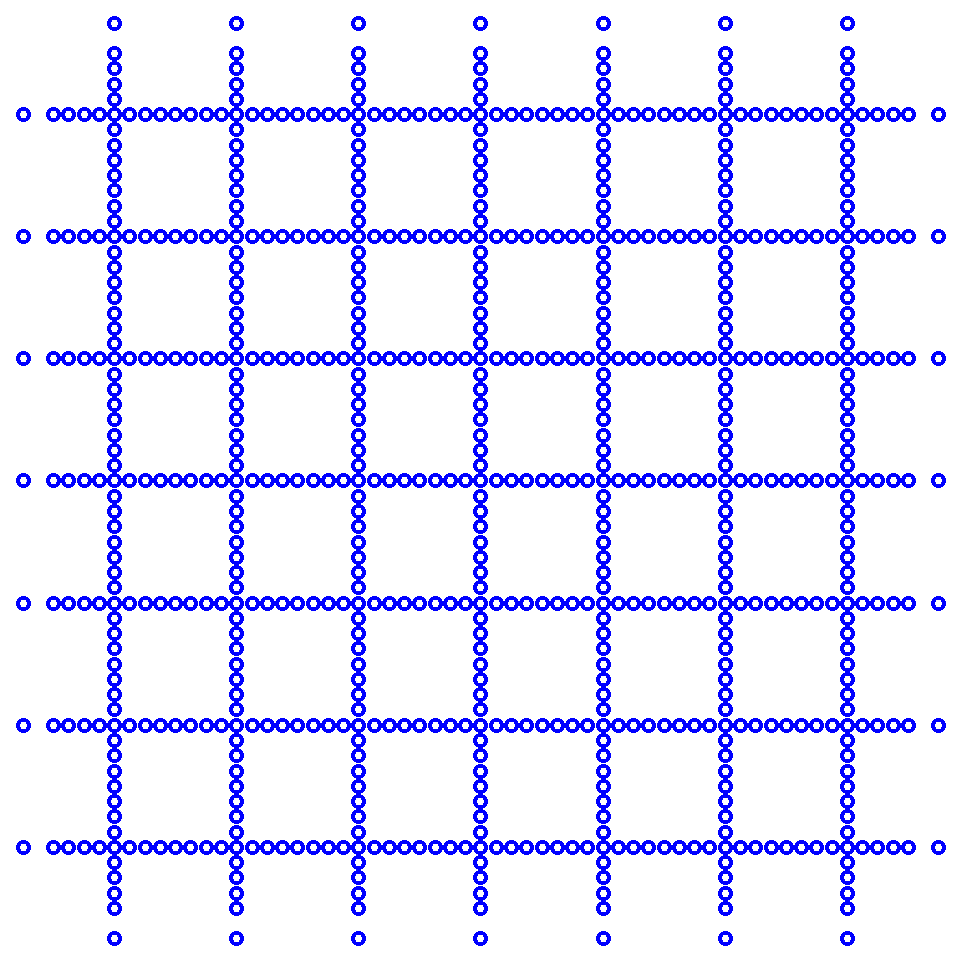}
  \caption*{$\ell = 1$}
 \end{subfigure}
 \quad
 \begin{subfigure}{0.2\textwidth}
  \includegraphics[width=\textwidth]{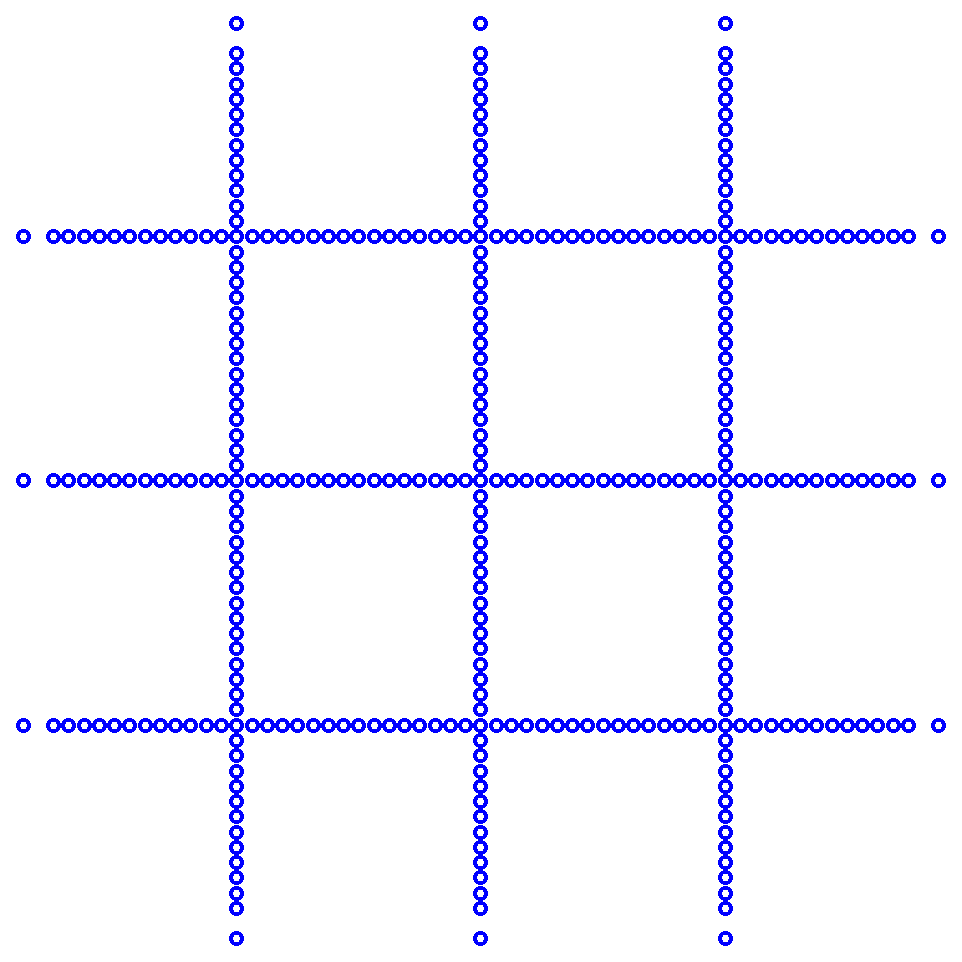}
  \caption*{$\ell = 3/2$}
 \end{subfigure}
 \\~\\~\\
 \begin{subfigure}{0.2\textwidth}
  \includegraphics[width=\textwidth]{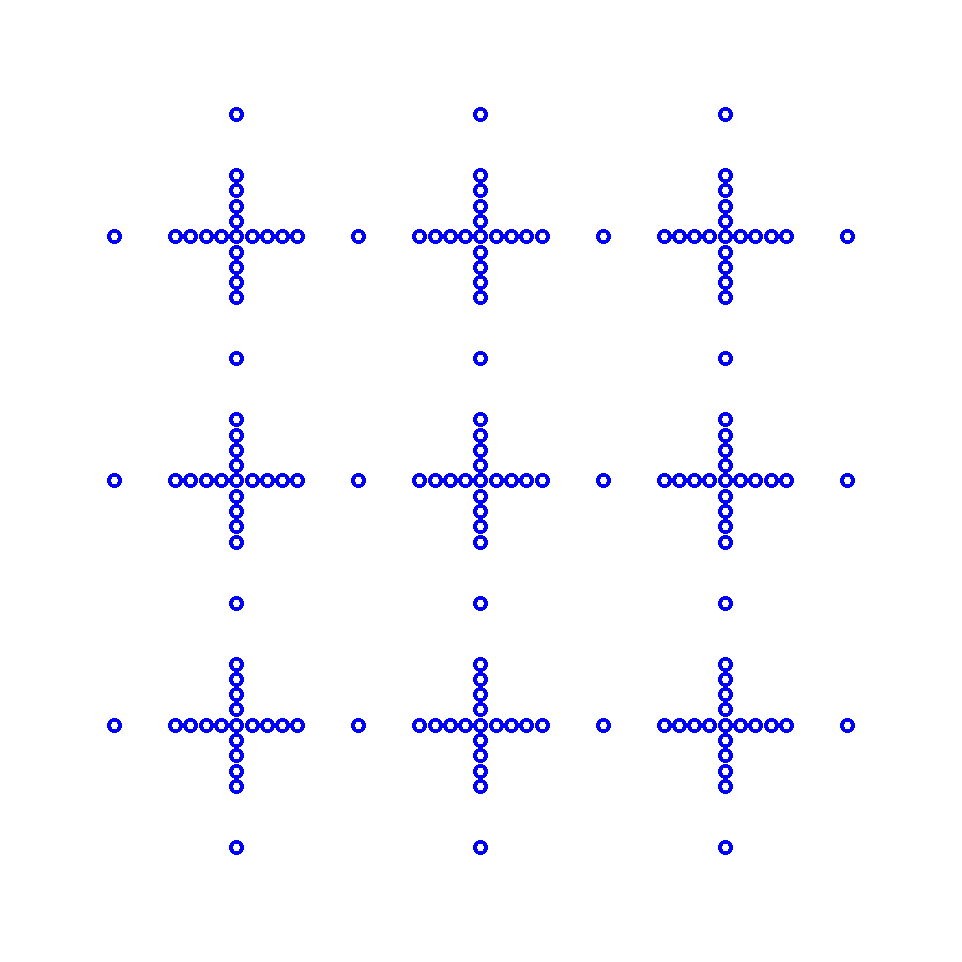}
  \caption*{$\ell = 2$}
 \end{subfigure}
 \quad
 \begin{subfigure}{0.2\textwidth}
  \includegraphics[width=\textwidth]{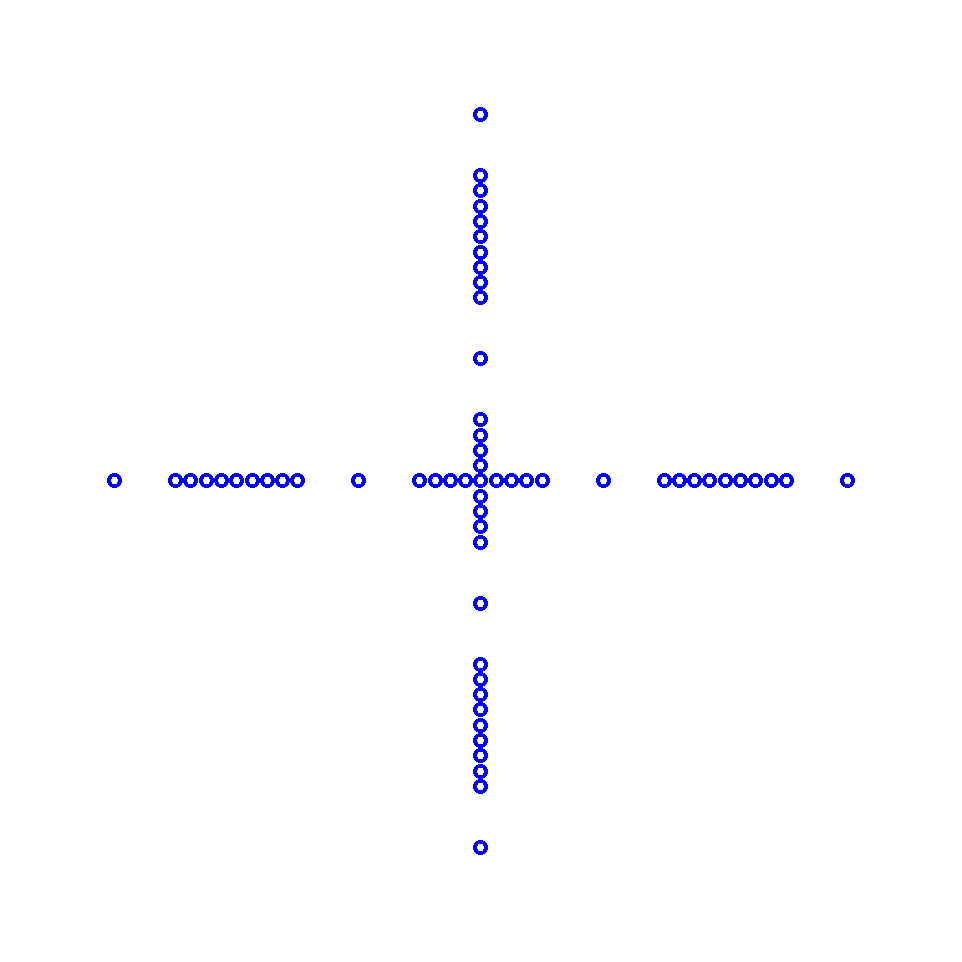}
  \caption*{$\ell = 5/2$}
 \end{subfigure}
 \quad
 \begin{subfigure}{0.2\textwidth}
  \includegraphics[width=\textwidth]{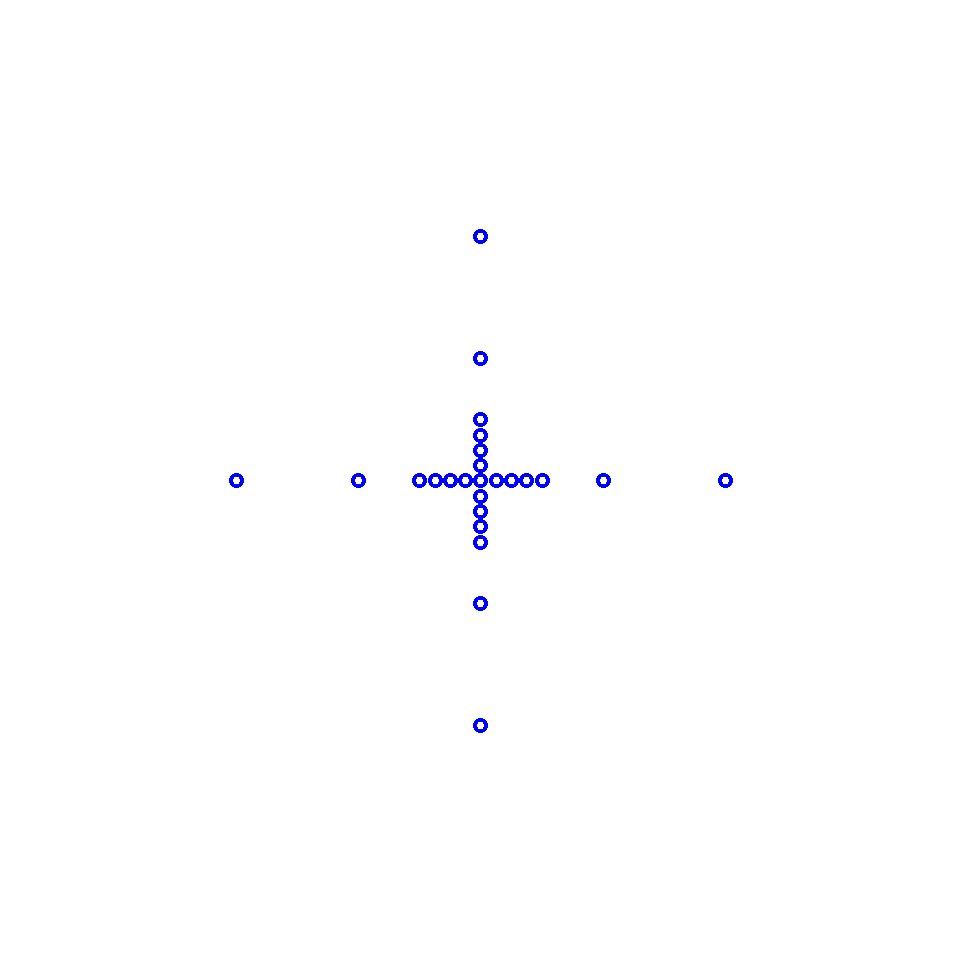}
  \caption*{$\ell = 3$}
 \end{subfigure}
 \caption{Active DOFs at each level $\ell$ of HIF-DE in 2D.}
 \label{fig:hifde2}
\end{figure}

\subsubsection*{Level $\ell$}
Partition $\Omega$ by 1D separators $2^{\ell} mh(j_{1}, \cdot)$ and $2^{\ell} mh(\cdot, j_{2})$ for $1 \leq j_{1}, j_{2} \leq 2^{L - \ell} - 1$ into interior square cells $2^{\ell} mh(j_{1} - 1, j_{1}) \times 2^{\ell} mh(j_{2} - 1, j_{2})$ for $1 \leq j_{1}, j_{2} \leq 2^{L - \ell}$. Let $C_{\ell}$ be the collection of index sets corresponding to the active DOFs of each cell. Elimination with respect to $C_{\ell}$ then gives
\begin{align*}
 A_{\ell + 1/2} = W_{\ell}^{\trans} A_{\ell} W_{\ell}, \quad W_{\ell} = \prod_{c \in C_{\ell}} S_{c},
\end{align*}
where the DOFs $\bigcup_{c \in C_{\ell}} c$ have been eliminated. The matrix $A_{\ell + 1/2}$ is block diagonal with block partitioning
\begin{align*}
 \theta_{\ell + 1/2} = \left( \bigcup_{c \in C_{0}} \{ c \} \right) \cup \left( \bigcup_{c \in C_{1/2}} \{ \rd{c} \} \right) \cdots \cup \left( \bigcup_{c \in C_{\ell}} \{ c \} \right) \cup \{ s_{\ell + 1/2} \},
\end{align*}
where $s_{\ell + 1/2} = s_{\ell} \setminus \bigcup_{c \in C_{\ell}} c$.

\subsubsection*{Level $\ell + 1/2$}
Partition $\Omega$ into Voronoi cells \cite{aurenhammer:1991:acm-comput-surv} about the edge centers $2^{\ell} mh(j_{1}, j_{2} - 1/2)$ for $1 \leq j_{1} \leq 2^{L - \ell} - 1$, $1 \leq j_{2} \leq 2^{L - \ell}$ and $2^{\ell} mh(j_{1} - 1/2, j_{2})$ for $1 \leq j_{1} \leq 2^{L - \ell}$, $1 \leq j_{2} \leq 2^{L - \ell} - 1$. Let $C_{\ell + 1/2}$ be the collection of index sets corresponding to the active DOFs of each cell. Skeletonization with respect to $C_{\ell + 1/2}$ then gives
\begin{align*}
 A_{\ell + 1} = \skel_{C_{\ell + 1/2}} (A_{\ell + 1/2}) \approx U_{\ell + 1/2}^{\trans} A_{\ell + 1/2} U_{\ell + 1/2}, \quad U_{\ell + 1/2} = \prod_{c \in C_{\ell + 1/2}} Q_{c} S_{\rd{c}},
\end{align*}
where the DOFs $\bigcup_{c \in C_{\ell + 1/2}} \rd{c}$ have been eliminated. Note that no fill-in is generated since the DOFs $\sk{c}$ for each $c \in C_{\ell + 1/2}$ are already connected via SCIs from elimination at level $\ell$. The matrix $A_{\ell + 1}$ is block diagonal with block partitioning
\begin{align*}
 \theta_{\ell + 1} = \left( \bigcup_{c \in C_{0}} \{ c \} \right) \cup \left( \bigcup_{c \in C_{1/2}} \{ \rd{c} \} \right) \cdots \cup \left( \bigcup_{c \in C_{\ell}} \{ c \} \right) \cup \left( \bigcup_{c \in C_{\ell + 1/2}} \{ \rd{c} \} \right) \cup \{ s_{\ell + 1/2} \},
\end{align*}
where $s_{\ell + 1} = s_{\ell + 1/2} \setminus \bigcup_{c \in C_{\ell}} \rd{c}$.

\subsubsection*{Level $L$}
Combining over all levels gives
\begin{align*}
 D \equiv A_{L} \approx U_{L - 1/2}^{\trans} W_{L - 1}^{\trans} \cdots U_{1/2}^{\trans} W_{0}^{\trans} A W_{0} U_{1/2} \cdots W_{L - 1} U_{L - 1/2}
\end{align*}
or, more simply,
\begin{align*}
 D \approx V_{L - 1/2}^{\trans} \cdots V_{1/2}^{\trans} V_{0}^{\trans} A V_{0} V_{1/2} \cdots V_{L - 1/2},
\end{align*}
where
\begin{align}
 V_{\ell} =
 \begin{cases}
  W_{\ell}, & \ell = 0, 1, \dots, L - 1\\
  U_{\ell}, & \text{otherwise},
 \end{cases}
 \label{eqn:matrix-factor}
\end{align}
so
\begin{subequations}
 \label{eqn:hifde2}
 \begin{align}
  A &\approx V_{0}^{-\trans} V_{1/2}^{-\trans} \cdots V_{L - 1/2}^{-\trans} D V_{L - 1/2}^{-1} \cdots V_{1/2}^{-1} V_{0}^{-1} \equiv F,\\
  A^{-1} &\approx V_{0} V_{1/2} \cdots V_{L - 1/2} D^{-1} V_{L - 1/2}^{\trans} \cdots V_{1/2}^{\trans} V_{0}^{\trans} = F^{-1}.
 \end{align}
\end{subequations}
This is a factorization very similar to that in \eqref{eqn:mf} except (1) it has twice as many factors, (2) it is now an approximation, and (3) the skeletonization matrices $U_{\ell}$ are composed of both upper and lower triangular factors and so are not themselves triangular (but are still easily invertible). We call \eqref{eqn:hifde2} an approximate generalized LDL decomposition, with $F^{-1}$ serving as a direct solver at high accuracy or as a preconditioner otherwise.

Unlike MF, if $A$ is SPD, then $D$ and hence $F$ now only approximate SPD matrices. The extent of this approximation is governed by Weyl's inequality.

\begin{theorem}
 \label{thm:weyl}
 If $A, B \in \mathbb{R}^{N \times N}$ are symmetric, then
 \begin{align*}
  |\lambda_{i} (A) - \lambda_{i} (B)| \leq \| A - B \|, \quad i = 1, \dots, N,
 \end{align*}
 where $\lambda_{i} (\cdot)$ returns the $i$th largest eigenvalue of a symmetric matrix.
\end{theorem}

\begin{corollary}
 If $A$ is SPD with $F = A + E$ symmetric such that $\| E \| \leq \epsilon \| A \|$ for $\epsilon \kappa (A) < 1$, where $\kappa (A) = \| A \| \| A^{-1} \|$ is the condition number of $A$, then $F$ is SPD.
\end{corollary}

\begin{proof}
 By Theorem \ref{thm:weyl}, $|\lambda_{i} (A) - \lambda_{i} (F)| \leq \| E \| = \epsilon \| A \|$ for all $i = 1, \dots, N$, so
 \begin{align}
  \left| \frac{\lambda_{i} (A) - \lambda_{i} (F)}{\lambda_{i} (A)} \right| \leq \left| \frac{\lambda_{i} (A) - \lambda_{i} (F)}{\lambda_{N} (A)} \right| \leq \epsilon \| A \| \| A^{-1} \| = \epsilon \kappa (A).
  \label{eqn:eigenvalues}
 \end{align}
 This implies that $\lambda_{i} (F) \geq (1 - \epsilon \kappa (A)) \, \lambda_{i} (A)$, so $\lambda_{i} (F) > 0$ if $\epsilon \kappa (A) < 1$ since $\lambda_{i} (A) > 0$ by assumption.
\end{proof}

\begin{remark}
 Equation \eqref{eqn:eigenvalues} actually proves a much more general result, namely that all eigenvalues are approximated to relative precision $\epsilon \kappa (A)$.
\end{remark}

The requirement that $\epsilon \kappa (A) < 1$ is necessary for $F^{-1}$ to achieve any accuracy whatsoever and hence is quite weak. Therefore, $F$ is SPD under very mild conditions, in which case \eqref{eqn:hifde2} can be interpreted as a generalized Cholesky decomposition. Its inverse $F^{-1}$ is then also SPD and can be used, e.g., as a preconditioner in CG.

The entire procedure is summarized as Algorithm \ref{alg:hifde2}.
\begin{algorithm}
 \caption{HIF-DE.}
 \label{alg:hifde2}
 \begin{algorithmic}
  \State $A_{0} = A$ \Comment{initialize}
  \For{$\ell = 0, 1, \dots, L - 1$} \Comment{loop from finest to coarsest level}
   \State $A_{\ell + 1/2} = W_{\ell}^{\trans} A_{\ell} W_{\ell}$ \Comment{eliminate interior cells}
   \State $A_{\ell + 1} = \skel_{C_{\ell + 1/2}} (A_{\ell + 1/2}) \approx U_{\ell + 1/2}^{\trans} A_{\ell + 1/2} U_{\ell + 1/2}$ \Comment{skeletonize edges (faces)}
  \EndFor
  \State $A \approx V_{0}^{-\trans} V_{1/2}^{-\trans} \cdots V_{L - 1/2}^{-\trans} A_{L} V_{L - 1/2}^{-1} \cdots V_{1/2}^{-1} V_{0}^{-1}$ \Comment{generalized LDL decomposition}
 \end{algorithmic}
\end{algorithm}

\subsection{Three Dimensions}
\label{sec:hifde:3d}
Assume the same setup as in Section \ref{sec:mf:3d}. There are two variants of HIF-DE in 3D: a direct generalization of the 2D algorithm by combining interior cell elimination (3D to 2D) with face skeletonization (2D to 1D) and a more complicated version adding also edge skeletonization (1D to 0D) afterward. We will continue to refer to the former simply as HIF-DE and call the latter ``HIF-DE in 3D with edge skeletonization''. For unity of presentation, we will discuss only HIF-DE here, postponing the alternative formulation until Section \ref{sec:hifde:3dx}. Figure \ref{fig:hifde3} shows the active DOFs at each level for HIF-DE on a representative example.
\begin{figure}
 \centering
 \begin{subfigure}{0.2\textwidth}
  \includegraphics[width=\textwidth]{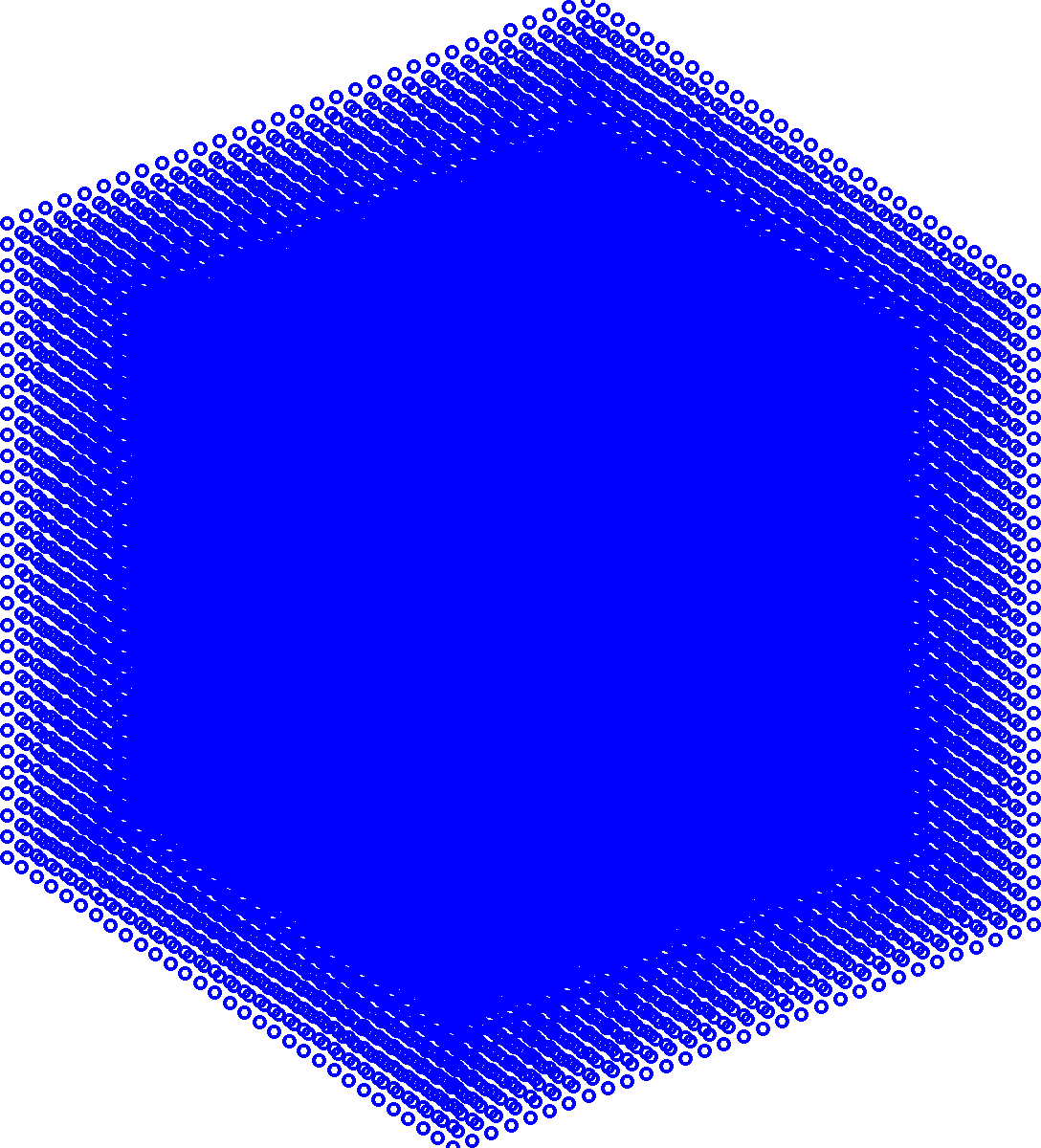}
  \caption*{$\ell = 0$}
 \end{subfigure}
 \quad
 \begin{subfigure}{0.2\textwidth}
  \includegraphics[width=\textwidth]{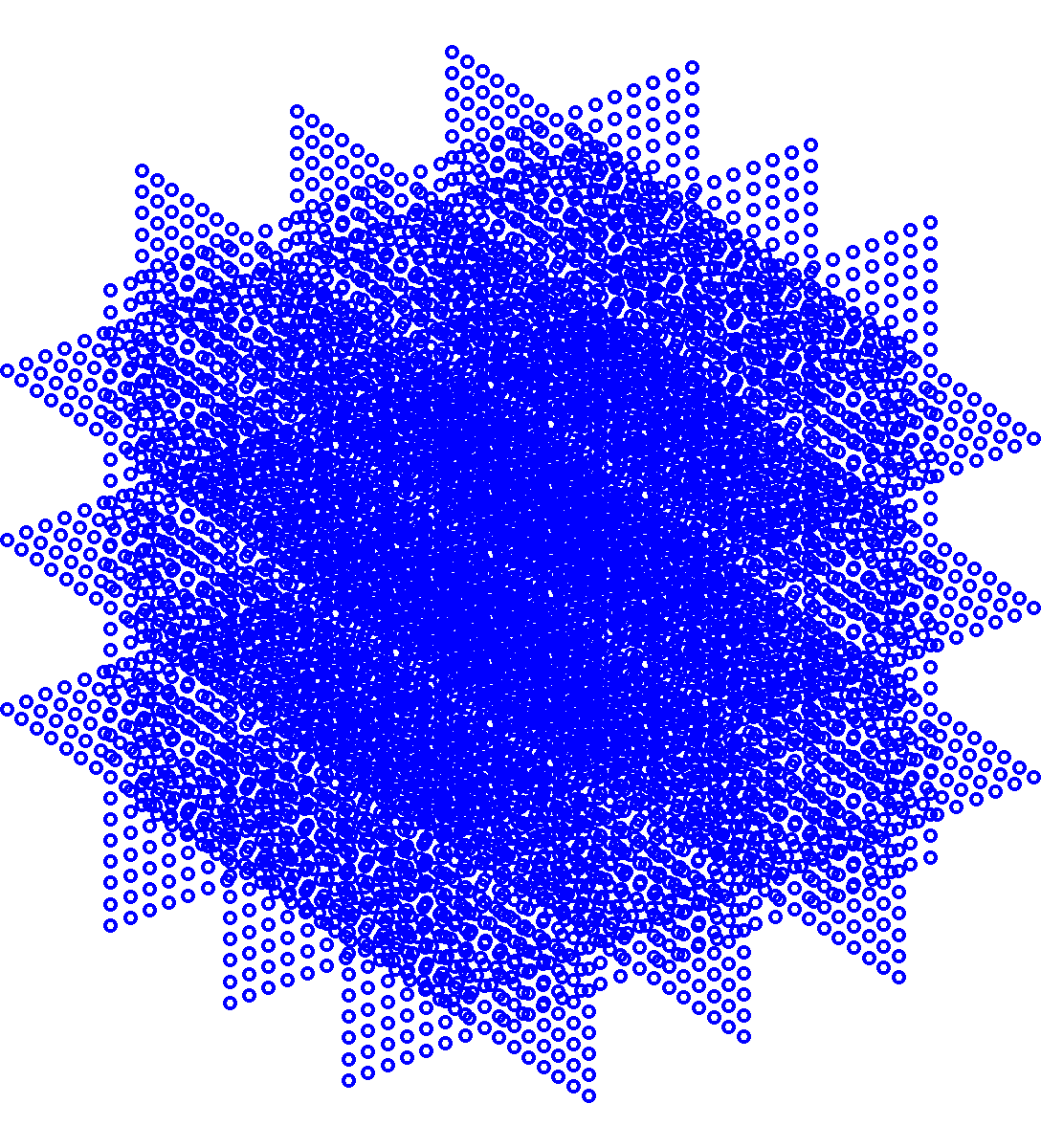}
  \caption*{$\ell = 1/2$}
 \end{subfigure}
 \quad
 \begin{subfigure}{0.2\textwidth}
  \includegraphics[width=\textwidth]{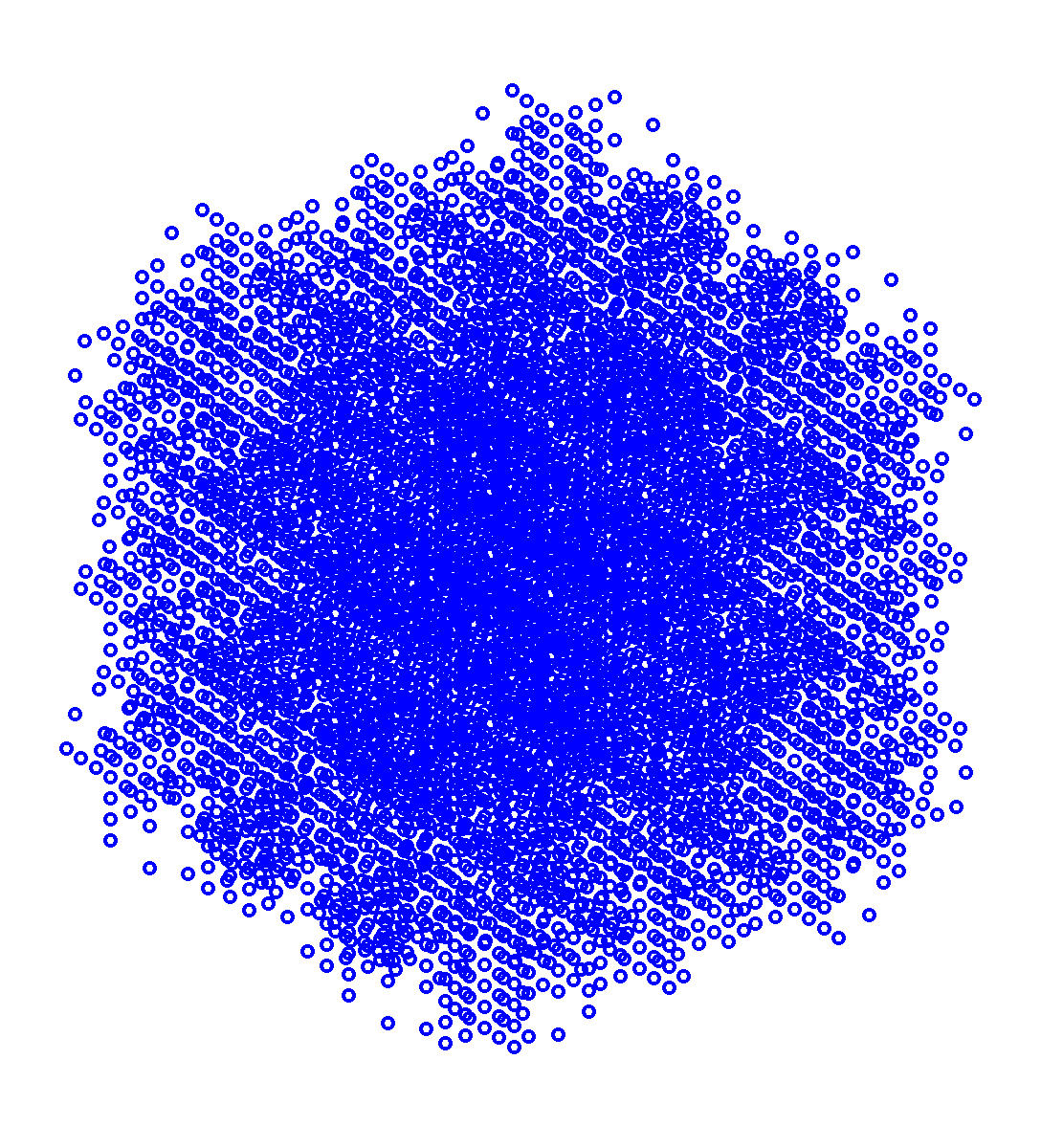}
  \caption*{$\ell = 1$}
 \end{subfigure}
 \\~\\~\\
 \begin{subfigure}{0.2\textwidth}
  \includegraphics[width=\textwidth]{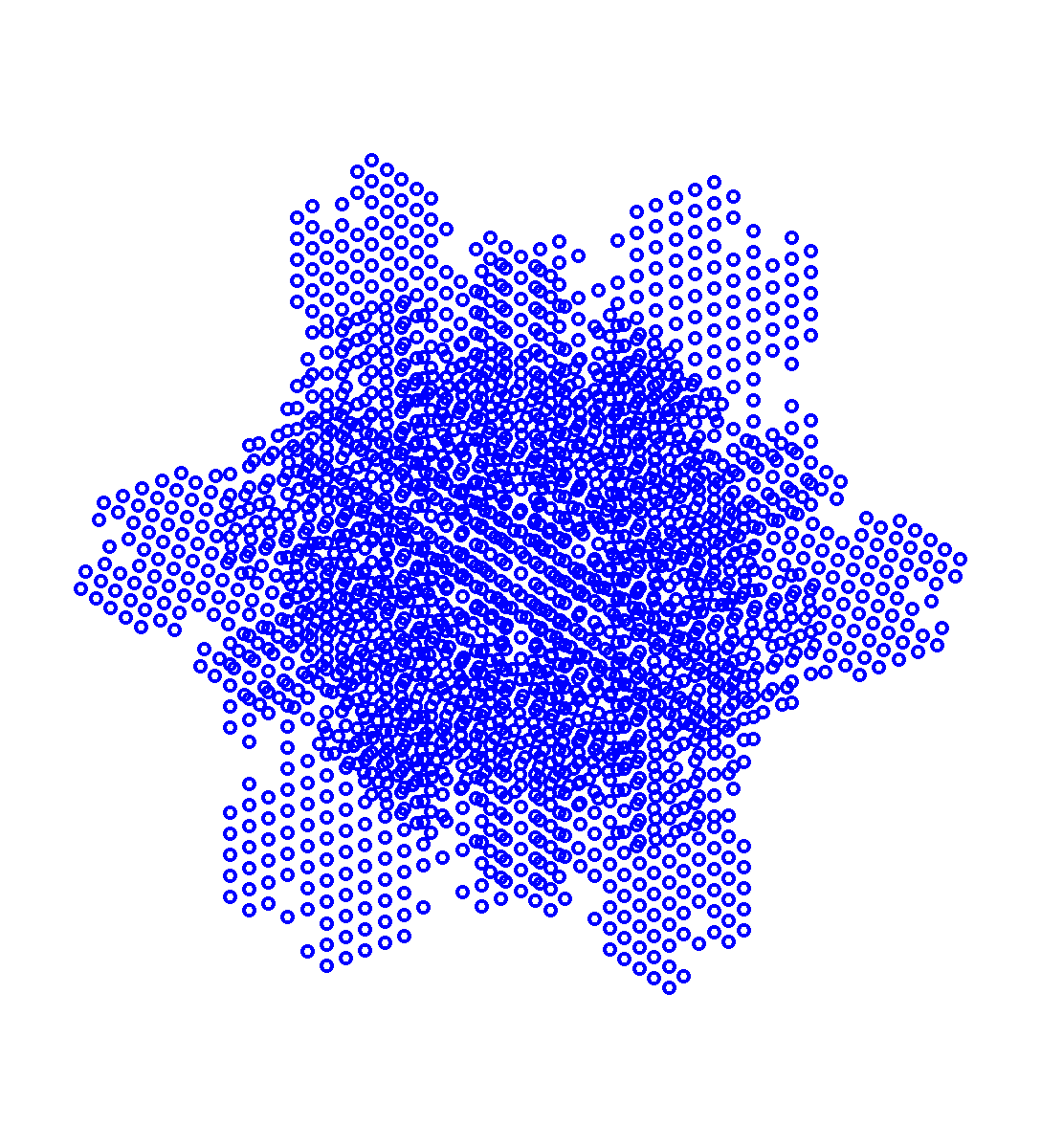}
  \caption*{$\ell = 3/2$}
 \end{subfigure}
 \quad
 \begin{subfigure}{0.2\textwidth}
  \includegraphics[width=\textwidth]{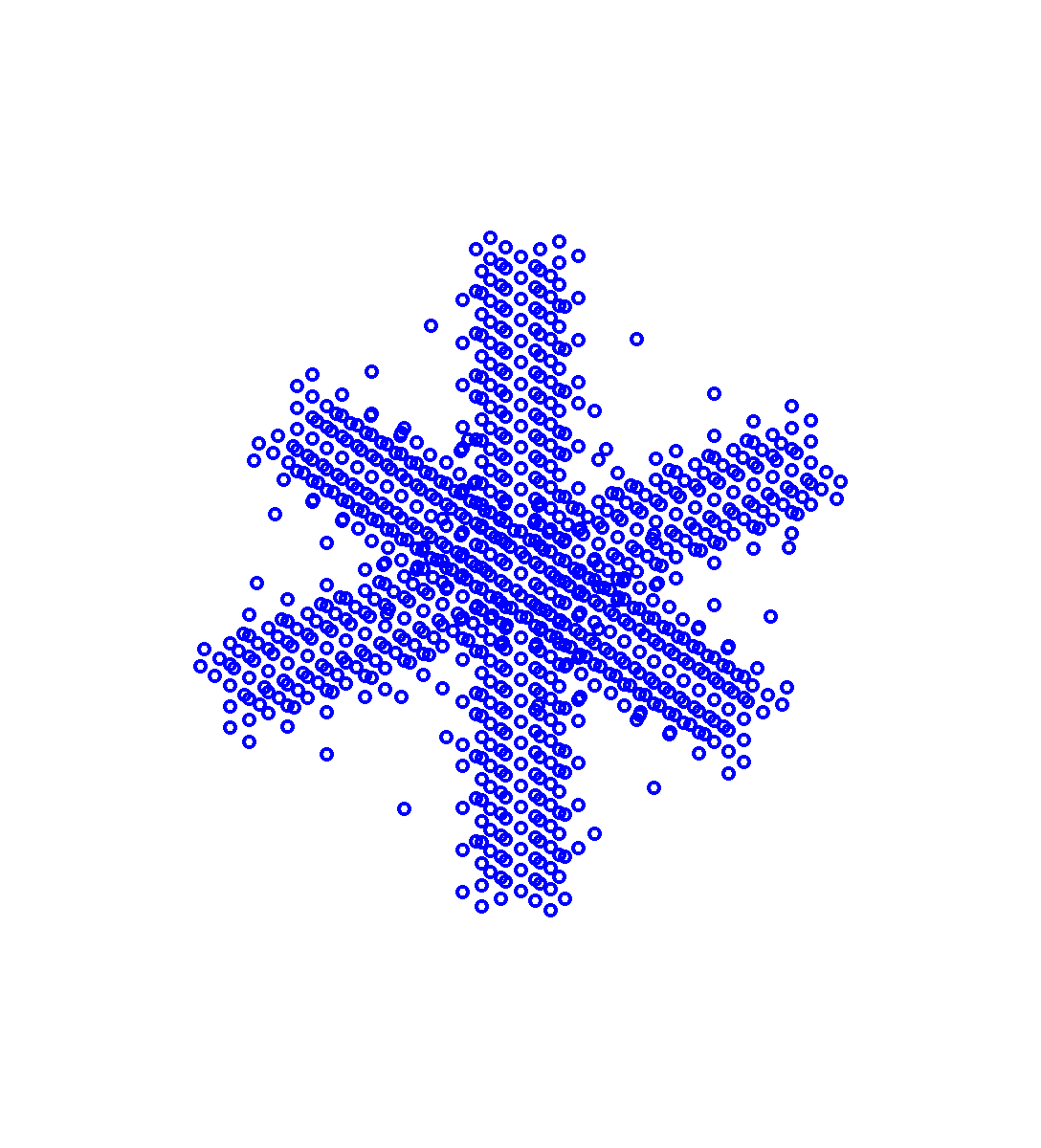}
  \caption*{$\ell = 2$}
 \end{subfigure}
 \caption{Active DOFs at each level $\ell$ of HIF-DE in 3D.}
 \label{fig:hifde3}
\end{figure}

\subsubsection*{Level $\ell$}
Partition $\Omega$ by 2D separators $2^{\ell} mh(j_{1}, \cdot, \cdot)$, $2^{\ell} mh(\cdot, j_{2}, \cdot)$, and $2^{\ell} mh(\cdot, \cdot, j_{3})$ for $1 \leq j_{1}, j_{2}, j_{3} \leq 2^{L - \ell} - 1$ into interior cubic cells $2^{\ell} mh(j_{1} - 1, j_{1}) \times 2^{\ell} mh(j_{2} - 1, j_{2}) \times 2^{\ell} mh(j_{3} - 1, j_{3})$ for $1 \leq j_{1}, j_{2}, j_{3} \leq 2^{L - \ell}$. Let $C_{\ell}$ be the collection of index sets corresponding to the active DOFs of each cell. Elimination with respect to $C_{\ell}$ then gives
\begin{align*}
 A_{\ell + 1/2} = W_{\ell}^{\trans} A_{\ell} W_{\ell}, \quad W_{\ell} = \prod_{c \in C_{\ell}} S_{c},
\end{align*}
where the DOFs $\bigcup_{c \in C_{\ell}} c$ have been eliminated.

\subsubsection*{Level $\ell + 1/2$}
Partition $\Omega$ into Voronoi cells about the face centers
\begin{align*}
 &2^{\ell} mh \left( j_{1}, j_{2} - \frac{1}{2}, j_{3} - \frac{1}{2} \right), & &1 \leq j_{1} \leq 2^{L - \ell} - 1, & &1 \leq j_{2}, j_{3} \leq 2^{L - \ell},\\
 &2^{\ell} mh \left( j_{1} - \frac{1}{2}, j_{2}, j_{3} - \frac{1}{2} \right), & &1 \leq j_{2} \leq 2^{L - \ell} - 1, & &1 \leq j_{1}, j_{3} \leq 2^{L - \ell},\\
 &2^{\ell} mh \left( j_{1} - \frac{1}{2}, j_{2} - \frac{1}{2}, j_{3} \right), & &1 \leq j_{3} \leq 2^{L - \ell} - 1, & &1 \leq j_{1}, j_{2} \leq 2^{L - \ell}.
\end{align*}
Let $C_{\ell + 1/2}$ be the collection of index sets corresponding to the active DOFs of each cell. Skeletonization with respect to $C_{\ell + 1/2}$ then gives
\begin{align*}
 A_{\ell + 1} = \skel_{C_{\ell + 1/2}} (A_{\ell + 1/2}) \approx U_{\ell + 1/2}^{\trans} A_{\ell + 1/2} U_{\ell + 1/2}, \quad U_{\ell + 1/2} = \prod_{c \in C_{\ell + 1/2}} Q_{c} S_{\rd{c}},
\end{align*}
where the DOFs $\bigcup_{c \in C_{\ell + 1/2}} \rd{c}$ have been eliminated.

\subsubsection*{Level $L$}
Combining the approximation over all levels gives a factorization of the form \eqref{eqn:hifde2}. The overall procedure is the same as that in Algorithm \ref{alg:hifde2}.

\subsection{Accelerated Compression}
\label{sec:hifde:accel-comp}
A dominant contribution to the cost of HIF-DE is computing IDs for skeletonization. The basic operation required is the construction of an ID of $(A_{\ell + 1/2})_{c^{\cmp},c}$, where $c \in C_{\ell + 1/2}$ and $c^{\cmp} = s_{\ell + 1/2} \setminus c$, following Lemma \ref{lem:skel}. We hereafter drop the dependence on $\ell$ for notational convenience. Note that $A_{c^{\cmp},c}$ is a tall-and-skinny matrix of size $O(N) \times |c|$, so forming its ID takes at least $O(N|c|)$ work. By construction, however, $A_{c^{\cmp},c}$ is very sparse and can be written without loss of generality as
\begin{align*}
 A_{c^{\cmp},c} =
 \begin{bmatrix}
  A_{c^{\nbr},c}\\
  0
 \end{bmatrix},
\end{align*}
where the DOFs $c^{\nbr}$ are restricted to the immediately adjacent edges or faces, as appropriate. Thus, $|c^{\nbr}| = O(|c|)$ and an ID of the much smaller matrix $A_{c^{\nbr},c}$ of size $O(|c|) \times |c|$ suffices. In other words, the global compression of $A_{c^{\cmp},c}$ can be performed via the local compression of $A_{c^{\nbr},c}$. This observation is critical for reducing the asymptotic complexity.

We can pursue further acceleration by optimizing $|c^{\nbr}|$ as follows. Consider the reference domain configuration depicted in Figure \ref{fig:accel-comp}, which shows the active DOFs $s_{\ell + 1/2}$ after interior cell elimination at level $\ell$ in 2D.
\begin{figure}
 \includegraphics{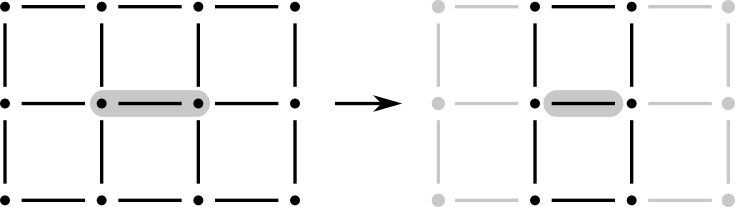}
 \caption{Accelerated compression by exploiting sparsity. In 2D, the number of neighboring edges (left) that must be included when skeletonizing a given edge (gray outline) can be substantially reduced by restricting to only the interior DOFs of that edge (right). An analogous setting applies for faces in 3D.}
 \label{fig:accel-comp}
\end{figure}
The Voronoi partitioning scheme clearly groups together all interior DOFs of each edge, but those at the corner points $2^{\ell} mh(j_{1}, j_{2})$ for $1 \leq j_{1}, j_{2} \leq 2^{L - \ell} - 1$ are equidistant to multiple Voronoi centers and can be assigned arbitrarily (or, in fact, not at all). Let $c \in C_{\ell + 1/2}$ be a given edge and suppose that it includes both of its endpoints. Then its neighbor set $c^{\nbr}$ includes all immediately adjacent edges as shown (left). But the only DOFs in $c$ that interact with the edges to the left or right are precisely the corresponding corner points. Therefore, we can reduce $c^{\nbr}$ to only those edges belonging to the two cells on either side of the edge defining $c$ by restricting to only its interior DOFs (right), i.e., we exclude from $C_{\ell + 1/2}$ all corner points. This can also be interpreted as pre-selecting the corner points as skeletons (as must be the case because of the sparsity pattern of $A_{c^{\cmp},c}$) and suitably modifying the remaining computation. In 2D, this procedure lowers the cost of the ID by about a factor of $17/6 = 2.8333...$. In 3D, an analogous situation holds for faces with respect to ``corner'' edges and the cost is reduced by a factor of $37/5 = 7.4$.

It is also possible to accelerate the ID using fast randomized methods \cite{halko:2011:siam-rev} based on compressing $\Phi_{c} A_{c^{\nbr},c}$, where $\Phi_{c}$ is a small Gaussian random sampling matrix. However, we did not find a significant improvement in performance and so did not use this optimization in our tests for simplicity (see also Section \ref{sec:hifde:3dx}).

\subsection{Optimal Low-Rank Approximation}
Although we have built our algorithms around the ID, it is actually not essential (at least with HIF-DE as presently formulated) and other low-rank approximations can just as well be used. Perhaps the most natural of these is the singular value decomposition (SVD), which is optimal in the sense that it achieves the minimal approximation error for a given rank \cite{golub:1996:johns-hopkins-univ}. Recall that the SVD of a matrix $A \in \mathbb{R}^{m \times n}$ is a factorization of the form $A = U \Sigma V^{\trans}$, where $U \in \mathbb{R}^{m \times m}$ and $V \in \mathbb{R}^{n \times n}$ are orthogonal, and $\Sigma \in \mathbb{R}^{m \times n}$ is diagonal with the singular values of $A$ as its entries. The following is the analogue of Corollary \ref{cor:id-sparse} using the SVD.

\begin{lemma}
 Let $A \in \mathbb{R}^{m \times n}$ with rank $k \leq \min (m, n)$ and SVD
 \begin{align*}
  A = U \Sigma V^{\trans} =
  \begin{bmatrix}
   U_{1} & U_{2}
  \end{bmatrix}
  \begin{bmatrix}
   0\\
   & \Sigma_{2}
  \end{bmatrix}
  \begin{bmatrix}
   V_{1} & V_{2}
  \end{bmatrix}^{\trans},
 \end{align*}
 where $\Sigma_{2} \in \mathbb{R}^{k \times k}$. Then
 \begin{align*}
  U^{\trans} A = \Sigma V^{\trans} =
  \begin{bmatrix}
   0\\
   \Sigma_{2} V_{2}^{\trans}
  \end{bmatrix}, \quad AV = U \Sigma =
  \begin{bmatrix}
   0 & U_{2} \Sigma_{2}
  \end{bmatrix}.
 \end{align*}
\end{lemma}

The analogue of Lemma \ref{lem:skel} is then:

\begin{lemma}
 Let
 \begin{align*}
  A =
  \begin{bmatrix}
   A_{pp} & A_{qp}^{\trans}\\
   A_{qp} & A_{qq}
  \end{bmatrix}
 \end{align*}
 be symmetric for $A_{qp}$ low-rank with SVD
 \begin{align*}
   A_{qp} = A_{p^{\cmp},p} = U_{p} \Sigma_{p} V_{p}^{\trans} =
  \begin{bmatrix}
   U_{p,1} & U_{p,2}
  \end{bmatrix}
  \begin{bmatrix}
   0\\
   & \Sigma_{p,2}
  \end{bmatrix}
  \begin{bmatrix}
   V_{p,1} & V_{p,2}
  \end{bmatrix}^{\trans}.
 \end{align*}
 If $Q_{p} = \diag (V_{p}, I)$, then
 \begin{align}
  Q_{p}^{\trans} A Q_{p} =
  \begin{bmatrix}
   V_{p}^{\trans} A_{pp} V_{p} & \begin{bmatrix} 0\\\Sigma_{p,2} U_{p,2}^{\trans} \end{bmatrix}\\
   \begin{bmatrix} 0 & U_{p,2} \Sigma_{p,2} \end{bmatrix} & A_{qq}
  \end{bmatrix} \equiv
  \begin{bmatrix}
   B_{p_{1},p_{1}} & B_{p_{2},p_{1}}^{\trans}\\
   B_{p_{2},p_{1}} & B_{p_{2},p_{2}} & \Sigma_{p,2} U_{p,2}^{\trans}\\
   & U_{p,2} \Sigma_{p,2} & A_{qq}
  \end{bmatrix}
  \label{eqn:svd-sparse}
 \end{align}
 on conformably partitioning $p = p_{1} \cup p_{2}$, so
 \begin{align*}
  S_{p_{1}}^{\trans} Q_{p}^{\trans} A Q_{p} S_{p_{1}} =
  \begin{bmatrix}
   D_{p_{1}}\\
   & \tilde{B}_{p_{2},p_{2}} & \Sigma_{p,2} U_{p,2}^{\trans}\\
   & U_{p,2} \Sigma_{p,2} & A_{qq}
  \end{bmatrix},
 \end{align*}
 where $S_{p_{1}}$ is the elimination operator of Lemma \ref{lem:sparse-elim} associated with $p_{1}$ and $\tilde{B}_{p_{2},p_{2}} = B_{p_{2},p_{2}} - B_{p_{2},p_{1}} B_{p_{1},p_{1}}^{-1} B_{p_{2},p_{1}}^{\trans}$, assuming that $B_{p_{1},p_{1}}$ is nonsingular.
\end{lemma}

The external interactions $U_{p,2} \Sigma_{p,2}$ with the SVD ``skeletons'' $p_{2}$ are a linear combination of the original external interactions $A_{qp}$ involving all of $p$. Thus, the DOFs $p_{2}$ are, in a sense, delocalized across all points associated with $p$, though they can still be considered to reside on the separators.

The primary advantages of using the SVD over the ID are that (1) it can achieve better compression since a smaller rank may be required for any given precision and (2) the sparsification matrix $Q_{p}$ in \eqref{eqn:svd-sparse} is orthogonal, which provides improved numerical stability, especially when used in a multilevel setting such as \eqref{eqn:hifde2}. However, there are several disadvantages as well, chief among them:
\begin{itemize}
 \item
  the extra computational cost, which typically is about $2$--$3$ times larger;
 \item
  the need to overwrite matrix entries involving the index set $q$ in \eqref{eqn:svd-sparse}, which we remark is still sparse; and
 \item
  the loss of precise geometrical information associated with each DOF.
\end{itemize}
Of these, the last is arguably the most important since it destroys the dimensional reduction interpretation of HIF-DE, which is crucial for achieving estimated $O(N)$ complexity in 3D, as we shall see next.

\subsection{Three-Dimensional Variant with Edge Skeletonization}
\label{sec:hifde:3dx}
In Section \ref{sec:hifde:3d}, we presented a ``basic'' version of HIF-DE in 3D based on interior cell elimination and face skeletonization, which from Figure \ref{fig:hifde3} is seen to retain active DOFs only on the edges of cubic cells. All fronts are hence reduced to 1D, which yields estimated $O(N \log N)$ complexity for the algorithm (Section \ref{sec:hifde:complexity}). Here, we seek to further accelerate this to $O(N)$ by skeletonizing each cell edge and reducing it completely to 0D, as guided by our assumptions on SCIs. However, a complication now arises in that fill-in can occur, which can be explained as follows.

Consider the 3D problem and suppose that both interior cell elimination and face skeletonization have been performed. Then as noted in Section \ref{sec:hifde:accel-comp}, the remaining DOFs with respect to each face will be those on its boundary edges plus a few interior layers near the edges (Figure \ref{fig:fill-in}A) (the depth of these layers depends on the compression tolerance $\epsilon$).
\begin{figure}
 \centering
 \begin{subfigure}{116px}
  \centering
  \includegraphics{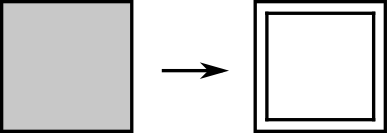}
  \caption{Face skeletonization.}
 \end{subfigure}
 \quad
 \begin{subfigure}{153px}
  \centering
  \includegraphics{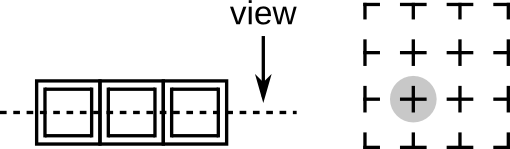}
  \caption{Edge configuration from top view at interior slice.}
 \end{subfigure}
 \caption{Loss of sparsity from edge skeletonization in 3D. Face skeletonization (left) typically leaves several layers of DOFs along the perimeter, which lead to thick edges (right) that connect DOFs across cubic cell boundaries upon skeletonization ($3 \times 3$ grid of cells shown in example with a thick edge outlined in gray).}
 \label{fig:fill-in}
\end{figure}
Therefore, grouping the active DOFs by cell edge gives ``thick'' edges consisting not only of the DOFs on the edges themselves but also those in the interior layers in the four transverse directions surrounding each edge (Figure \ref{fig:fill-in}B). Skeletonizing these thick edges then generates SCIs acting on the skeletons of each edge group by Lemma \ref{lem:skel}, which generally causes DOFs to interact across cubic cell boundaries. The consequence of this is that the next level of interior cell elimination must take into account, in effect, thick separators of width twice the layer depth, which can drastically reduce the number of DOFs eliminated and thus increase the cost. Of course, this penalty does not apply at any level $\ell$ before edge skeletonization has occurred. As a rule of thumb, edge skeletonization should initially be skipped until it reduces the number of active DOFs by a factor of at least the resulting separator width.

For completeness, we now describe HIF-DE in 3D with edge skeletonization following the structure of Section \ref{sec:hifde:3d}, where interior cell elimination (3D to 2D) at level $\ell$ is supplemented with face skeletonization (2D to 1D) at level $\ell + 1/3$ and edge skeletonization (1D to 0D) at level $\ell + 2/3$ for each $\ell = 0, 1, \dots, L - 1$. Figure \ref{fig:hifde3x} shows the active DOFs at each level for a representative example, from which we observe that further compression is clearly achieved on comparing with Figure \ref{fig:hifde3}.
\begin{figure}
 \centering
 \begin{subfigure}{0.2\textwidth}
  \includegraphics[width=\textwidth]{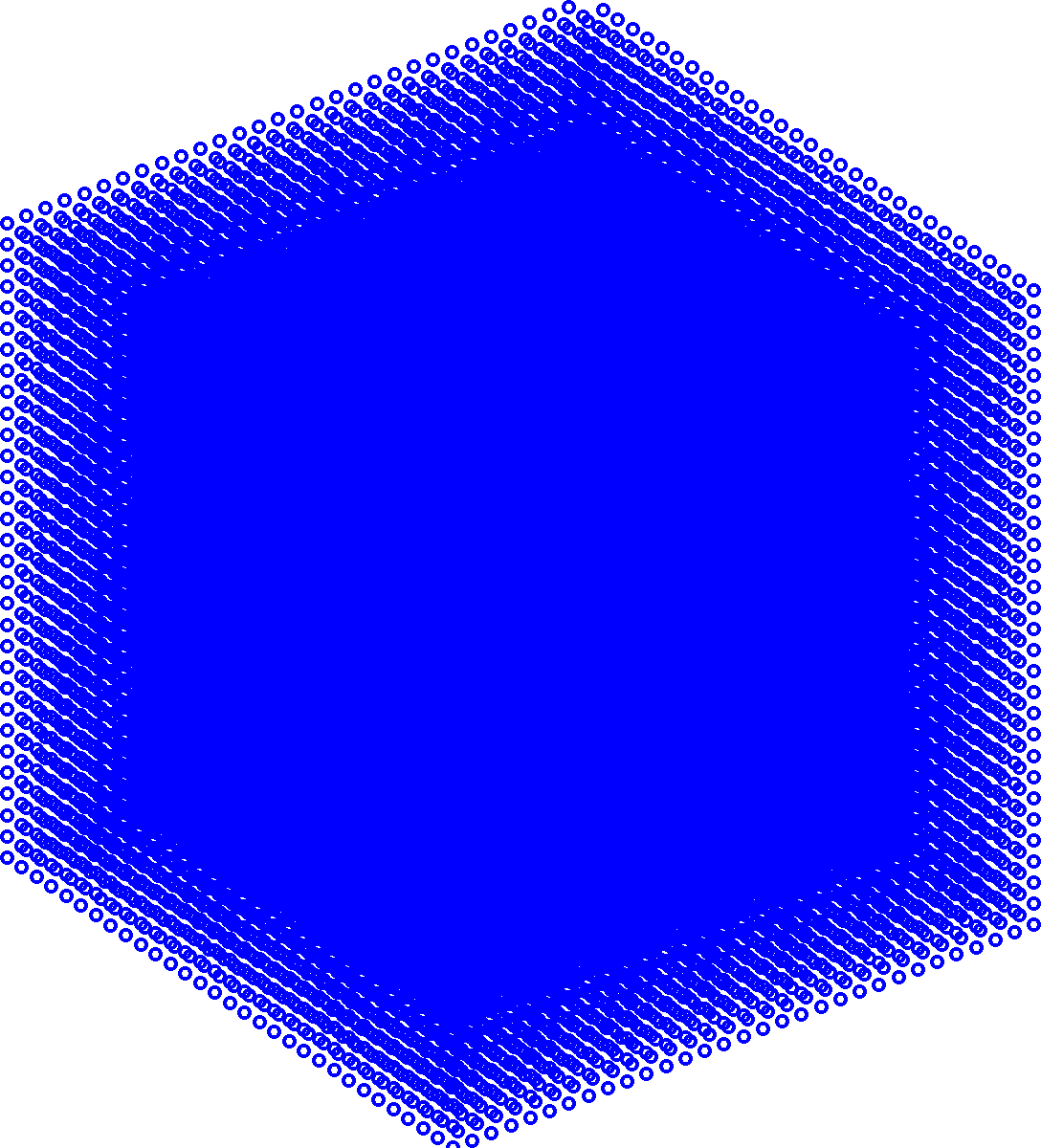}
  \caption*{$\ell = 0$}
 \end{subfigure}
 \quad
 \begin{subfigure}{0.2\textwidth}
  \includegraphics[width=\textwidth]{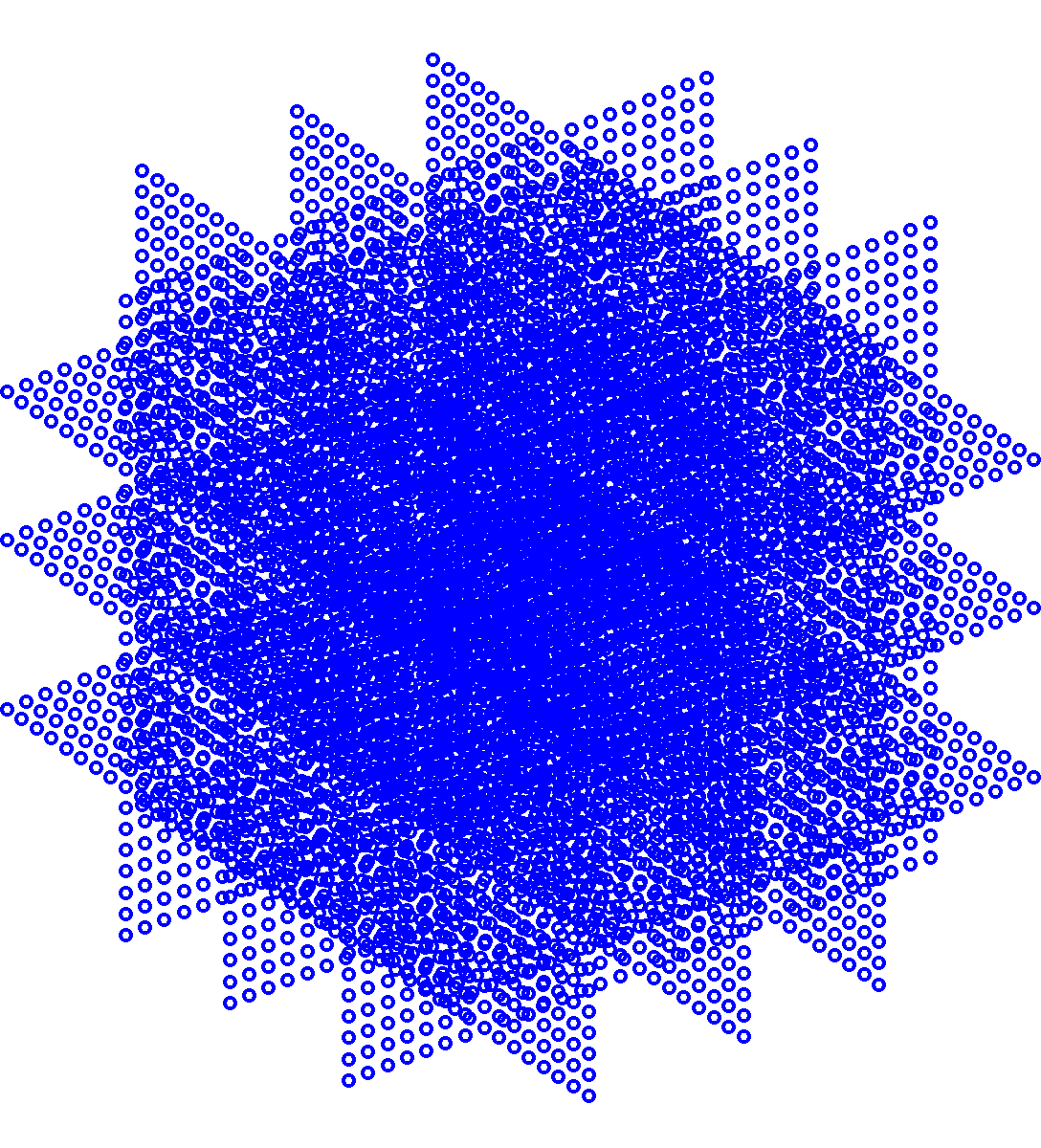}
  \caption*{$\ell = 1/3$}
 \end{subfigure}
 \quad
 \begin{subfigure}{0.2\textwidth}
  \includegraphics[width=\textwidth]{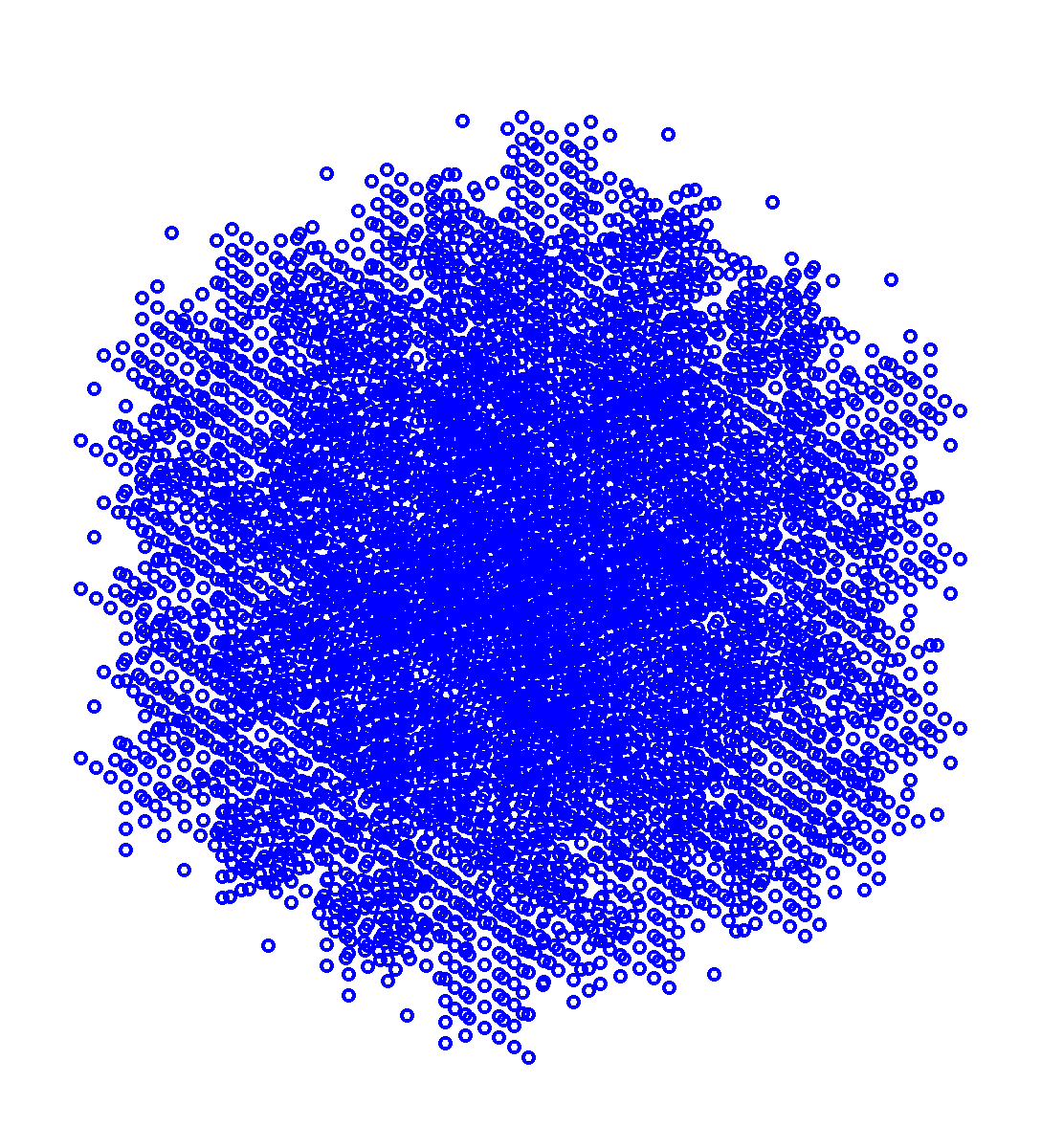}
  \caption*{$\ell = 2/3$}
 \end{subfigure}
 \quad
 \begin{subfigure}{0.2\textwidth}
  \includegraphics[width=\textwidth]{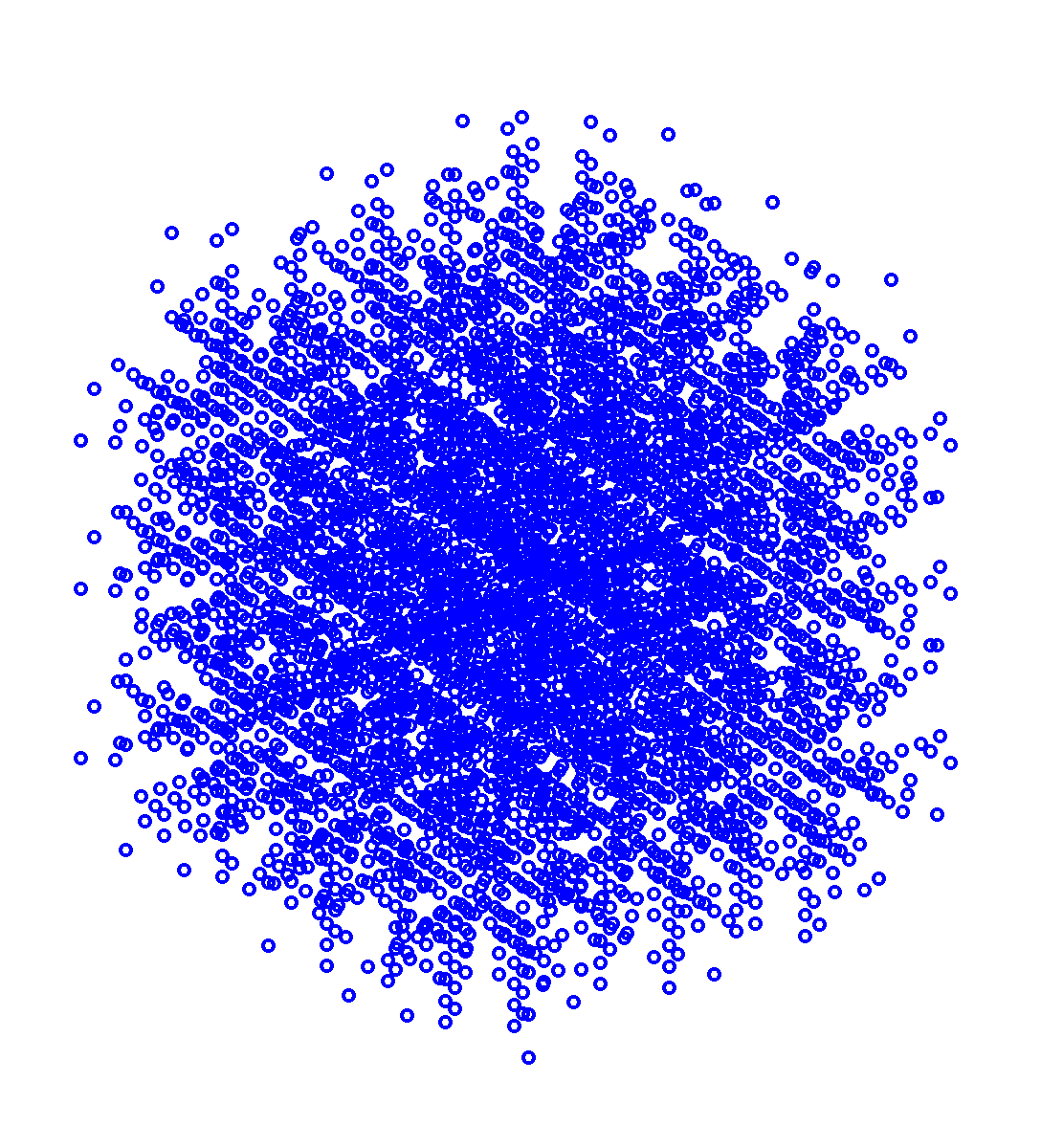}
  \caption*{$\ell = 1$}
 \end{subfigure}
 \\~\\~\\
 \begin{subfigure}{0.2\textwidth}
  \includegraphics[width=\textwidth]{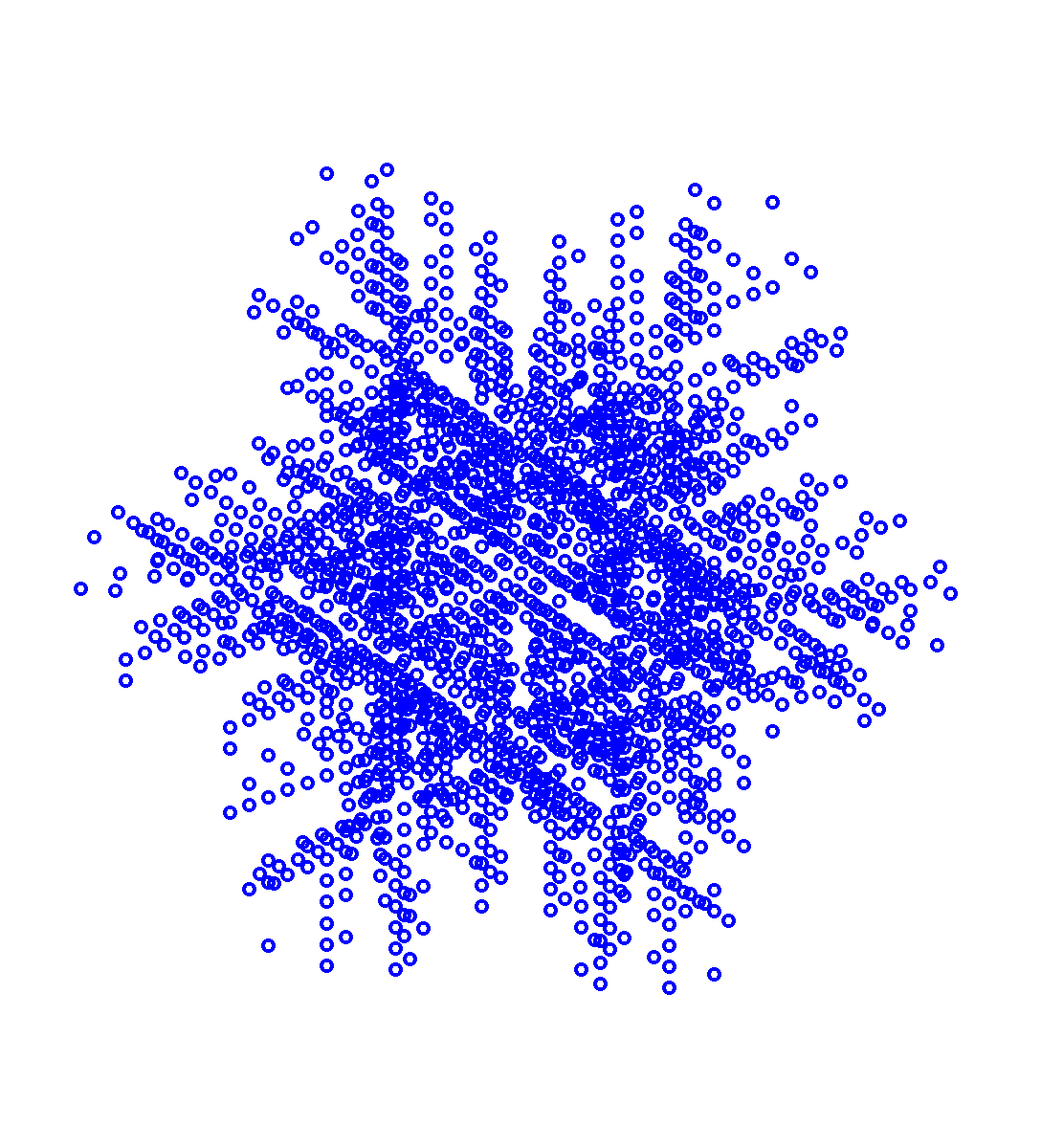}
  \caption*{$\ell = 4/3$}
 \end{subfigure}
 \quad
 \begin{subfigure}{0.2\textwidth}
  \includegraphics[width=\textwidth]{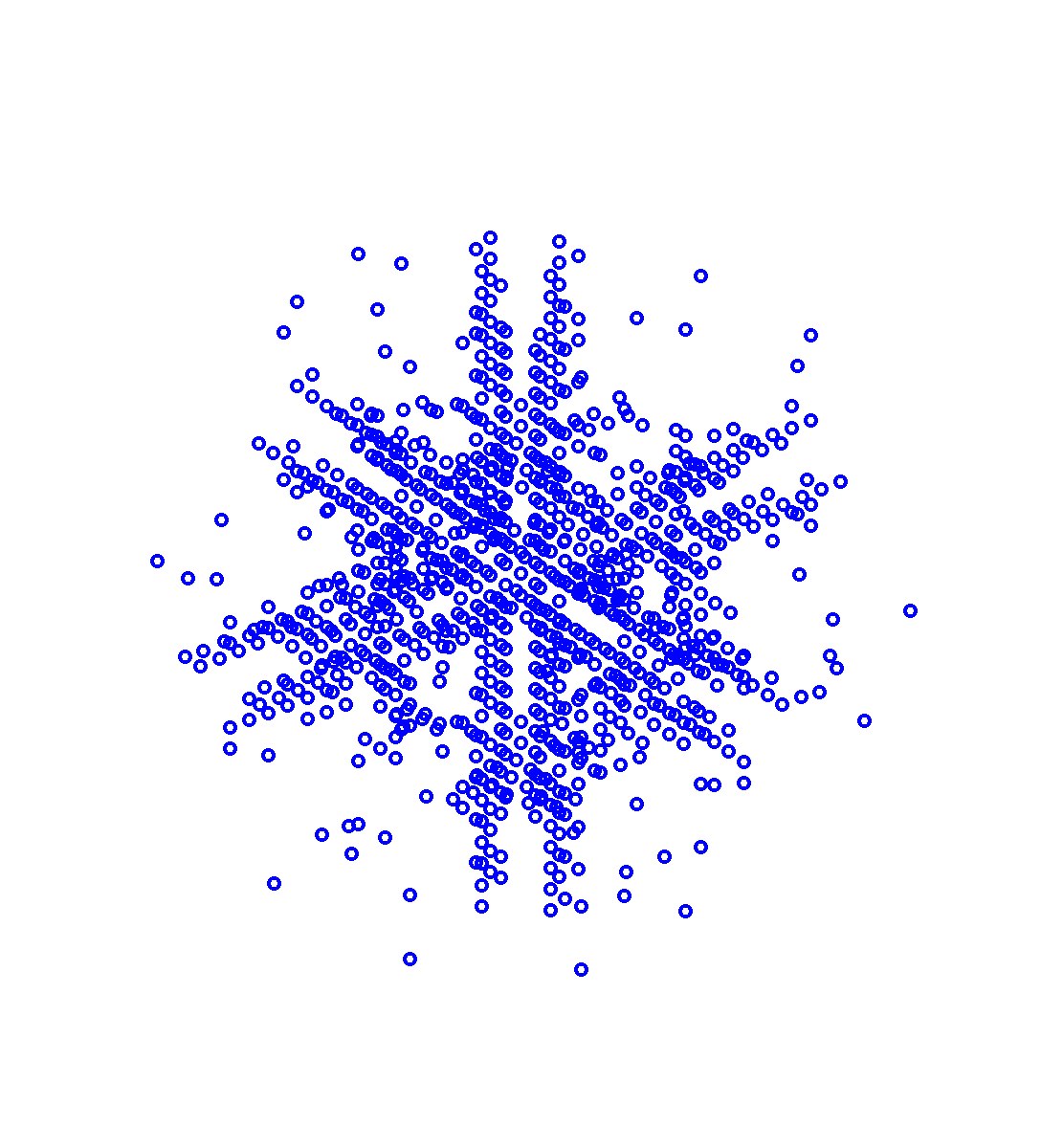}
  \caption*{$\ell = 5/3$}
 \end{subfigure}
 \quad
 \begin{subfigure}{0.2\textwidth}
  \includegraphics[width=\textwidth]{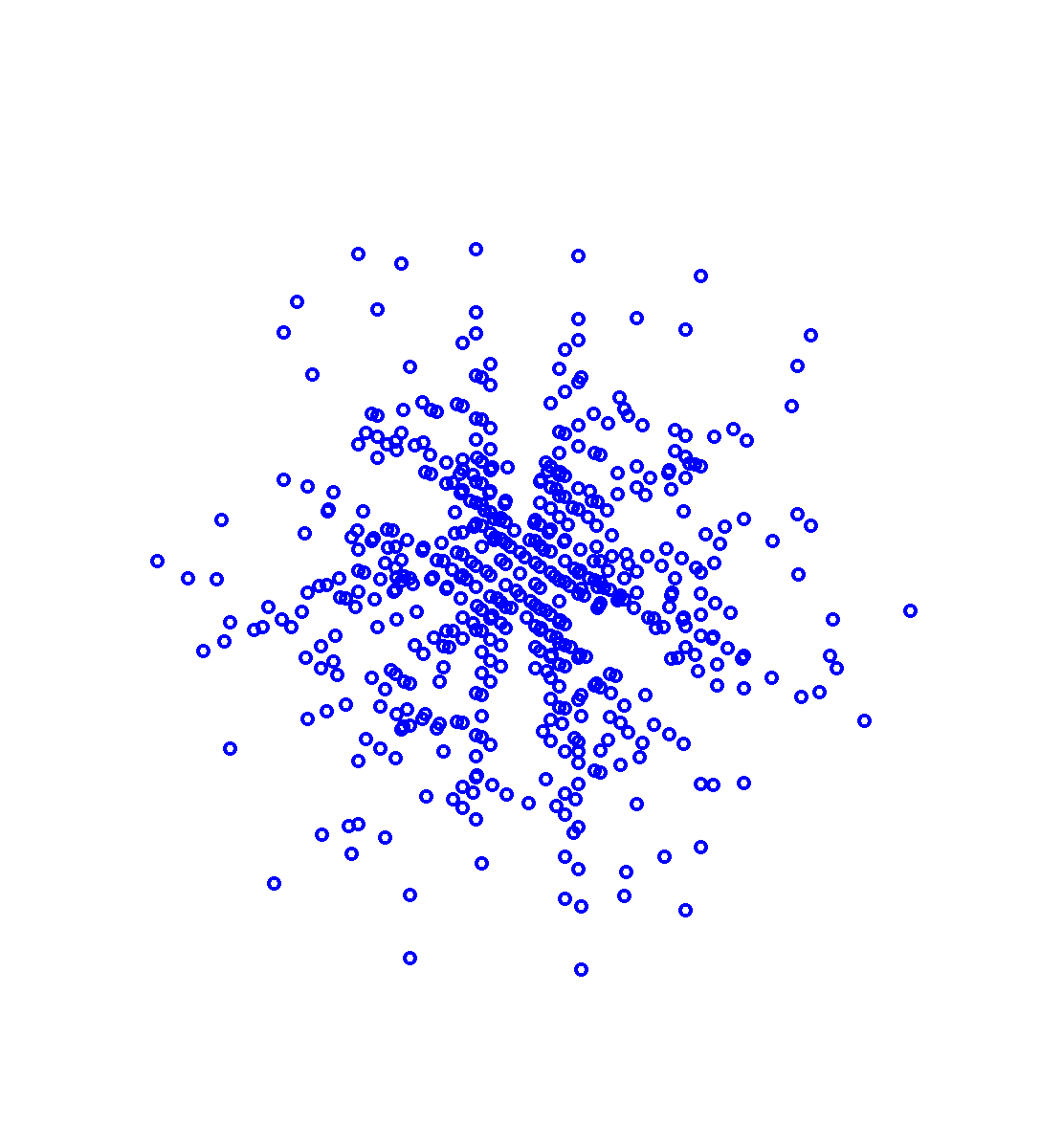}
  \caption*{$\ell = 2$}
 \end{subfigure}
 \caption{Active DOFs at each level $\ell$ of HIF-DE in 3D with edge skeletonization.}
 \label{fig:hifde3x}
\end{figure}

\subsubsection*{Level $\ell$}
Partition $\Omega$ by separators into interior cells. If $\ell = 0$, then these are the same as those in the standard HIF-DE (Section \ref{sec:hifde:3d}), but if $\ell \geq 1$, then this must, in general, be done somewhat more algebraically according to the sparsity pattern of $A_{\ell}$. We propose the following procedure. First, partition all active DOFs into Voronoi cells about the cell centers $2^{\ell} mh(j_{1} - 1/2, j_{2} - 1/2, j_{3} - 1/2)$ for $1 \leq j_{1}, j_{2}, j_{3} \leq 2^{L - \ell}$. This creates an initial geometric partitioning $C_{\ell}$, which we remark is unbuffered (no separators) and so does not satisfy the hypotheses of Section \ref{sec:sparse-elim}. Then for each $c \in C_{\ell}$ in some order:
\begin{enumerate}
 \itemsep 1ex
 \item
  Let
  \begin{align*}
   c^{\ext} = \{ i \in c : (A_{\ell})_{c^{\cmp},i} \neq 0 \}, \quad c^{\cmp} = \left( \bigcup_{c' \in C_{\ell}} c' \right) \setminus c
  \end{align*}
  be the set of indices of the DOFs in $c$ with external interactions.
 \item
  Replace $c$ by $c \setminus c^{\ext}$ in $C_{\ell}$.
\end{enumerate}
On termination, this process produces a collection $C_{\ell}$ of interior cells with minimal separators adaptively constructed. Elimination with respect to $C_{\ell}$ then gives
\begin{align*}
 A_{\ell + 1/3} = W_{\ell}^{\trans} A_{\ell} W_{\ell}, \quad W_{\ell} = \prod_{c \in C_{\ell}} S_{c},
\end{align*}
where the DOFs $\bigcup_{c \in C_{\ell}} c$ have been eliminated.

\subsubsection*{Level $\ell + 1/3$}
Partition $\Omega$ into Voronoi cells about the face centers
\begin{align*}
 &2^{\ell} mh \left( j_{1}, j_{2} - \frac{1}{2}, j_{3} - \frac{1}{2} \right), & &1 \leq j_{1} \leq 2^{L - \ell} - 1, & &1 \leq j_{2}, j_{3} \leq 2^{L - \ell},\\
 &2^{\ell} mh \left( j_{1} - \frac{1}{2}, j_{2}, j_{3} - \frac{1}{2} \right), & &1 \leq j_{2} \leq 2^{L - \ell} - 1, & &1 \leq j_{1}, j_{2} \leq 2^{L - \ell},\\
 &2^{\ell} mh \left( j_{1} - \frac{1}{2}, j_{2} - \frac{1}{2}, j_{3} \right), & &1 \leq j_{3} \leq 2^{L - \ell} - 1, & &1 \leq j_{1}, j_{2} \leq 2^{L - \ell}.
\end{align*}
Let $C_{\ell + 1/3}$ be the collection of index sets corresponding to the active DOFs of each cell. Skeletonization with respect to $C_{\ell + 1/3}$ then gives
\begin{align*}
 A_{\ell + 2/3} = \skel_{C_{\ell + 1/3}} (A_{\ell + 1/3}) \approx U_{\ell + 1/3}^{\trans} A_{\ell + 1/3} U_{\ell + 1/3}, \quad U_{\ell + 1/3} = \prod_{c \in C_{\ell + 1/3}} Q_{c} S_{\rd{c}},
\end{align*}
where the DOFs $\bigcup_{c \in C_{\ell + 1/3}} \rd{c}$ have been eliminated.

\subsubsection*{Level $\ell + 2/3$}
Partition $\Omega$ into Voronoi cells about the edge centers
\begin{align*}
 &2^{\ell} mh \left( j_{1}, j_{2}, j_{3} - \frac{1}{2} \right), & &1 \leq j_{1}, j_{2} \leq 2^{L - \ell} - 1, & &1 \leq j_{3} \leq 2^{L - \ell},\\
 &2^{\ell} mh \left( j_{1}, j_{2} - \frac{1}{2}, j_{3} \right), & &1 \leq j_{1}, j_{3} \leq 2^{L - \ell} - 1, & &1 \leq j_{2} \leq 2^{L - \ell},\\
 &2^{\ell} mh \left( j_{1} - \frac{1}{2}, j_{2}, j_{3} \right), & &1 \leq j_{2}, j_{3} \leq 2^{L - \ell} - 1, & &1 \leq j_{1} \leq 2^{L - \ell}.
\end{align*}
Let $C_{\ell + 2/3}$ be the collection of index sets corresponding to the active DOFs of each cell. Skeletonization with respect to $C_{\ell + 2/3}$ then gives
\begin{align*}
 A_{\ell + 1} = \skel_{C_{\ell + 2/3}} (A_{\ell + 2/3}) \approx U_{\ell + 2/3}^{\trans} A_{\ell + 2/3} U_{\ell + 2/3}, \quad U_{\ell + 2/3} = \prod_{c \in C_{\ell + 2/3}} Q_{c} S_{\rd{c}},
\end{align*}
where the DOFs $\bigcup_{c \in C_{\ell + 2/3}} \rd{c}$ have been eliminated.

\subsubsection*{Level $L$}
Combining the approximation over all levels gives
\begin{align*}
 D \equiv A_{L} \approx V_{L - 1/3}^{\trans} \cdots V_{2/3}^{\trans} V_{1/3}^{\trans} V_{0}^{\trans} A V_{0} V_{1/3} V_{2/3} \cdots V_{L - 1/3},
\end{align*}
where $V_{\ell}$ is as defined in \eqref{eqn:matrix-factor}, so
\begin{subequations}
 \label{eqn:hifde3x}
 \begin{align}
  A &\approx V_{0}^{-\trans} V_{1/3}^{-\trans} V_{2/3}^{-\trans} \cdots V_{L - 1/3}^{-\trans} D V_{L - 1/3}^{-1} \cdots V_{2/3}^{-1} V_{1/3}^{-1} V_{0}^{-1} \equiv F,\\
  A^{-1} &\approx V_{0} V_{1/3} V_{2/3} \cdots V_{L - 1/3} D^{-1} V_{L - 1/3}^{\trans} \cdots V_{2/3}^{\trans} V_{1/3}^{\trans} V_{0}^{\trans} = F^{-1}.
 \end{align}
\end{subequations}
As in Section \ref{sec:hifde:2d}, if $A$ is SPD, then so are $F$ and $F^{-1}$, provided that very mild conditions hold. We summarize the overall scheme as Algorithm \ref{alg:hifde3x}.
\begin{algorithm}
 \caption{HIF-DE in 3D with edge skeletonization.}
 \label{alg:hifde3x}
 \begin{algorithmic}
  \State $A_{0} = A$ \Comment{initialize}
  \For{$\ell = 0, 1, \dots, L - 1$} \Comment{loop from finest to coarsest level}
   \State $A_{\ell + 1/3} = W_{\ell}^{\trans} A_{\ell} W_{\ell}$ \Comment{eliminate interior cells}
   \State $A_{\ell + 2/3} = \skel_{C_{\ell + 1/3}} (A_{\ell + 1/3}) \approx U_{\ell + 1/3}^{\trans} A_{\ell + 1/3} U_{\ell + 1/3}$ \Comment{skeletonize faces}
   \State $A_{\ell + 1} = \skel_{C_{\ell + 2/3}} (A_{\ell + 2/3}) \approx U_{\ell + 2/3}^{\trans} A_{\ell + 2/3} U_{\ell + 2/3}$ \Comment{skeletonize edges}
  \EndFor
  \State $A \approx V_{0}^{-\trans} V_{1/3}^{-\trans} \cdots V_{L - 1/3}^{-\trans} D V_{L - 1/3}^{-1} \cdots V_{1/3}^{-1} V_{0}^{-1}$ \Comment{generalized LDL decomposition}
 \end{algorithmic}
\end{algorithm}

Unlike for the standard HIF-DE, randomized methods (Section \ref{sec:hifde:accel-comp}) now tend to be inaccurate when compressing SCIs. This could be remedied by considering instead $\Phi_{c} (A_{c^{\nbr},c} A_{c^{\nbr},c}^{\trans})^{\gamma} A_{c^{\nbr},c}$ for some small integer $\gamma = 1, 2, \dots$, but the expense of the extra multiplications usually outweighed any efficiency gains.

\subsection{Complexity Estimates}
\label{sec:hifde:complexity}
We now investigate the computational complexity of HIF-DE. For this, we need to estimate the skeleton size $|\sk{c}|$ for a typical index set $c \in C_{\ell}$ at fractional level $\ell$. This is determined by the rank behavior of SCIs, which we assume satisfy standard multipole estimates \cite{greengard:1987:j-comput-phys,greengard:1997:acta-numer} as motivated by experimental observations. Then it can be shown \cite{ho:2012:siam-j-sci-comput,ho:comm-pure-appl-math} that the typical skeleton size is
\begin{align}
 k_{\ell} =
 \begin{cases}
  O(\ell), & \delta = 1\\
  O(2^{(\delta - 1) \ell}), & \delta \geq 2,
 \end{cases}
 \label{eqn:inter-rank}
\end{align}
where $\delta$ is the intrinsic dimension of a typical DOF cluster at level $\ell$, i.e., $\delta = 1$ for edges (2D and 3D) and $\delta = 2$ for faces (3D only). Note that we have suggestively used the same notation as for the index set size $|c|$ in Section \ref{sec:mf:complexity}, which can be justified by recognizing that the active DOFs $c \in C_{\ell}$ for any $\ell$ are obtained by merging skeletons from at most one integer level prior. We emphasize that \eqref{eqn:inter-rank} has yet to be proven, so all following results should formally be understood as conjectures, albeit ones with strong numerical support (Section \ref{sec:results}).

\begin{theorem}
 \label{thm:hifde}
 Assume that \eqref{eqn:inter-rank} holds. Then the costs of constructing the factorization $F$ in \eqref{eqn:hifde2} or \eqref{eqn:hifde3x} using HIF-DE with accelerated compression and of applying $F$ or $F^{-1}$ are, respectively, $t_{f}, t_{a/s} = O(N)$ in 2D; $t_{f} = O(N \log N)$ and $t_{a/s} = O(N)$ in 3D; and $t_{f}, t_{a/s} = O(N)$ in 3D with edge skeletonization.
\end{theorem}

\begin{proof}
 The costs of constructing and applying the factorization are clearly
 \begin{align*}
  t_{f} = \sumprime_{\ell = 0}^{L} O(2^{d(L - \ell)} k_{\ell}^{3}), \quad t_{a/s} = \sumprime_{\ell = 0}^{L} O(2^{d(L - \ell)} k_{\ell}^{2}),
 \end{align*}
 where prime notation denotes summation over all levels, both integer and fractional, and $k_{\ell}$ is as given by \eqref{eqn:inter-rank} for $\delta$ appropriately chosen. In 2D, all fronts are reduced to 1D edges, so $\delta = 1$; in 3D, compression on 2D faces has $\delta = 2$; and in 3D with edge skeletonization, we again have $\delta = 1$. The claim follows by direct computation.
\end{proof}

\section{Numerical Results}
\label{sec:results}
In this section, we demonstrate the efficiency of HIF-DE by reporting numerical results for some benchmark problems in 2D and 3D. All algorithms and examples were implemented in MATLAB and are freely available at \url{https://github.com/klho/FLAM/}. In what follows, we refer to MF as \alg{mf2} in 2D and \alg{mf3} in 3D. Similarly, we call HIF-DE \alg{hifde2} and \alg{hifde3}, respectively, and denote by \alg{hifde3x} the 3D variant with edge skeletonization. All codes are fully adaptive and built on quadtrees in 2D and octrees in 3D. The average block size $|c|$ at level $0$ (and hence the tree depth $L$) was chosen so that roughly half of the initial DOFs are eliminated. In select cases, the first few fractional levels of HIF-DE were skipped to optimize the running time. Diagonal blocks, i.e., $A_{pp}$ in Lemma \ref{lem:sparse-elim}, were factored using the Cholesky decomposition if $A$ is SPD and the (partially pivoted) LDL decomposition otherwise.

For each example, the following are given:
\begin{itemize}
 \item
  $\epsilon$: relative precision of the ID;
 \item
  $N$: total number of DOFs in the problem;
 \item
  $|s_{L}|$: number of active DOFs remaining at the highest level;
 \item
  $t_{f}$: wall clock time for constructing the factorization $F$ in seconds;
 \item
  $m_{f}$: memory required to store $F$ in GB;
 \item
  $t_{a/s}$: wall clock time for applying $F$ or $F^{-1}$ in seconds;
 \item
  $e_{a}$: {\it a posteriori} estimate of $\| A - F \| / \| A \|$ (see below);
 \item
  $e_{s}$: {\it a posteriori} estimate of $\| I - A F^{-1} \| \geq \| A^{-1} - F^{-1} \| / \| A^{-1} \|$;
 \item
  $n_{i}$: number of iterations to solve \eqref{eqn:linear-system} using CG \cite{hestenes:1952:j-res-nat-bur-stand,van-der-vorst:1992:siam-j-sci-stat-comput} (if SPD) or GMRES \cite{saad:1986:siam-j-sci-stat-comput} with preconditioner $F^{-1}$ to a tolerance of $10^{-12}$, where $f$ is a standard uniform random vector.
\end{itemize}
We also compare against MF, which is numerically exact.

The operator errors $e_{a}$ and $e_{s}$ were estimated using power iteration with a standard uniform random start vector \cite{dixon:1983:siam-j-numer-anal,kuczynski:1992:siam-j-matrix-anal-appl} and a convergence criterion of $10^{-2}$ relative precision in the matrix norm. This has a small probability of underestimating the error but seems to be quite robust in practice.

For simplicity, all PDEs were defined over $\Omega = (0, 1)^{d}$ with (arbitrary) Dirichlet boundary conditions as in Section \ref{sec:mf}, discretized on a uniform $n \times n$ or $n \times n \times n$ mesh using second-order central differences via the five-point stencil in 2D and the seven-point stencil in 3D.

All computations were performed in MATLAB R2010b on a single core (without parallelization) of an Intel Xeon E7-4820 CPU at 2.0 GHz on a 64-bit Linux server with 256 GB of RAM.

\subsection{Two Dimensions}
We begin first in 2D, where we present three examples.

\subsubsection*{Example 1}
Consider \eqref{eqn:pde} with $a(x) \equiv 1$, $b(x) \equiv 0$, and $\Omega = (0, 1)^{2}$, i.e., a simple Laplacian in the unit square. The resulting matrix $A$ is SPD, which we factored using both \alg{mf2} and \alg{hifde2} at $\epsilon = 10^{-6}$, $10^{-9}$, and $10^{-12}$ (the compression tolerances are for HIF-DE only). The data are summarized in Tables \ref{tab:fd_square1-f} and \ref{tab:fd_square1-a} with scaling results shown in Figure \ref{fig:fd_square1}.

\begin{table}
 \caption{Factorization results for Example 1.}
 \label{tab:fd_square1-f}
 \scriptsize
 \begin{tabular}{cr|rcc|rcc}
  \toprule
  & & \multicolumn{3}{|c|}{\alg{mf2}} & \multicolumn{3}{|c}{\alg{hifde2}}\\
  $\epsilon$ & \multicolumn{1}{c|}{$N$} & \multicolumn{1}{|c}{$|s_{L}|$} & $t_{f}$ & $m_{f}$ & \multicolumn{1}{|c}{$|s_{L}|$} & $t_{f}$ & $m_{f}$\\
  \midrule
  \multirow{4}{*}{$10^{-06}$} & $1023^{2}$ & \tabdash &        --- &        --- &     $56$ & $5.5$e$+1$ & $7.9$e$-1$\\
                              & $2047^{2}$ & \tabdash &        --- &        --- &     $57$ & $2.4$e$+2$ & $3.2$e$+0$\\
                              & $4095^{2}$ & \tabdash &        --- &        --- &     $57$ & $1.0$e$+3$ & $1.3$e$+1$\\
                              & $8191^{2}$ & \tabdash &        --- &        --- &     $52$ & $4.0$e$+3$ & $5.1$e$+1$\\
  \midrule
  \multirow{4}{*}{$10^{-09}$} & $1023^{2}$ & \tabdash &        --- &        --- &     $85$ & $6.1$e$+1$ & $8.2$e$-1$\\
                              & $2047^{2}$ & \tabdash &        --- &        --- &     $93$ & $2.7$e$+2$ & $3.3$e$+0$\\
                              & $4095^{2}$ & \tabdash &        --- &        --- &     $99$ & $1.1$e$+3$ & $1.3$e$+1$\\
                              & $8191^{2}$ & \tabdash &        --- &        --- &    $102$ & $4.5$e$+3$ & $5.3$e$+1$\\
  \midrule
  \multirow{4}{*}{$10^{-12}$} & $1023^{2}$ & \tabdash &        --- &        --- &    $114$ & $6.7$e$+1$ & $8.4$e$-1$\\
                              & $2047^{2}$ & \tabdash &        --- &        --- &    $125$ & $2.9$e$+2$ & $3.4$e$+0$\\
                              & $4095^{2}$ & \tabdash &        --- &        --- &    $134$ & $1.3$e$+3$ & $1.4$e$+1$\\
                              & $8191^{2}$ & \tabdash &        --- &        --- &    $144$ & $5.1$e$+3$ & $5.5$e$+1$\\
  \midrule
  \multirow{3}{*}{---}        & $1023^{2}$ &   $2045$ & $8.6$e$+1$ & $1.1$e$+0$ & \tabdash &        --- &        ---\\
                              & $2047^{2}$ &   $4093$ & $4.5$e$+2$ & $4.8$e$+0$ & \tabdash &        --- &        ---\\
                              & $4095^{2}$ &   $8189$ & $2.5$e$+3$ & $2.1$e$+1$ & \tabdash &        --- &        ---\\
  \bottomrule
 \end{tabular}
\end{table}

\begin{table}
 \caption{Matrix application results for Example 1.}
 \label{tab:fd_square1-a}
 \scriptsize
 \begin{tabular}{cr|c|cccr}
  \toprule
  & & \multicolumn{1}{|c|}{\alg{mf2}} & \multicolumn{4}{|c}{\alg{hifde2}}\\
  $\epsilon$ & \multicolumn{1}{c|}{$N$} & $t_{a/s}$ & $t_{a/s}$ & $e_{a}$ & $e_{s}$ & \multicolumn{1}{c}{$n_{i}$}\\
  \midrule
  \multirow{4}{*}{$10^{-06}$} & $1023^{2}$ &        --- & $2.1$e$+0$ & $8.3$e$-06$ & $2.4$e$-03$ &  $6$\\
                              & $2047^{2}$ &        --- & $9.1$e$+0$ & $2.1$e$-05$ & $1.5$e$-02$ &  $7$\\
                              & $4095^{2}$ &        --- & $3.9$e$+1$ & $8.4$e$-05$ & $2.6$e$-01$ & $13$\\
                              & $8191^{2}$ &        --- & $1.8$e$+2$ & $1.1$e$-04$ & $6.3$e$-01$ & $15$\\
  \midrule
  \multirow{4}{*}{$10^{-09}$} & $1023^{2}$ &        --- & $2.1$e$+0$ & $5.5$e$-09$ & $8.7$e$-07$ &  $4$\\
                              & $2047^{2}$ &        --- & $8.6$e$+0$ & $1.4$e$-08$ & $6.5$e$-06$ &  $4$\\
                              & $4095^{2}$ &        --- & $3.8$e$+1$ & $2.9$e$-08$ & $2.0$e$-05$ &  $4$\\
                              & $8191^{2}$ &        --- & $1.9$e$+2$ & $5.3$e$-08$ & $1.0$e$-04$ &  $3$\\
  \midrule
  \multirow{4}{*}{$10^{-12}$} & $1023^{2}$ &        --- & $2.1$e$+0$ & $5.5$e$-12$ & $8.0$e$-10$ &  $3$\\
                              & $2047^{2}$ &        --- & $8.9$e$+0$ & $9.4$e$-12$ & $3.0$e$-09$ &  $3$\\
                              & $4095^{2}$ &        --- & $4.2$e$+1$ & $3.2$e$-11$ & $2.2$e$-08$ &  $3$\\
                              & $8191^{2}$ &        --- & $1.8$e$+2$ & $6.3$e$-11$ & $6.5$e$-08$ &  $3$\\
  \midrule
  \multirow{3}{*}{---}        & $1023^{2}$ & $2.4$e$+0$ &        --- &         --- &         --- &  ---\\
                              & $2047^{2}$ & $1.0$e$+1$ &        --- &         --- &         --- &  ---\\
                              & $4095^{2}$ & $4.4$e$+1$ &        --- &         --- &         --- &  ---\\
  \bottomrule
 \end{tabular}
\end{table}

\begin{figure}
 \includegraphics{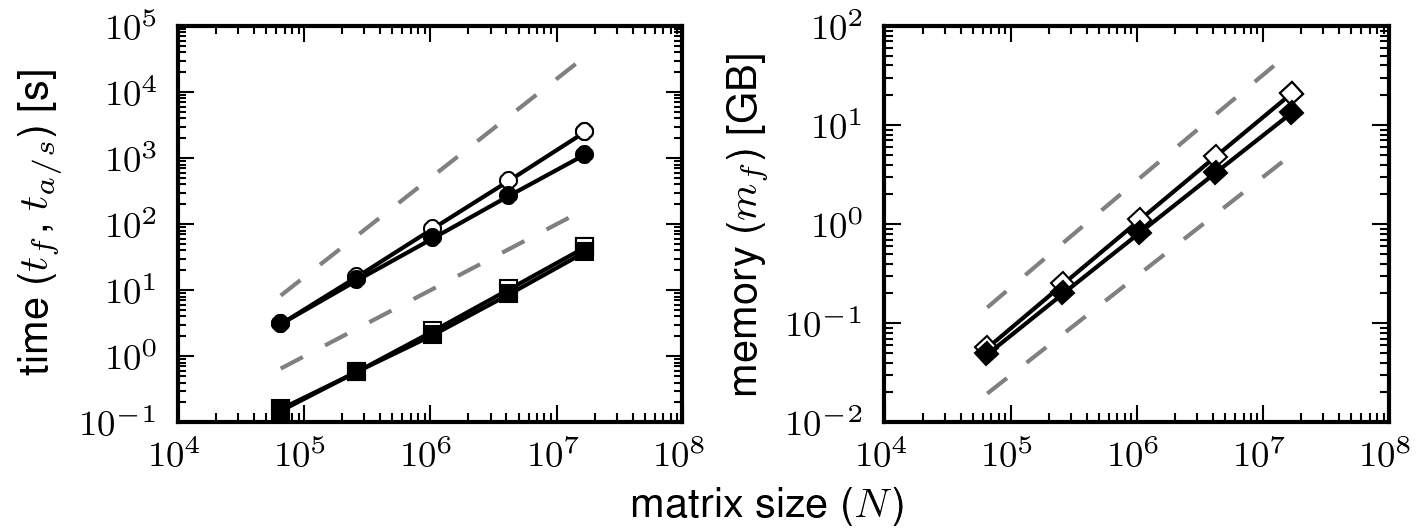}
 \caption{Scaling results for Example 1. Wall clock times $t_{f}$ ($\circ$) and $t_{a/s}$ ($\Box$) and storage requirements $m_{f}$ ($\diamond$) are shown for \alg{mf2} (white) and \alg{hifde2} (black) at precision $\epsilon = 10^{-9}$. Dotted lines denote extrapolated values. Included also are reference scalings (gray dashed lines) of $O(N)$ and $O(N^{3/2})$ (left, from bottom to top), and $O(N)$ and $O(N \log N)$ (right). The lines for $t_{a/s}$ (bottom left) lie nearly on top of each other.}
 \label{fig:fd_square1}
\end{figure}

It is evident that $|s_{L}| \sim k_{L}$ behaves as predicted, with HIF-DE achieving significant compression over MF (but, of course, at the cost of introducing approximation error). Consequently, we find strong support for asymptotic complexities consistent with Theorems \ref{thm:mf} and \ref{thm:hifde}, though MF scales much better than predicted due to its favorable constants. We remark that obtaining a speedup in 2D is not our primary goal since MF is already so efficient in this regime. Still, we see a modest increase in performance and memory savings that allow us to run HIF-DE up to $N = 8191^{2}$, for which MF was not successful.

For all problem sizes tested, $t_{f}$ and $m_{f}$ are always smaller for HIF-DE, though $t_{a/s}$ is quite comparable. This is because $t_{a/s}$ is dominated by memory access (at least in our current implementation), which also explains its relative insensitivity to $\epsilon$. Furthermore, we observe that $t_{a/s} \ll t_{f}$ for both methods, which makes them ideally suited to systems involving multiple right-hand sides.

The forward approximation error $e_{a} = O(\epsilon)$ for all $N$ and seems to increase only mildly with $N$. This indicates that the local accuracy of the ID provides a good estimate of the overall accuracy of the algorithm, which is not easy to prove since the multilevel matrix factors constituting $F$ are not orthogonal. On the other hand, we expect the inverse approximation error to scale as $e_{s} = O(\kappa (A) e_{a})$, where $\kappa (A) = O(N)$ for this example, and indeed we see that $e_{s}$ is much larger due to ill-conditioning. When using $F^{-1}$ to precondition CG, however, the number of iterations required is always very small. This indicates that $F^{-1}$ is a highly effective preconditioner.

\subsubsection*{Example 2}
Consider now the same setup as in Example 1 but with $a(x)$ a quantized high-contrast random field defined as follows:
\begin{enumerate}
 \itemsep 1ex
 \item
  Initialize by sampling each staggered grid point $a_{j}$ from the standard uniform distribution.
 \item
  Impose some correlation structure by convolving $\{ a_{j} \}$ with an isotropic Gaussian of width $4h$.
 \item
  Quantize by setting
  \begin{align*}
   a_{j} =
   \begin{cases}
    10^{-2}, & a_{j} \leq \mu\\
    10^{+2}, & a_{j} > \mu,
   \end{cases}
  \end{align*}
  where $\mu$ is the median of $\{ a_{j} \}$.
\end{enumerate}
Figure \ref{fig:contrast-field} shows a sample realization of such a high-contrast random field in 2D.
\begin{figure}
 \includegraphics[width=1.5in]{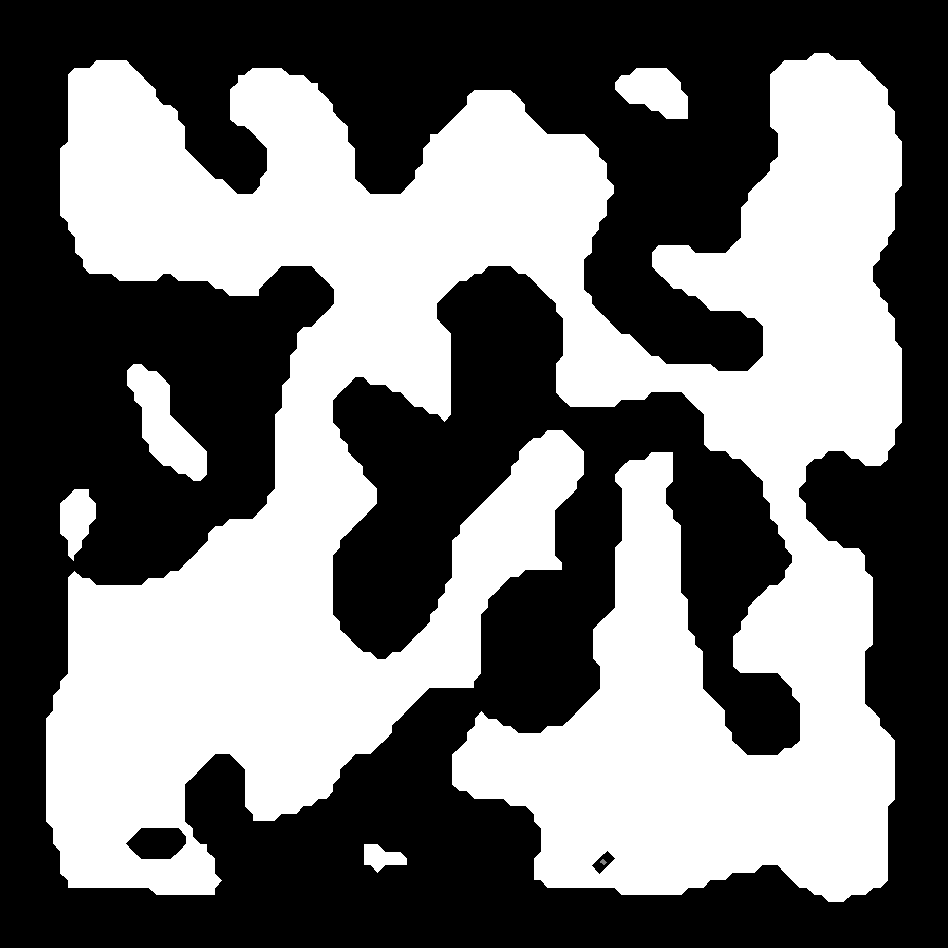}
 \caption{Sample realization of a quantized high-contrast random field in 2D.}
 \label{fig:contrast-field}
\end{figure}
The matrix $A$ now has condition number $\kappa (A) = O(\rho N)$, where $\rho = 10^{4}$ is the contrast ratio. Such high-contrast problems are typically extremely difficult to solve by iteration. Data for \alg{mf2} and \alg{hifde2} at $\epsilon = 10^{-9}$ and $10^{-12}$ are given in Tables \ref{tab:fd_square2-f} and \ref{tab:fd_square2-a}.

\begin{table}
 \caption{Factorization results for Example 2.}
 \label{tab:fd_square2-f}
 \scriptsize
 \begin{tabular}{cr|rcc|rcc}
  \toprule
  & & \multicolumn{3}{|c|}{\alg{mf2}} & \multicolumn{3}{|c}{\alg{hifde2}}\\
  $\epsilon$ & \multicolumn{1}{c|}{$N$} & \multicolumn{1}{|c}{$|s_{L}|$} & $t_{f}$ & $m_{f}$ & \multicolumn{1}{|c}{$|s_{L}|$} & $t_{f}$ & $m_{f}$\\
  \midrule
  \multirow{4}{*}{$10^{-09}$} & $1023^{2}$ & \tabdash &        --- &        --- &     $97$ & $6.5$e$+1$ & $8.3$e$-1$\\
                              & $2047^{2}$ & \tabdash &        --- &        --- &    $110$ & $2.8$e$+2$ & $3.3$e$+0$\\
                              & $4095^{2}$ & \tabdash &        --- &        --- &    $113$ & $1.2$e$+3$ & $1.3$e$+1$\\
                              & $8191^{2}$ & \tabdash &        --- &        --- &    $141$ & $4.6$e$+3$ & $5.4$e$+1$\\
  \midrule
  \multirow{4}{*}{$10^{-12}$} & $1023^{2}$ & \tabdash &        --- &        --- &    $134$ & $7.4$e$+1$ & $8.7$e$-1$\\
                              & $2047^{2}$ & \tabdash &        --- &        --- &    $148$ & $3.2$e$+2$ & $3.5$e$+0$\\
                              & $4095^{2}$ & \tabdash &        --- &        --- &    $160$ & $1.4$e$+3$ & $1.4$e$+1$\\
                              & $8191^{2}$ & \tabdash &        --- &        --- &    $191$ & $5.5$e$+3$ & $5.7$e$+1$\\
  \midrule
  \multirow{3}{*}{---}        & $1023^{2}$ &   $2045$ & $8.4$e$+1$ & $1.1$e$+0$ & \tabdash &        --- &        ---\\
                              & $2047^{2}$ &   $4093$ & $4.6$e$+2$ & $4.8$e$+0$ & \tabdash &        --- &        ---\\
                              & $4095^{2}$ &   $8189$ & $2.5$e$+3$ & $2.1$e$+1$ & \tabdash &        --- &        ---\\
  \bottomrule
 \end{tabular}
\end{table}

\begin{table}
 \caption{Matrix application results for Example 2.}
 \label{tab:fd_square2-a}
 \scriptsize
 \begin{tabular}{cr|c|cccr}
  \toprule
  & & \multicolumn{1}{|c|}{\alg{mf2}} & \multicolumn{4}{|c}{\alg{hifde2}}\\
  $\epsilon$ & \multicolumn{1}{c|}{$N$} & $t_{a/s}$ & $t_{a/s}$ & $e_{a}$ & $e_{s}$ & \multicolumn{1}{c}{$n_{i}$}\\
  \midrule
  \multirow{4}{*}{$10^{-09}$} & $1023^{2}$ &        --- & $2.3$e$+0$ & $3.1$e$-09$ & $2.0$e$-4$ &  $3$\\
                              & $2047^{2}$ &        --- & $9.3$e$+0$ & $2.5$e$-09$ & $2.4$e$-4$ &  $3$\\
                              & $4095^{2}$ &        --- & $3.9$e$+1$ & $3.4$e$-08$ & $3.1$e$-4$ &  $8$\\
                              & $8191^{2}$ &        --- & $1.9$e$+2$ & $4.5$e$-09$ & $1.2$e$-3$ &  $4$\\
  \midrule
  \multirow{4}{*}{$10^{-12}$} & $1023^{2}$ &        --- & $2.3$e$+0$ & $1.8$e$-12$ & $1.7$e$-7$ &  $2$\\
                              & $2047^{2}$ &        --- & $8.9$e$+0$ & $2.3$e$-12$ & $3.0$e$-7$ &  $2$\\
                              & $4095^{2}$ &        --- & $4.0$e$+1$ & $3.5$e$-12$ & $5.8$e$-7$ &  $2$\\
                              & $8191^{2}$ &        --- & $1.9$e$+2$ & $4.5$e$-12$ & $6.0$e$-7$ &  $2$\\
  \midrule
  \multirow{3}{*}{---}        & $1023^{2}$ & $2.5$e$+0$ &        --- &         --- &        --- &  ---\\
                              & $2047^{2}$ & $1.0$e$+1$ &        --- &         --- &        --- &  ---\\
                              & $4095^{2}$ & $4.1$e$+1$ &        --- &         --- &        --- &  ---\\
  \bottomrule
 \end{tabular}
\end{table}

As expected, factorization results for MF are essentially the same as those in Example 1 since the elimination procedure is identical. Results are also very similar for HIF-DE, with only slightly increased skeleton sizes, presumably to resolve the more detailed structure of $a(x)$. Thus, high-contrast problems do not appear to pose any challenge. However, $e_{s}$ naturally suffers due to the additional ill-conditioning, though $F^{-1}$ remains a very good preconditioner for CG in all cases tested.

\subsubsection*{Example 3}
We then turn to the Helmholtz equation \eqref{eqn:pde} with $a(x) \equiv 1$ and $b(x) \equiv -k^{2}$, where $k = 2 \pi \kappa$ is the wave frequency for $\kappa$ the number of wavelengths in $\Omega$. We kept a fixed number of $32$ DOFs per wavelength by increasing $k$ with $n = \sqrt{N}$. The resulting matrix is {\em indefinite} and was factored using both \alg{mf2} and \alg{hifde2} with $\kappa = 32$, $64$, and $128$ at $\epsilon = 10^{-6}$, $10^{-9}$, and $10^{-12}$. Since $A$ is no longer SPD, $F^{-1}$ now applies as a preconditioner for GMRES. The data are summarized in Tables \ref{tab:fd_square3-f} and \ref{tab:fd_square3-a} with scaling results in Figure \ref{fig:fd_square3}.

\begin{table}
 \caption{Factorization results for Example 3.}
 \label{tab:fd_square3-f}
 \scriptsize
 \begin{tabular}{crr|rcc|rcc}
  \toprule
  & & & \multicolumn{3}{|c|}{\alg{mf2}} & \multicolumn{3}{|c}{\alg{hifde2}}\\
  $\epsilon$ & \multicolumn{1}{c}{$N$} & \multicolumn{1}{c|}{$\kappa$} & \multicolumn{1}{|c}{$|s_{L}|$} & $t_{f}$ & $m_{f}$ & \multicolumn{1}{|c}{$|s_{L}|$} & $t_{f}$ & $m_{f}$\\
  \midrule
  \multirow{3}{*}{$10^{-06}$} & $1023^{2}$ &  $32$ & \tabdash &        --- &        --- &    $156$ & $5.7$e$+1$ & $1.2$e$+0$\\
                              & $2047^{2}$ &  $64$ & \tabdash &        --- &        --- &    $271$ & $2.4$e$+2$ & $4.8$e$+0$\\
                              & $4095^{2}$ & $128$ & \tabdash &        --- &        --- &    $408$ & $1.0$e$+3$ & $1.9$e$+1$\\
  \midrule
  \multirow{3}{*}{$10^{-09}$} & $1023^{2}$ &  $32$ & \tabdash &        --- &        --- &    $180$ & $6.2$e$+1$ & $1.2$e$+0$\\
                              & $2047^{2}$ &  $64$ & \tabdash &        --- &        --- &    $286$ & $2.7$e$+2$ & $4.9$e$+0$\\
                              & $4095^{2}$ & $128$ & \tabdash &        --- &        --- &    $442$ & $1.2$e$+3$ & $2.0$e$+1$\\
  \midrule
  \multirow{3}{*}{$10^{-12}$} & $1023^{2}$ &  $32$ & \tabdash &        --- &        --- &    $207$ & $6.9$e$+1$ & $1.3$e$+0$\\
                              & $2047^{2}$ &  $64$ & \tabdash &        --- &        --- &    $310$ & $3.0$e$+2$ & $5.1$e$+0$\\
                              & $4095^{2}$ & $128$ & \tabdash &        --- &        --- &    $482$ & $1.3$e$+3$ & $2.0$e$+1$\\
  \midrule
  \multirow{3}{*}{---}        & $1023^{2}$ &  $32$ &   $2045$ & $1.1$e$+2$ & $1.6$e$+0$ & \tabdash &        --- &        ---\\
                              & $2047^{2}$ &  $64$ &   $4093$ & $7.2$e$+2$ & $7.1$e$+0$ & \tabdash &        --- &        ---\\
                              & $4095^{2}$ & $128$ &   $8189$ & $4.9$e$+3$ & $3.0$e$+1$ & \tabdash &        --- &        ---\\
  \bottomrule
 \end{tabular}
\end{table}

\begin{table}
 \caption{Matrix application results for Example 3.}
 \label{tab:fd_square3-a}
 \scriptsize
 \begin{tabular}{crr|c|cccr}
  \toprule
  & & & \multicolumn{1}{|c|}{\alg{mf2}} & \multicolumn{4}{|c}{\alg{hifde2}}\\
  $\epsilon$ & \multicolumn{1}{c}{$N$} & \multicolumn{1}{c|}{$\kappa$} & $t_{a/s}$ & $t_{a/s}$ & $e_{a}$ & $e_{s}$ & \multicolumn{1}{c}{$n_{i}$}\\
  \midrule
  \multirow{3}{*}{$10^{-06}$} & $1023^{2}$ &  $32$ &        --- & $2.4$e$+0$ & $5.5$e$-06$ & $3.1$e$-3$ &  $4$\\
                              & $2047^{2}$ &  $64$ &        --- & $9.7$e$+0$ & $6.8$e$-06$ & $7.4$e$-3$ &  $7$\\
                              & $4095^{2}$ & $128$ &        --- & $4.0$e$+1$ & $3.7$e$-05$ & $2.3$e$-2$ & $10$\\
  \midrule
  \multirow{3}{*}{$10^{-09}$} & $1023^{2}$ &  $32$ &        --- & $2.3$e$+0$ & $3.8$e$-09$ & $2.7$e$-6$ &  $2$\\
                              & $2047^{2}$ &  $64$ &        --- & $9.5$e$+0$ & $5.9$e$-09$ & $2.9$e$-5$ &  $6$\\
                              & $4095^{2}$ & $128$ &        --- & $3.7$e$+1$ & $3.5$e$-08$ & $8.5$e$-6$ &  $6$\\
  \midrule
  \multirow{3}{*}{$10^{-12}$} & $1023^{2}$ &  $32$ &        --- & $2.2$e$+0$ & $4.8$e$-12$ & $3.9$e$-9$ &  $2$\\
                              & $2047^{2}$ &  $64$ &        --- & $9.6$e$+0$ & $5.7$e$-12$ & $1.1$e$-8$ &  $2$\\
                              & $4095^{2}$ & $128$ &        --- & $4.3$e$+1$ & $6.8$e$-11$ & $1.7$e$-8$ &  $2$\\
  \midrule
  \multirow{3}{*}{---}        & $1023^{2}$ &  $32$ & $2.6$e$+0$ &        --- &         --- &        --- &  ---\\
                              & $2047^{2}$ &  $64$ & $1.1$e$+1$ &        --- &         --- &        --- &  ---\\
                              & $4095^{2}$ & $128$ & $4.6$e$+1$ &        --- &         --- &        --- &  ---\\
  \bottomrule
 \end{tabular}
\end{table}

\begin{figure}
 \includegraphics{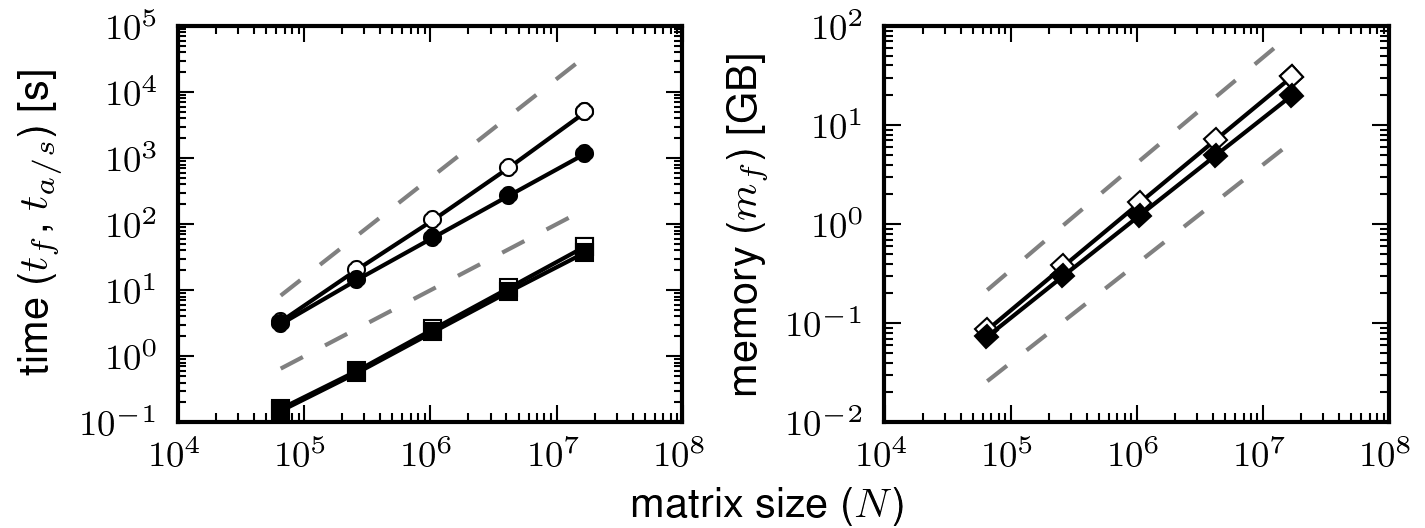}
 \caption{Scaling results for Example 3, comparing \alg{mf2} (white) with \alg{hifde2} (black) at precision $\epsilon = 10^{-9}$; all other notation as in Figure \ref{fig:fd_square1}.}
 \label{fig:fd_square3}
\end{figure}

Overall, the results are very similar to those in Example 1 but with larger skeleton sizes and some extra ill-conditioning of order $O(k)$. We remark, however, that HIF-DE is effective only at low to moderate frequency since the rank structures employed break down as $k \to \infty$. This can be understood by analogy with the Helmholtz Green's function, whose off-diagonal blocks are full-rank in the limit (though other rank structures are possible \cite{engquist:2009:commun-math-sci,engquist:2007:siam-j-sci-comput}). Indeed, we can already see an increasing trend in $|s_{L}|$ beyond that observed in Examples 1 and 2. In the high-frequency regime, the only compression available is due to sparsity, with HIF-DE essentially reducing to MF. Nonetheless, our results reveal no significant apparent failure and demonstrate that HIF-DE achieves linear complexity up to at least $\kappa \sim 10^{2}$.

\subsection{Three Dimensions}
We next present three examples in 3D generalizing each of the 2D cases above.

\subsubsection*{Example 4}
Consider the 3D analogue of Example 1, i.e., \eqref{eqn:pde} with $a(x) \equiv 1$, $b(x) \equiv 0$, and $\Omega = (0, 1)^{3}$. Data for \alg{mf3}, \alg{hifde3}, and \alg{hifde3x} at $\epsilon = 10^{-3}$, $10^{-6}$, and $10^{-9}$ are given in Tables \ref{tab:fd_cube1-f} and \ref{tab:fd_cube1-a} with scaling results shown in Figure \ref{fig:fd_cube1}.

\begin{table}
 \caption{Factorization results for Example 4.}
 \label{tab:fd_cube1-f}
 \scriptsize
 \begin{tabular}{cr|rcc|rcc|rcc}
  \toprule
  & & \multicolumn{3}{|c|}{\alg{mf3}} & \multicolumn{3}{|c|}{\alg{hifde3}} & \multicolumn{3}{|c}{\alg{hifde3x}}\\
  $\epsilon$ & \multicolumn{1}{c|}{$N$} & \multicolumn{1}{|c}{$|s_{L}|$} & $t_{f}$ & $m_{f}$ & \multicolumn{1}{|c}{$|s_{L}|$} & $t_{f}$ & $m_{f}$ & \multicolumn{1}{|c}{$|s_{L}|$} & $t_{f}$ & $m_{f}$\\
  \midrule
  \multirow{3}{*}{$10^{-3}$} &  $31^{3}$ & \tabdash &        --- &        --- &    $950$ & $1.0$e$+1$ & $1.1$e$-1$ &    $331$ & $1.0$e$+1$ & $9.4$e$-2$\\
                             &  $63^{3}$ & \tabdash &        --- &        --- &   $2019$ & $1.9$e$+2$ & $1.2$e$+0$ &    $578$ & $1.7$e$+2$ & $9.6$e$-1$\\
                             & $127^{3}$ & \tabdash &        --- &        --- &   $4153$ & $2.8$e$+3$ & $1.3$e$+1$ &    $890$ & $2.2$e$+3$ & $9.0$e$+0$\\
  \midrule
  \multirow{3}{*}{$10^{-6}$} &  $31^{3}$ & \tabdash &        --- &        --- &   $1568$ & $1.1$e$+1$ & $1.2$e$-1$ &    $931$ & $1.1$e$+1$ & $1.0$e$-1$\\
                             &  $63^{3}$ & \tabdash &        --- &        --- &   $3607$ & $3.0$e$+2$ & $1.7$e$+0$ &   $2466$ & $3.2$e$+2$ & $1.3$e$+0$\\
                             & $127^{3}$ & \tabdash &        --- &        --- &   $7651$ & $6.2$e$+3$ & $2.0$e$+1$ &   $3562$ & $6.2$e$+3$ & $1.6$e$+1$\\
  \midrule
  \multirow{3}{*}{$10^{-9}$} &  $31^{3}$ & \tabdash &        --- &        --- &   $2030$ & $1.3$e$+1$ & $1.3$e$-1$ &   $1495$ & $1.3$e$+1$ & $1.1$e$-1$\\
                             &  $63^{3}$ & \tabdash &        --- &        --- &   $5013$ & $4.3$e$+2$ & $2.0$e$+0$ &   $4295$ & $4.7$e$+2$ & $1.6$e$+0$\\
                             & $127^{3}$ & \tabdash &        --- &        --- &  $11037$ & $1.1$e$+4$ & $2.6$e$+1$ &   $7288$ & $1.1$e$+4$ & $2.1$e$+1$\\
  \midrule
  \multirow{2}{*}{---}       &  $31^{3}$ &   $2791$ & $1.6$e$+1$ & $1.6$e$-1$ & \tabdash &        --- &        --- & \tabdash &        --- &        ---\\
                             &  $63^{3}$ &  $11719$ & $8.2$e$+2$ & $3.0$e$+0$ & \tabdash &        --- &        --- & \tabdash &        --- &        ---\\
  \bottomrule
 \end{tabular}
\end{table}

\begin{table}
 \caption{Matrix application results for Example 4.}
 \label{tab:fd_cube1-a}
 \scriptsize
 \begin{tabular}{cr|c|cccr|cccr}
  \toprule
  & & \multicolumn{1}{|c|}{\alg{mf3}} & \multicolumn{4}{|c|}{\alg{hifde3}} & \multicolumn{4}{|c}{\alg{hifde3x}}\\
  $\epsilon$ & \multicolumn{1}{c|}{$N$} & $t_{a/s}$ & $t_{a/s}$ & $e_{a}$ & $e_{s}$ & \multicolumn{1}{c|}{$n_{i}$} & $t_{a/s}$ & $e_{a}$ & $e_{s}$ & \multicolumn{1}{c}{$n_{i}$}\\
  \midrule
  \multirow{3}{*}{$10^{-3}$} &  $31^{3}$ &        --- & $1.8$e$-1$ & $2.1$e$-03$ & $5.6$e$-2$ &  $7$ & $1.5$e$-1$ & $3.6$e$-03$ & $7.0$e$-2$ &  $8$\\
                             &  $63^{3}$ &        --- & $1.8$e$+0$ & $5.0$e$-03$ & $3.4$e$-1$ & $11$ & $1.3$e$+0$ & $4.3$e$-03$ & $3.3$e$-1$ & $11$\\
                             & $127^{3}$ &        --- & $1.9$e$+1$ & $7.8$e$-03$ & $7.5$e$-1$ & $19$ & $1.2$e$+1$ & $4.8$e$-03$ & $6.8$e$-1$ & $17$\\
  \midrule
  \multirow{3}{*}{$10^{-6}$} &  $31^{3}$ &        --- & $1.9$e$-1$ & $8.5$e$-07$ & $7.5$e$-6$ &  $3$ & $1.4$e$-1$ & $9.8$e$-07$ & $9.8$e$-6$ &  $3$\\
                             &  $63^{3}$ &        --- & $2.1$e$+0$ & $3.9$e$-06$ & $5.8$e$-5$ &  $3$ & $1.4$e$+0$ & $2.5$e$-06$ & $4.4$e$-5$ &  $3$\\
                             & $127^{3}$ &        --- & $2.6$e$+1$ & $2.3$e$-05$ & $1.3$e$-3$ &  $4$ & $1.9$e$+1$ & $9.1$e$-06$ & $2.6$e$-4$ &  $4$\\
  \midrule
  \multirow{3}{*}{$10^{-9}$} &  $31^{3}$ &        --- & $1.5$e$-1$ & $6.1$e$-10$ & $3.4$e$-9$ &  $2$ & $1.6$e$-1$ & $7.4$e$-10$ & $4.0$e$-9$ &  $2$\\
                             &  $63^{3}$ &        --- & $2.0$e$+0$ & $4.0$e$-09$ & $3.5$e$-8$ &  $2$ & $1.7$e$+0$ & $2.0$e$-09$ & $2.1$e$-8$ &  $2$\\
                             & $127^{3}$ &        --- & $3.3$e$+1$ & $1.7$e$-08$ & $4.6$e$-7$ &  $2$ & $2.7$e$+1$ & $6.4$e$-09$ & $1.5$e$-7$ &  $2$\\
  \midrule
  \multirow{2}{*}{---}       &  $31^{3}$ & $2.0$e$-1$ &        --- &         --- &        --- &  --- &        --- &        --- &        --- &  ---\\
                             &  $63^{3}$ & $3.3$e$+0$ &        --- &         --- &        --- &  --- &        --- &        --- &        --- &  ---\\
  \bottomrule
 \end{tabular}
\end{table}

\begin{figure}
 \includegraphics{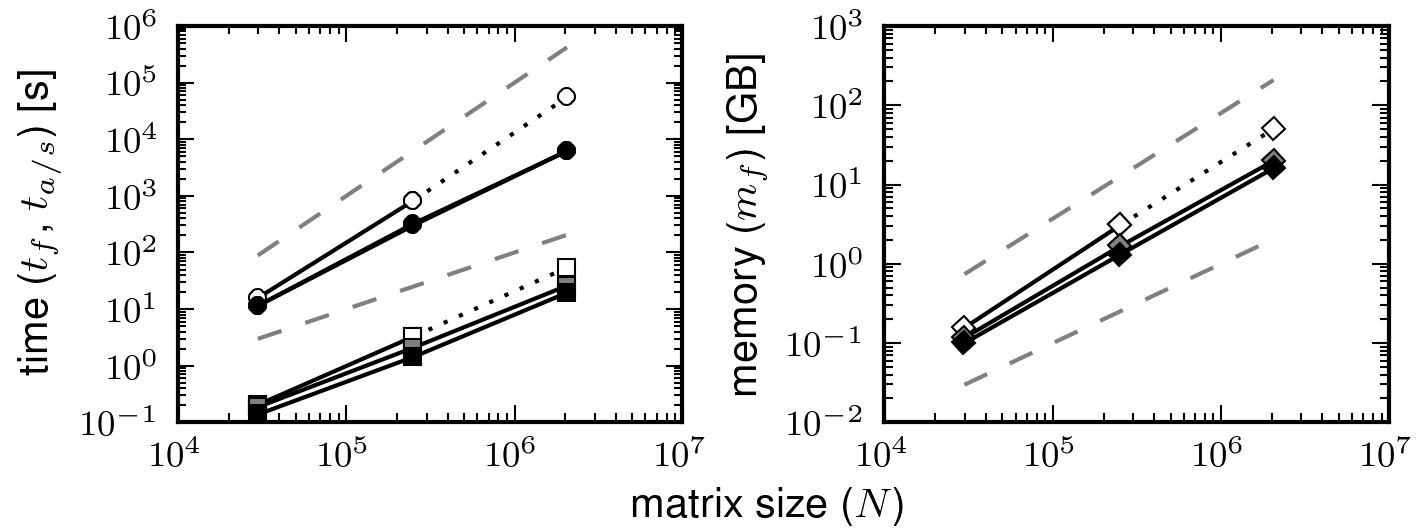}
 \caption{Scaling results for Example 4, comparing \alg{mf3} (white) with \alg{hifde3} (gray) and \alg{hifde3x} (black) at precision $\epsilon = 10^{-6}$. Included also are reference scalings of $O(N)$ and $O(N^{2})$ (left), and $O(N)$ and $O(N^{4/3})$ (right); all other notation as in Figure \ref{fig:fd_square1}. The lines for \alg{hifde3} and \alg{hifde3x} lie nearly on top of each other; for $t_{f}$ (top left), they overlap almost exactly.}
 \label{fig:fd_cube1}
\end{figure}

It is immediate that $t_{f} = O(N^{2})$ and $t_{a/s} = O(N^{4/3})$ for MF, which considerably degrades its performance for large $N$. Indeed, we were unable to run \alg{mf3} for $N = 127^{3}$ because of the excessive memory cost. In contrast, HIF-DE scales much better, with $|s_{L}|$ growing consistently with \eqref{eqn:inter-rank} for both variants. This provides strong evidence for Theorem \ref{thm:hifde}. However, the skeleton size is substantially larger than in 2D, and neither \alg{hifde3} nor \alg{hifde3x} quite achieve quasilinear complexity as predicted: the empiricial scaling for $t_{f}$ for both algorithms at, e.g., $\epsilon = 10^{-6}$ is approximately $O(N^{1.4})$. We believe this to be a consequence of the large interaction ranks, which make the asymptotic regime rather difficult to reach. In parallel with Example 1, $e_{a} = O(\epsilon)$ but $e_{s}$ is somewhat larger due to ill-conditioning. We found $F^{-1}$ to be a very effective preconditioner throughout. There were no significant differences in either computation time or accuracy between \alg{hifde3} and \alg{hifde3x}, though the latter does provide some appreciable memory savings.

\subsubsection*{Example 5}
Now consider the 3D analogue of Example 2, i.e., Example 4 but with $a(x)$ a quantized high-contrast random field as previously defined, extended to 3D in the natural way. Data for \alg{mf3}, \alg{hifde3}, and \alg{hifde3x} at $\epsilon = 10^{-6}$ and $10^{-9}$ are given in Tables \ref{tab:fd_cube2-f} and \ref{tab:fd_cube2-a}.

\begin{table}
 \caption{Factorization results for Example 5.}
 \label{tab:fd_cube2-f}
 \scriptsize
 \begin{tabular}{cr|rcc|rcc|rcc}
  \toprule
  & & \multicolumn{3}{|c|}{\alg{mf3}} & \multicolumn{3}{|c|}{\alg{hifde3}} & \multicolumn{3}{|c}{\alg{hifde3x}}\\
  $\epsilon$ & \multicolumn{1}{c|}{$N$} & \multicolumn{1}{|c}{$|s_{L}|$} & $t_{f}$ & $m_{f}$ & \multicolumn{1}{|c}{$|s_{L}|$} & $t_{f}$ & $m_{f}$ & \multicolumn{1}{|c}{$|s_{L}|$} & $t_{f}$ & $m_{f}$\\
  \midrule
  \multirow{3}{*}{$10^{-6}$} &  $31^{3}$ & \tabdash &        --- &        --- &   $1441$ & $1.1$e$+1$ & $1.1$e$-1$ &    $948$ & $1.1$e$+1$ & $1.0$e$-1$\\
                             &  $63^{3}$ & \tabdash &        --- &        --- &   $3271$ & $2.5$e$+2$ & $1.5$e$+0$ &   $2337$ & $2.8$e$+2$ & $1.2$e$+0$\\
                             & $127^{3}$ & \tabdash &        --- &        --- &   $6679$ & $4.9$e$+3$ & $1.7$e$+1$ &   $3294$ & $4.9$e$+3$ & $1.4$e$+1$\\
  \midrule
  \multirow{3}{*}{$10^{-9}$} &  $31^{3}$ & \tabdash &        --- &        --- &   $1893$ & $1.2$e$+1$ & $1.2$e$-1$ &   $1423$ & $1.3$e$+1$ & $1.1$e$-1$\\
                             &  $63^{3}$ & \tabdash &        --- &        --- &   $4755$ & $3.6$e$+2$ & $1.8$e$+0$ &   $3924$ & $4.0$e$+2$ & $1.4$e$+0$\\
                             & $127^{3}$ & \tabdash &        --- &        --- &  $10913$ & $9.4$e$+3$ & $2.4$e$+1$ &   $7011$ & $9.9$e$+3$ & $1.9$e$+1$\\
  \midrule
  \multirow{2}{*}{---}       &  $31^{3}$ &   $2791$ & $1.5$e$+1$ & $1.6$e$-1$ & \tabdash &        --- &        --- & \tabdash &        --- &        ---\\
                             &  $63^{3}$ &  $11719$ & $8.4$e$+2$ & $3.0$e$+0$ & \tabdash &        --- &        --- & \tabdash &        --- &        ---\\
  \bottomrule
 \end{tabular}
\end{table}

\begin{table}
 \caption{Matrix application results for Example 5.}
 \label{tab:fd_cube2-a}
 \scriptsize
 \begin{tabular}{cr|c|cccr|cccr}
  \toprule
  & & \multicolumn{1}{|c|}{\alg{mf3}} & \multicolumn{4}{|c|}{\alg{hifde3}} & \multicolumn{4}{|c}{\alg{hifde3x}}\\
  $\epsilon$ & \multicolumn{1}{c|}{$N$} & $t_{a/s}$ & $t_{a/s}$ & $e_{a}$ & $e_{s}$ & \multicolumn{1}{c|}{$n_{i}$} & $t_{a/s}$ & $e_{a}$ & $e_{s}$ & \multicolumn{1}{c}{$n_{i}$}\\
  \midrule
  \multirow{3}{*}{$10^{-6}$} &  $31^{3}$ &        --- & $1.8$e$-1$ & $5.1$e$-07$ & $6.1$e$-3$ &  $6$ & $1.6$e$-1$ & $6.4$e$-07$ & $1.1$e$-2$ &  $5$\\
                             &  $63^{3}$ &        --- & $2.0$e$+0$ & $2.1$e$-06$ & $6.4$e$-2$ &  $7$ & $1.6$e$+0$ & $1.5$e$-06$ & $5.8$e$-2$ & $12$\\
                             & $127^{3}$ &        --- & $2.2$e$+1$ & $8.8$e$-06$ & $3.4$e$-1$ & $16$ & $1.6$e$+1$ & $6.0$e$-06$ & $3.3$e$-1$ & $16$\\
  \midrule
  \multirow{3}{*}{$10^{-9}$} &  $31^{3}$ &        --- & $1.9$e$-1$ & $3.3$e$-10$ & $1.5$e$-5$ &  $4$ & $1.4$e$-1$ & $3.8$e$-10$ & $1.3$e$-5$ &  $4$\\
                             &  $63^{3}$ &        --- & $2.2$e$+0$ & $1.6$e$-09$ & $1.7$e$-4$ &  $6$ & $1.8$e$+0$ & $1.9$e$-09$ & $1.7$e$-4$ &  $4$\\
                             & $127^{3}$ &        --- & $3.1$e$+1$ & $1.8$e$-08$ & $3.7$e$-3$ &  $8$ & $2.3$e$+1$ & $1.2$e$-08$ & $3.5$e$-3$ &  $8$\\
  \midrule
  \multirow{2}{*}{---}       &  $31^{3}$ & $2.0$e$-1$ &        --- &         --- &        --- &  --- &        --- &        --- &        --- &  ---\\
                             &  $63^{3}$ & $3.4$e$+0$ &        --- &         --- &        --- &  --- &        --- &        --- &        --- &  ---\\
  \bottomrule
 \end{tabular}
\end{table}

Again, the results are quite similar to those in Example 5, but with $e_{s}$ necessarily larger by a factor of about $\rho$ due to ill-conditioning. There are no evident difficulties arising from the high contrast ratio for either \alg{hifde3} or \alg{hifde3x}.

\subsubsection*{Example 6}
Finally, we consider the 3D analogue of Example 3, where now $k = 2 \pi \kappa$ is increased in proportion to $n = N^{1/3}$ at a fixed resolution of $8$ DOFs per wavelength. The matrix $A$ is once again indefinite, which we factored using \alg{mf3}, \alg{hifde3}, and \alg{hifde3x} with $\kappa = 4$, $8$, and $16$ at $\epsilon = 10^{-6}$ and $10^{-9}$. The data are summarized in Tables \ref{tab:fd_cube3-f} and \ref{tab:fd_cube3-a} with scaling results in Figure \ref{fig:fd_cube3}.

\begin{table}
 \caption{Factorization results for Example 6.}
 \label{tab:fd_cube3-f}
 \scriptsize
 \begin{tabular}{crr|rcc|rcc|rcc}
  \toprule
  & & & \multicolumn{3}{|c|}{\alg{mf3}} & \multicolumn{3}{|c|}{\alg{hifde3}} & \multicolumn{3}{|c}{\alg{hifde3x}}\\
  $\epsilon$ & \multicolumn{1}{c}{$N$} & \multicolumn{1}{c|}{$\kappa$} & \multicolumn{1}{|c}{$|s_{L}|$} & $t_{f}$ & $m_{f}$ & \multicolumn{1}{|c}{$|s_{L}|$} & $t_{f}$ & $m_{f}$ & \multicolumn{1}{|c}{$|s_{L}|$} & $t_{f}$ & $m_{f}$\\
  \midrule
  \multirow{3}{*}{$10^{-6}$} &  $31^{3}$ &  $4$ & \tabdash &        --- &        --- &   $1702$ & $2.4$e$+1$ & $1.8$e$-1$ &   $1215$ & $1.8$e$+1$ & $1.5$e$-1$\\
                             &  $63^{3}$ &  $8$ & \tabdash &        --- &        --- &   $4275$ & $8.2$e$+2$ & $2.5$e$+0$ &   $2934$ & $5.1$e$+2$ & $1.9$e$+0$\\
                             & $127^{3}$ & $16$ & \tabdash &        --- &        --- &  $10683$ & $1.9$e$+4$ & $3.0$e$+1$ &   $4071$ & $8.1$e$+3$ & $2.2$e$+1$\\
  \midrule
  \multirow{3}{*}{$10^{-9}$} &  $31^{3}$ &  $4$ & \tabdash &        --- &        --- &   $2144$ & $3.7$e$+1$ & $2.1$e$-1$ &   $1685$ & $2.5$e$+1$ & $1.7$e$-1$\\
                             &  $63^{3}$ &  $8$ & \tabdash &        --- &        --- &   $5614$ & $1.3$e$+3$ & $3.1$e$+0$ &   $4684$ & $9.3$e$+2$ & $2.3$e$+0$\\
                             & $127^{3}$ & $16$ & \tabdash &        --- &        --- &  $14088$ & $3.4$e$+4$ & $3.9$e$+1$ &   $7806$ & $1.7$e$+4$ & $2.9$e$+1$\\
  \midrule
  \multirow{2}{*}{---}       &  $31^{3}$ &  $4$ &   $2791$ & $6.4$e$+1$ & $2.5$e$-1$ & \tabdash &        --- &        --- & \tabdash &        --- &        ---\\
                             &  $63^{3}$ &  $8$ &  $11719$ & $5.5$e$+3$ & $4.9$e$+0$ & \tabdash &        --- &        --- & \tabdash &        --- &        ---\\
  \bottomrule
 \end{tabular}
\end{table}

\begin{table}
 \caption{Matrix application results for Example 6.}
 \label{tab:fd_cube3-a}
 \scriptsize
 \begin{tabular}{crr|c|cccr|cccr}
  \toprule
  & & & \multicolumn{1}{|c|}{\alg{mf3}} & \multicolumn{4}{|c|}{\alg{hifde3}} & \multicolumn{4}{|c}{\alg{hifde3x}}\\
  $\epsilon$ & \multicolumn{1}{c}{$N$} & \multicolumn{1}{c|}{$\kappa$} & $t_{a/s}$ & $t_{a/s}$ & $e_{a}$ & $e_{s}$ & \multicolumn{1}{c|}{$n_{i}$} & $t_{a/s}$ & $e_{a}$ & $e_{s}$ & \multicolumn{1}{c}{$n_{i}$}\\
  \midrule
  \multirow{3}{*}{$10^{-6}$} &  $31^{3}$ &  $4$ &        --- & $1.6$e$-1$ & $8.2$e$-07$ & $3.4$e$-6$ & $3$ & $1.7$e$-1$ & $6.5$e$-07$ & $1.3$e$-5$ & $3$\\
                             &  $63^{3}$ &  $8$ &        --- & $2.0$e$+0$ & $2.4$e$-06$ & $3.3$e$-5$ & $3$ & $1.8$e$+0$ & $2.0$e$-06$ & $4.5$e$-5$ & $3$\\
                             & $127^{3}$ & $16$ &        --- & $3.0$e$+1$ & $3.7$e$-06$ & $1.3$e$-3$ & $8$ & $2.1$e$+1$ & $9.7$e$-06$ & $4.7$e$-4$ & $4$\\
  \midrule
  \multirow{3}{*}{$10^{-9}$} &  $31^{3}$ &  $4$ &        --- & $1.9$e$-1$ & $5.0$e$-10$ & $2.4$e$-9$ & $2$ & $1.7$e$-1$ & $5.9$e$-10$ & $1.1$e$-8$ & $2$\\
                             &  $63^{3}$ &  $8$ &        --- & $2.4$e$+0$ & $1.7$e$-09$ & $2.1$e$-8$ & $2$ & $2.2$e$+0$ & $2.0$e$-09$ & $3.2$e$-8$ & $2$\\
                             & $127^{3}$ & $16$ &        --- & $3.3$e$+1$ & $3.3$e$-09$ & $1.2$e$-6$ & $6$ & $2.6$e$+1$ & $5.2$e$-09$ & $1.4$e$-7$ & $2$\\
  \midrule
  \multirow{2}{*}{---}       &  $31^{3}$ &  $4$ & $2.1$e$-1$ &        --- &         --- &        --- &  --- &        --- &        --- &        --- &  ---\\
                             &  $63^{3}$ &  $8$ & $2.6$e$+0$ &        --- &         --- &        --- &  --- &        --- &        --- &        --- &  ---\\
  \bottomrule
 \end{tabular}
\end{table}

\begin{figure}
 \includegraphics{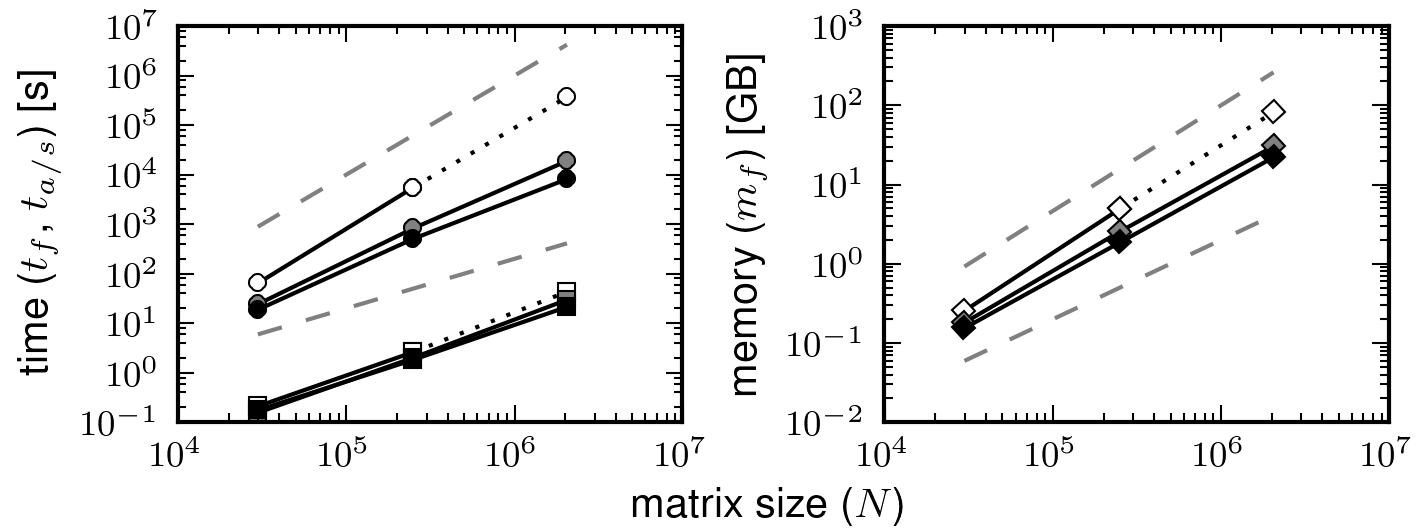}
 \caption{Scaling results for Example 6, comparing \alg{mf3} (white) with \alg{hifde3} (gray) and \alg{hifde3x} (black) at precision $\epsilon = 10^{-6}$; all other notation as in Figure \ref{fig:fd_cube1}.}
 \label{fig:fd_cube3}
\end{figure}

All algorithms behave essentially as expected, but the skeleton size is substantially larger for \alg{hifde3} than in the Laplace case (Example 4). The same increase, however, was not observed for \alg{hifde3x}. We take this to imply that the 1D nature of \alg{hifde3x} is less sensitive to the oscillatory character of the Helmholtz problem, at least at low frequency, though any definitive conclusion is difficult to draw. The empirical complexity at $\epsilon = 10^{-6}$ is now $t_{f} \simeq O(N^{1.5})$ for \alg{hifde3} and $O(N^{1.3})$ for \alg{hifde3x}. Both solvers remain quite favorable compared to \alg{mf3} and give very good preconditioners for GMRES.

\section{Generalizations and Conclusions}
\label{sec:conclusion}
In this paper, we have introduced HIF-DE for the efficient factorization of discretized elliptic partial differential operators in 2D and 3D. HIF-DE combines MF \cite{duff:1983:acm-trans-math-software,george:1973:siam-j-numer-anal,liu:1992:siam-rev} with recursive dimensional reduction via frontal skeletonization to construct an approximate generalized LU/LDL decomposition at estimated quasilinear cost. The latter enables significant compression over MF and is critical for improving the asymptotic complexity, while the former is essential for optimally exploiting sparsity and hence for achieving good practical performance. The resulting factorization allows the rapid application of the matrix inverse, which provides a fast direct solver or preconditioner, depending on the accuracy. Furthermore, although we have focused here only on symmetric matrices, our techniques generalize also to the unsymmetric case by defining analogous two-sided elimination operators $R_{p}$ and $S_{p}$ in \eqref{eqn:sparse-elim} as in \cite{ho:comm-pure-appl-math} and by compressing
\begin{align*}
 B_{c} =
 \begin{bmatrix}
  A_{c^{\nbr},c}\\
  A_{c,c^{\nbr}}
 \end{bmatrix}
\end{align*}
instead of just $A_{c^{\nbr},c}$.

While we have reported numerical data only for PDEs with Dirichlet boundary conditions, HIF-DE extends trivially to other types of boundary conditions as well. Preliminary tests with mixed Dirichlet-Neumann conditions reveal no discernable change in performance.

The skeletonization operator at the core of HIF-DE can be interpreted in several ways. For example, we can view it as an approximate local change of basis in order to gain sparsity. Unlike traditional approaches, however, this basis is determined optimally on the fly using the ID. Skeletonization can also be regarded as adaptive numerical upscaling or as implementing specialized restriction and prolongation operators in the context of multigrid methods \cite{brandt:1977:math-comp,hackbusch:1985:springer,xu:1992:siam-rev}.

Although we have presently only considered sparse matrices arising from PDEs, the same basic approach can also be applied to structured dense matrices such as those derived from the integral equation formulations of elliptic PDEs. This is described in detail as algorithm HIF-IE in the companion paper \cite{ho:comm-pure-appl-math}, which uses skeletonization for all compression steps and likewise has quasilinear complexity estimates in both 2D and 3D. In particular, HIF-DE can be viewed as a heavily specialized version of HIF-IE by embedding it into the framework of MF in order to exploit sparsity. The elimination operations in MF can also be seen as a trivial form of skeletonization acting on overlapping subdomains. Indeed, \cite{ho:comm-pure-appl-math} shows that recursive skeletonization \cite{gillman:2012:front-math-china,greengard:2009:acta-numer,ho:2012:siam-j-sci-comput,martinsson:2005:j-comput-phys}, a precursor of HIF-IE based on cell compression, is essentially equivalent to MF.

Some important directions for future research include:
\begin{itemize}
 \itemsep 1ex
 \item
  Obtaining analytical estimates of the interaction rank for SCIs, even for the simple case of the Laplacian. This would enable a much more precise understanding of the complexity of HIF-DE, which has yet to be rigorously established.
 \item
  Parallelizing HIF-DE, which, like MF, is organized according to a tree structure where each node at a given level can be processed independently of the rest. In particular, the frontal matrices are now much more compact, which should support better parallelization, and we anticipate that the overall scheme will have significant impact on practical scientific computing. This is currently in active development.
 \item
  Investigating alternative strategies for reducing skeleton sizes in 3D, which can still be quite large, especially at high precision.
 \item
  Understanding the extent to which our current techniques can be adapted to highly indefinite problems, some of which have a Helmholtz character and possess rank structures of a different type than that exploited here \cite{engquist:2009:commun-math-sci,engquist:2007:siam-j-sci-comput}. Such problems can be very challenging to solve iteratively and present a prime target area for future fast direct solvers.
\end{itemize}







\ack


We would like to thank Jack Poulson for helpful discussions, Lenya Ryzhik for providing computing resources, and the anonymous referees for their careful reading of the manuscript, which have improved the paper tremendously. K.L.H.\ was partially supported by the National Science Foundation under award DMS-1203554. L.Y.\ was partially supported by the National Science Foundation under award DMS-1328230 and the U.S.\ Department of Energy's Advanced Scientific Computing Research program under award DE-FC02-13ER26134/DE-SC0009409.


\frenchspacing
\bibliographystyle{plain}

\begin{thebibliography}{99}
 \bibitem{amestoy:siam-j-sci-comput}
  Amestroy, P. R.; Ashcraft, C.; Boiteau, O.; Buttari, A.; L'Excellent, J.-Y.; Weisbecker, C. Improving multifrontal methods by means of block low-rank representations. Submitted to {\it SIAM J. Sci. Comput.}
 \bibitem{aminfar:arxiv}
  Aminfar, A.; Ambikasaran, S.; Darve, E. A fast block low-rank dense solver with applications to finite-element matrices. Preprint, arXiv:1403.5337 [cs.NA].
 \bibitem{aurenhammer:1991:acm-comput-surv}
  Aurenhammer, F. Voronoi diagrams --- A survey of a fundamental geometric data structure. {\it ACM Comput. Surv.} {\bf 23} (1991), no. 3, 345--405.
 \bibitem{bebendorf:2005:math-comp}
  Bebendorf, M. Efficient inversion of the Galerkin matrix of general second-order elliptic operators with nonsmooth coefficients. {\it Math. Comp.} {\bf 74} (2005), no. 251, 1179--1199.
 \bibitem{bebendorf:2003:numer-math}
  Bebendorf, M.; Hackbusch, W. Existence of $\mathcal{H}$-matrix approximants to the inverse FE-matrix of elliptic operators with $L^{\infty}$-coefficients. {\it Numer. Math.} {\bf 95} (2003), 1--28.
 \bibitem{borm:2010:numer-math}
  B\"{o}rm, S. Approximation of solution operators of elliptic partial differential equations by $\mathcal{H}$- and $\mathcal{H}^{2}$-matrices. {\it Numer. Math.} {\bf 115} (2010), 165--193.
 \bibitem{brandt:1977:math-comp}
  Brandt, A. Multi-level adaptive solutions to boundary-value problems. {\it Math. Comp.} {\bf 31} (1977), no. 138, 333--390.
 \bibitem{chandrasekaran:2010:siam-j-matrix-anal-appl}
  Chandrasekaran, S.; Dewilde, P.; Gu, M.; Somasunderam, N. On the numerical rank of the off-diagonal blocks of Schur complements of discretized elliptic PDEs. {\it SIAM J. Matrix Anal. Appl.} {\bf 31} (2010), no. 5, 2261--2290.
 \bibitem{cheng:2005:siam-j-sci-comput}
  Cheng, H.; Gimbutas, G.; Martinsson, P. G.; Rokhlin, V. On the compression of low rank matrices. {\it SIAM J. Sci. Comput.} {\bf 26} (2005), no. 4, 1389--1404.
 \bibitem{davis:2006:siam}
  Davis, T. A. {\it Direct Methods for Sparse Linear Systems}. Society for Industrial and Applied Mathematics, Philadelphia, 2006.
 \bibitem{dixon:1983:siam-j-numer-anal}
  Dixon, J. D. Estimating extremal eigenvalues and condition numbers of matrices. {\it SIAM J. Numer. Anal.} {\bf 20} (1983), no. 4, 812--814.
 \bibitem{duff:1983:acm-trans-math-software}
  Duff, I. S.; Reid, J. K. The multifrontal solution of indefinite sparse symmetric linear equations. {\it ACM Trans. Math. Software} {\bf 9} (1983), no. 3, 302--325.
 \bibitem{engquist:2009:commun-math-sci}
  Engquist, B.; Ying, L. A fast directional algorithm for high frequency acoustic scattering in two dimensions. {\it Commun. Math. Sci.} {\bf 7} (2009), no. 2, 327--345.
 \bibitem{engquist:2007:siam-j-sci-comput}
  Engquist, B.; Ying, L. Fast directional multilevel algorithms for oscillatory kernels. {\it SIAM J. Sci. Comput.} {\bf 29} (2007), no. 4, 1710--1737.
 \bibitem{george:1973:siam-j-numer-anal}
  George, A. Nested dissection of a regular finite element mesh. {\it SIAM J. Numer. Anal.} {\bf 10} (1973), no. 2, 345--363.
 \bibitem{gillman:2014:siam-j-sci-comput}
  Gillman, A.; Martinsson, P. G. A direct solver with $O(N)$ complexity for variable coefficient elliptic PDEs discretized via a high-order composite spectral collocation method. {\it SIAM J. Sci. Comput.} {\bf 36} (2014), no. 4, A2023--A2046.
 \bibitem{gillman:2014:adv-comput-math}
  Gillman, A.; Martinsson, P.-G. An $O(N)$ algorithm for constructing the solution operator to 2D elliptic boundary value problems in the absence of body loads. {\it Adv. Comput. Math.} {\bf 40} (2014), 773--796.
 \bibitem{gillman:2012:front-math-china}
  Gillman, A.; Young, P. M.; Martinsson, P.-G. A direct solver with $O(N)$ complexity for integral equations on one-dimensional domains. {\it Front. Math. China} {\bf 7} (2012), no. 2, 217--247.
 \bibitem{golub:1996:johns-hopkins-univ}
  Golub, G. H.; van Loan, C. F. {\it Matrix Computations}, 3rd ed. Johns Hopkins University Press, Baltimore, 1996.
 \bibitem{grasedyck:2009:numer-math}
  Grasedyck, L.; Kriemann, R.; Le Borne, S. Domain decomposition based $\mathcal{H}$-LU preconditioning. {\it Numer. Math.} {\bf 112} (2009), 565--600.
 \bibitem{greengard:2009:acta-numer}
  Greengard, L.; Gueyffier, D.; Martinsson, P.-G.; Rokhlin, V. Fast direct solvers for integral equations in complex three-dimensional domains. {\it Acta Numer.} {\bf 18} (2009), 243--275.
 \bibitem{greengard:1987:j-comput-phys}
  Greengard, L.; Rokhlin, V. A fast algorithm for particle simulations. {\it J. Comput. Phys.} {\bf 73} (1987), 325--348.
 \bibitem{greengard:1997:acta-numer}
  Greengard, L.; Rokhlin, V. A new version of the Fast Multipole Method for the Laplace equation in three dimensions. {\it Acta Numer.} {\bf 6} (1997), 229--269.
 \bibitem{hackbusch:1999:computing}
  Hackbusch, W. A sparse matrix arithmetic based on $\mathcal{H}$-matrices. Part I: Introduction to $\mathcal{H}$-matrices. {\it Computing} {\bf 62} (1999), 89--108.
 \bibitem{hackbusch:1985:springer}
  Hackbusch, W. {\it Multi-Grid Methods and Applications}. Springer, Berlin, 1985.
 \bibitem{hackbusch:2002:computing}
  Hackbusch, W.; B\"{o}rm, S. Data-sparse approximation by adaptive $\mathcal{H}^{2}$-matrices. {\it Computing} {\bf 69} (2002), 1--35.
 \bibitem{hackbusch:2000:computing}
  Hackbusch, W.; Khoromskij, B. N. A sparse $\mathcal{H}$-matrix arithmetic. Part II: Application to multi-dimensional problems. {\it Computing} {\bf 64} (2000), 21--47.
 \bibitem{halko:2011:siam-rev}
  Halko, N.; Martinsson, P. G.; Tropp, J. A. Finding structure with randomness: Probabilistic algorithms for constructing approximate matrix decompositions. {\it SIAM Rev.} {\bf 53} (2011), no. 2, 217--288.
 \bibitem{hestenes:1952:j-res-nat-bur-stand}
  Hestenes, M. R.; Stiefel, E. Method of conjugate gradients for solving linear systems. {\it J. Res. Nat. Bur. Stand.} {\bf 49} (1952), no. 6, 409--436.
 \bibitem{ho:2012:siam-j-sci-comput}
  Ho, K. L.; Greengard, L. A fast direct solver for structured linear systems by recursive skeletonization. {\it SIAM J. Sci. Comput.} {\bf 34} (2012), no. 5, A2507--A2532.
 \bibitem{ho:comm-pure-appl-math}
  Ho, K. L.; Ying, L. Hierarchical interpolative factorization for elliptic operators: integral equations. Submitted to {\it Comm. Pure Appl. Math.}
 \bibitem{kuczynski:1992:siam-j-matrix-anal-appl}
  Kuczy\'{n}ski, J.; Wo\'{z}niakowski, H. Estimating the largest eigenvalue by the power and Lanczos algorithms with a random start. {\it SIAM J. Matrix Anal. Appl.} {\bf 13} (1992), no. 4, 1094--1122.
 \bibitem{liu:1992:siam-rev}
  Liu, J. W. H. The multifrontal method for sparse matrix solution: theory and practice. {\it SIAM Rev.} {\bf 34} (1992), no. 1, 82--109.
 \bibitem{martinsson:2009:j-sci-comput}
  Martinsson, P.-G. A fast direct solver for a class of elliptic partial differential equations. {\it J. Sci. Comput.} {\bf 38} (2009), 316--330.
 \bibitem{martinsson:2005:j-comput-phys}
  Martinsson, P. G.; Rokhlin, V. A fast direct solver for boundary integral equations in two dimensions. {\it J. Comput. Phys.} {\bf 205} (2005), 1--23.
 \bibitem{saad:2003:siam}
  Saad, Y. {\it Iterative Methods for Sparse Linear Systems}, 2nd ed. Society for Industrial and Applied Mathematics, Philadelphia, 2003.
 \bibitem{saad:1986:siam-j-sci-stat-comput}
  Saad, Y.; Schultz, M. H. GMRES: A generalized minimal residual algorithm for solving nonsymmetric linear systems. {\it SIAM J. Sci. Stat. Comput.} {\bf 7} (1986), no. 3, 856--869.
 \bibitem{samet:1984:acm-comput-surv}
  Samet, H. The quadtree and related hierarchical data structures. {\it ACM Comput. Surv.} {\bf 16} (1984), no. 2, 187--260.
 \bibitem{schmitz:2012:j-comput-phys}
  Schmitz, P. G.; Ying, L. A fast direct solver for elliptic problems on general meshes in 2D. {\it J. Comput. Phys.} {\bf 231} (2012), 1314--1338.
 \bibitem{schmitz:2014:j-comput-phys}
  Schmitz, P. G.; Ying, L. A fast nested dissection solver for Cartesian 3D elliptic problems using hierarchical matrices. {\it J. Comput. Phys.} {\bf 258} (2014), 227--245.
 \bibitem{van-der-vorst:1992:siam-j-sci-stat-comput}
  van der Vorst, H. A. Bi-CGSTAB: A fast and smoothly converging variant of Bi-CG for the solution of nonsymmetric linear systems. {\it SIAM J. Sci. Stat. Comput.} {\bf 13} (1992), no. 2, 631--644.
 \bibitem{xia:2013:siam-j-sci-comput}
  Xia, J. Efficient structured multifrontal factorization for general large sparse matrices. {\it SIAM J. Sci. Comput.} {\bf 35} (2013), no. 2, A832--A860.
 \bibitem{xia:2013:siam-j-matrix-anal-appl}
  Xia, J. Randomized sparse direct solvers. {\it SIAM J. Matrix Anal. Appl.} {\bf 34} (2013), no. 1, 197--227.
 \bibitem{xia:2010:numer-linear-algebra-appl}
  Xia, J.; Chandrasekaran, S.; Gu, M.; Li, X. S. Fast algorithms for hierarchically semiseparable matrices. {\it Numer. Linear Algebra Appl.} {\bf 17} (2010), 953--976.
 \bibitem{xia:2009:siam-j-matrix-anal-appl}
  Xia, J.; Chandrasekaran, S.; Gu, M.; Li, X. S. Superfast multifrontal method for large structured linear systems of equations. {\it SIAM J. Matrix Anal. Appl.} {\bf 31} (2009), no. 3, 1382--1411.
 \bibitem{xia:2012:siam-j-matrix-anal-appl}
  Xia, J.; Xi, Y.; Gu, M. A superfast structured solver for Toeplitz linear systems via randomized sampling. {\it SIAM J. Matrix Anal. Appl.} {\bf 33} (2012) no. 3, 837--858.
 \bibitem{xu:1992:siam-rev}
  Xu, J. Iterative methods by space decomposition and subspace correction. {\it SIAM Rev.} {\bf 34} (1992), no. 4, 581--613.
\end{thebibliography}

\end{document}